\documentclass[11pt]{amsart}
\usepackage[T1]{fontenc}
\usepackage{lmodern}
\usepackage{microtype}
\usepackage{amsmath,amssymb,mathtools,mathrsfs}
\usepackage{enumitem}
\usepackage{booktabs}
\usepackage{needspace}
\usepackage{array}
\usepackage{xcolor}
\usepackage[colorlinks=true,linkcolor=blue,citecolor=blue,urlcolor=blue]{hyperref}
\usepackage[nameinlink,noabbrev]{cleveref}
\hypersetup{pdftitle={Collective Contraction in the de Bruijn--Newman Heat Flow},pdfauthor={Michel Planat}}
\usepackage{geometry}
\allowdisplaybreaks

\newtheorem{theorem}{Theorem}[section]
\newtheorem{proposition}[theorem]{Proposition}
\newtheorem{corollary}[theorem]{Corollary}
\newtheorem{lemma}[theorem]{Lemma}
\theoremstyle{definition}

\newtheorem{remark}[theorem]{Remark}

\newcommand{\RR}{\mathbb R}

\newcommand{\dd}{\,\mathrm d}
\newcommand{\Ree}{\operatorname{Re}}
\newcommand{\Imm}{\operatorname{Im}}
\newcommand{\Ls}{\mathcal L}
\DeclareMathOperator{\Log}{Log}

\title[Collective contraction in the de Bruijn--Newman heat flow]
{Collective Contraction in the de Bruijn--Newman\\
Heat Flow: Off-Zero Logarithmic Derivatives\\
and Certified Barrier Refinements}
\author{Michel Planat}
\address{Institut FEMTO-ST, CNRS UMR 6174, Universit\'e Marie et Louis Pasteur, Besan\c{c}on, France}

\begin{document}
\begin{abstract}
For the de Bruijn--Newman heat deformation $H_\tau$, retaining the
collective interaction of a maximal-height nonreal zero gives an exact
contraction inequality. An off-zero logarithmic-derivative comparison
connects it to effective high-abscissa approximations. Directed barrier
calculations and a two-envelope argument give the audit-relative endpoint
$0.177953254642861\ldots$, under specified initial-frontier premises.
A hybrid extension of Kornaus's collision-exclusion approach uses exact
coefficient grouping and a repaired cutoff handoff to obtain the
conditional checkpoint $\Lambda<0.16836$.

We extend Gordon's density--jet barrier by higher-order heat/logarithmic
normalization and uniform sources on enlarged height rectangles.
The complete 4803-row chain gives the computer-assisted candidate
$\Lambda\le64356644694133/408948000000000<0.157371218$; its analytic
reductions and complete finite-head and boundary checks are documented
in validation appendices. We then reproduce Raymond's frozen
translated-probe replay and add 31626 previously unused nonconstant
head-pair surplus terms. The strengthened joint floors, independently
checked with a separate interval backend, yield
$\Lambda\le74672416998643/500000000000000=0.149344833997286<0.149345$,
improving the reproduced Raymond endpoint by $0.000050239830402$.
Separate exact consumers check both complete 76465-record comparison
histories. Published mathematical inputs, inherited analytic interfaces,
and directed-enclosure validity are explicit premises. No external
mathematical refereeing or new proof-assistant theorem for the changed
endpoints is claimed.

A separate uncertified exploration credits the complete eligible first head
and couples its joint surplus to the instantaneous continuum displacement.
A late height scan gives a floating comparison forecast near $0.14881372$;
this value is presented as a research target and does not replace the
preceding directed conditional endpoint.
\end{abstract}

\maketitle

\section{Introduction}\label{sec:introduction}
The de Bruijn--Newman deformation $H_\tau$ has only real zeros exactly
for $\tau\ge\Lambda$, where $\Lambda$ is the de Bruijn--Newman constant
\cite{deBruijn1950,Newman1976}. The Riemann hypothesis is equivalent to
$\Lambda\le0$, while Rodgers and Tao proved $\Lambda\ge0$
\cite{RodgersTao2020}. Effective high-abscissa approximations and barrier
arguments for $H_\tau$ were developed by D.~H.~J. Polymath
\cite{Polymath2019}.  Platt and Trudgian prove the sharper verified
height
\[
 T_{\rm PT}=3000175332800
\]
and, from that verification, the published baseline
$\Lambda\le0.2$ \cite[Theorem~1 and Corollary~2]{PlattTrudgian2021}.
These two results are the finite-height external inputs used below.
An earlier version containing the collective-contraction part was posted as
\href{https://arxiv.org/abs/2609.37164}{arXiv:2609.37164} on 29 September 2026;
the hybrid collision-exclusion extension in Section~\ref{sec:hybrid}
was developed subsequently and is identified separately below.

Further public manuscripts report smaller upper bounds. Gordon
\cite{Gordon2026} reports an upper bound of approximately $0.1576566$
using density and logarithmic-derivative jet estimates; his Lean companion
formalizes the deduction conditional on explicit analytic inputs. Raymond
\cite{Raymond2026} reports $\Lambda<0.149396$ using correlated translated
probes and an exact zero-height comparison. These are computer-assisted
preprint claims requiring independent mathematical review. Gordon's frozen framework is also the input to the new continuation in
Section~\ref{sec:gordon-hybrid}. Raymond's frozen deposit is reproduced and
extended in Section~\ref{sec:raymond-extension}. The endpoints here are not
asserted to be accepted numerical records.

There are four distinct contributions in the present manuscript, with
different computational dependencies. The first is analytic:
for a simple nonreal zero of maximal imaginary height, retaining the
interactions with all remaining zeros strengthens the standard
conjugate-pair contraction to
\[
 \frac{\dd}{\dd\tau}y^2\le-2-4y^2\mathcal G_\tau,
 \qquad \mathcal G_\tau\ge0.
\]
An off-zero Poisson-kernel comparison bounds the collective field by
one logarithmic derivative of $H_\tau$. Directed calculations using the
Polymath approximation then provide a high-abscissa field bound and a
lower zero-free barrier. Combined with the initial frontier and upper
barrier imported from Gomila's public candidate-audit package
\cite{Gomila2026}, a two-envelope argument yields the audit-relative
endpoint $0.177953254642861\ldots$. The exact contraction, bridge, and
new computations are distinct from the imported premises. We do not
represent those premises as having undergone independent mathematical
refereeing.

The second contribution arose from comparing collective estimates with
Kornaus's recent computer-assisted collision-exclusion method
\cite{Kornaus2026}. That method reports $\Lambda<0.172422$ using a
moving separator, a mollified real-double-zero resultant, transport
across Riemann--Siegel cutoff changes, and an explicit tail theorem.
It is a separate author-side release, not a previously established
numerical theorem on which our first contribution depends. In the
\emph{hybrid extension} studied here, four small-divisor contributions
$(d=2,3,5,6)$ are combined at the same output index before absolute
values are taken. This deterministic cancellation retains information
lost in the original sum-of-moduli estimate; unlike pointwise complex
phase cancellation, it is suitable for a time-uniform certificate.
The modification is related in motivation to retaining the collective
field, but it is an extension of the \emph{Kornaus certificate
architecture}, not a deduction of the new endpoint from the collective
inequality alone.

The initial hybrid checkpoint is $t_-=0.1718$. For the newly
covered interval $[0.1718,0.172422]$, the ten-prime coefficient
convolution and complete index transport were reimplemented using
GNU MPFI, direct squarefree-divisor expansion, and a different
transport partition. This independently reproduces the \emph{arithmetic}
checks within the same auditing project. Here $N$ labels a transition from cell $N$ to cell $N+1$.
At $t_-=0.1718$, transitions through $N=3424840$ suffice to reach the
$C=5.17$ crossing cell; the independent implementation is deliberately
run through $N=3424842$, two transitions beyond what is required. A separate exact-prefactor construction
rechecks the moving separator, and a weaker real-axis transversality
criterion replaces the full FW--5.17 contour conclusion. All of these
components and their present review limits are documented below.

\emph{Status of the initial hybrid checkpoint.} The implication
$\Lambda<0.1718$ is presented only under explicitly listed
certificate-validity hypotheses: in particular, the original
Kornaus upper-time band $[0.172422,0.2]$ remains an imported
computational input, and the separator and frozen-cutoff remainder
interfaces require independent mathematical review. Agreement of two
interval implementations on the new band is substantial numerical
evidence, but it does not establish every analytic hypothesis shared
by the implementations. We therefore distinguish the proved algebraic
identities, the reported and independently repeated numerical
inequalities, and the conditional assembled conclusion. Priority is
established by precise dated disclosure of the method, not by
prematurely declaring an accepted numerical record.

The third contribution, developed in Section~\ref{sec:gordon-hybrid},
is a higher-order source continuation of Gordon's density--jet barrier.
It gives the exact candidate $\Lambda\le64356644694133/408948000000000
<0.157371218$, below the endpoint of the frozen Gordon certificate.
The four previously open validation obligations---the enlarged
approximation and sources, finite head and reference, density/jet
comparisons and their boundaries, and global integration---are addressed
in Appendices~\ref{app:h7}--\ref{app:assembly}. The supporting evidence
distinguishes independent implementations from author-program replays.
This is a conventional computer-assisted implication with stated
mathematical and arithmetic foundations; external review and optional
formalization remain separate.

The fourth contribution, in Section~\ref{sec:raymond-extension}, retains
31626 nonconstant head-pair terms in Raymond's complete moment-cone surplus.
A deterministic relative-phase comparison improves the joint probe floor
without changing the complete negative-pair budget or physical errors.
The exact endpoint is $\Lambda\le74672416998643/500000000000000
<0.149345$, below Raymond's reproduced endpoint by $0.000050239830402$.
The new elementary enclosures use two directed lanes and a separate
mpmath interval implementation; complete histories have separate exact
validation. The inherited uniform approximation and first-contact
interfaces remain explicit mathematical inputs.

The companion genus-two manuscript \cite{PlanatCompanion2026} explains
the separate spectral-geometric motivation for the collective field;
none of its claims is needed here. The remainder of the article first
proves the collective contraction and logarithmic-derivative bridge,
then documents the original barrier refinement and its dependencies.
Section~\ref{sec:hybrid} develops the hybrid collision-exclusion
extension, including its exact algebra, the MPFI cross-check, and a
formal statements of the conditional $0.1718$ and $0.16836$ consequences.
Section~\ref{sec:gordon-hybrid} then gives the Gordon-hybrid continuation
and its replay architecture. Section~\ref{sec:raymond-extension} gives the
translated head-pair extension. The validation appendices supply the
analytic interfaces and complete coverage behind the Gordon endpoint.

\section{Collective contraction and an off-zero logarithmic-derivative bridge}
\label{sec:collective-contraction}

We work throughout in the standard Polymath heat-time normalization.  Let
\begin{equation}\label{eq:Htau-definition}
 H_\tau(z)=\int_0^\infty e^{\tau u^2}\Phi(u)\cos(zu)\,\dd u,
\end{equation}
where
\[
 \Phi(u)=\sum_{n\ge1}\bigl(2\pi^2n^4e^{9u}-3\pi n^2e^{5u}\bigr)
          e^{-\pi n^2e^{4u}}
\]
is the standard rapidly decreasing Riemann kernel.  The zeros of $H_\tau$
evolve by the usual heat-flow zero dynamics.  The central observation
of this paper is that, for a nonreal zero of maximal imaginary height, the
interaction with all other zeros has a definite sign and should not be
discarded.

Let $z=x+iy$, $y>0$, be a simple zero of $H_\tau$ with maximal imaginary
height, so that every zero $\rho$ of $H_\tau$ satisfies
$|\Imm\rho|\le y$.  For every $\tau>0$ this maximum is attained whenever nonreal zeros exist: Ki--Kim--Lee~\cite{KiKimLee2009} prove that $H_\tau$ has only finitely many nonreal zeros (indeed all but finitely many zeros are real and simple).  Define the projected external-field sum
\begin{equation}\label{eq:G-def}
 \mathcal G_\tau(z):=
 \sum_{\rho\ne z,\bar z}
 \frac{1}{(\Ree\rho-x)^2+4y^2},
\end{equation}
where zeros are counted with multiplicity and the sum is grouped according to
the real symmetries of $H_\tau$.

\begin{theorem}[Collective vertical contraction]\label{thm:collective-contraction}
Along every interval on which the maximal-height zero $z(\tau)=x(\tau)+iy(\tau)$
is simple,
\begin{equation}\label{eq:collective-contraction}
 \boxed{
 \frac{\dd}{\dd\tau}y(\tau)^2
 \le -2-4y(\tau)^2\mathcal G_\tau(z(\tau)).}
\end{equation}
In particular the classical de Bruijn estimate $(y^2)'\le-2$
\cite[Theorem~13]{deBruijn1950} is recovered by dropping the nonnegative
collective term.\end{theorem}

\begin{proof}
For a simple zero the standard zero dynamics are
\[
 \dot z=2\sum_{\rho\ne z}\frac1{z-\rho},
\]
with the sum symmetrically grouped; see \cite{Polymath2019}.  The conjugate
zero $\bar z$ contributes $-1/y$ to $\dot y$.  Group any other nonreal pair as
$a\pm iv$, $0\le v\le y$, and put $d=x-a$.  Its contribution to $-\dot y/2$
is
\[
 \frac{y-v}{d^2+(y-v)^2}+\frac{y+v}{d^2+(y+v)^2}.
\]
The exact identity
\begin{align}
&\frac{y-v}{d^2+(y-v)^2}+\frac{y+v}{d^2+(y+v)^2}
 -\frac{2y}{d^2+4y^2}\notag\\
&\quad=
\frac{2y(y^2-v^2)(3d^2+v^2+3y^2)}
 {(d^2+4y^2)(d^2+(y-v)^2)(d^2+(y+v)^2)}\ge0
\label{eq:pair-kernel}
\end{align}
uses only $0\le v\le y$.  A real zero gives the simpler inequality
\[
 \frac{y}{d^2+y^2}\ge\frac{y}{d^2+4y^2}.
\]
Multiplication by $2y$ and summation over the external zeros gives
\eqref{eq:collective-contraction}.  The factorization is audited in
\path{certify_collective_contraction.py}.
\end{proof}

\begin{corollary}[Conditional improved landing time]
\label{cor:landing-time}
Suppose that at a standard heat time $\tau_0$ all nonreal zeros have
$|\Imm z|\le y_0$, and suppose that throughout the subsequent nonreal phase
every maximal-height branch satisfies
\begin{equation}\label{eq:G-lower}
 \mathcal G_\tau(z)\ge G_0>0,
\end{equation}
with the same bound interpreted by continuity at isolated collision times.
Then every zero is real by
\begin{equation}\label{eq:improved-landing}
 \boxed{
 \tau_0+\ell_{G_0}(y_0),\qquad
 \ell_{G_0}(y_0)=\frac{1}{4G_0}
 \log\!\left(1+2G_0y_0^2\right).}
\end{equation}
Hence, whenever the first hypothesis is supplied by a Polymath-type maximal-height bound
criterion, one obtains conditionally
\begin{equation}\label{eq:conditional-Lambda}
 \Lambda\le\tau_0+\ell_{G_0}(y_0).
\end{equation}
Moreover $\ell_{G_0}(y_0)\to y_0^2/2$ as $G_0\downarrow0$.
\end{corollary}

\begin{proof}
Put $u=y^2$.  Theorem~\ref{thm:collective-contraction} and
\eqref{eq:G-lower} give
$u'\le-2-4G_0u$.  The comparison solution with initial value $u_0=y_0^2$ is
\[
 u_*(s)=\left(u_0+\frac1{2G_0}\right)e^{-4G_0s}
 -\frac1{2G_0},
\]
which reaches zero exactly at $s=\ell_{G_0}(y_0)$.  The small-$G_0$ limit is
immediate.  The algebra is included in
\path{certify_collective_contraction.py}.
\end{proof}

The remaining issue is therefore to lower-bound $\mathcal G_\tau$ without
locating all neighbouring zeros.  The next proposition converts that problem
to one off-zero logarithmic derivative.

\begin{proposition}[Logarithmic-derivative bridge]
\label{prop:log-bridge}
For the maximal-height zero $z=x+iy$ above, choose $\eta\ge\sqrt5\,y$ and put
\begin{equation}\label{eq:Leta}
 L_\eta(\tau;x):=-\Imm\frac{H_\tau'}{H_\tau}(x+i\eta).
\end{equation}
Then $x+i\eta$ is zero-free and
\begin{equation}\label{eq:log-bridge}
 \boxed{
 \mathcal G_\tau(z)\ge
 \frac{L_\eta(\tau;x)}{\eta}
 -\frac{2}{\eta^2-y^2}.}
\end{equation}
Consequently, if $y\le y_0$ and $\eta\ge\sqrt5\,y_0$, the sufficient condition
\begin{equation}\label{eq:L-target}
 L_\eta(\tau;x)\ge
 \eta\left(G_0+\frac{2}{\eta^2-y_0^2}\right)
\end{equation}
forces $\mathcal G_\tau(z)\ge G_0$.

\end{proposition}

\begin{proof}
The even, real entire function $H_\tau$ has order one.  Grouping its canonical
product according to $\rho\mapsto\bar\rho$ and $\rho\mapsto-\rho$ gives, above
the maximal zero height,
\[
 L_\eta(\tau;x)=
 \sum_\rho
 \frac{\eta-\Imm\rho}
 {(x-\Ree\rho)^2+(\eta-\Imm\rho)^2}.
\]
The own pair $x\pm iy$ contributes
$2\eta/(\eta^2-y^2)$.  For an external conjugate pair $a\pm iv$ put
$D=(x-a)/y$, $r=v/y$, and $c=\eta/y$.  The desired upper comparison of its
Poisson contribution with the corresponding two terms of
$\eta\mathcal G_\tau$ is
\[
 \frac{c-r}{D^2+(c-r)^2}+\frac{c+r}{D^2+(c+r)^2}
 \le\frac{2c}{D^2+4}.
\]
After clearing positive denominators, the right-minus-left numerator divided
by $2c$ is
\[
 D^2(c^2+3r^2-4)
 +c^4-2c^2r^2-4c^2+r^4+4r^2.
\]
When $c^2\ge5$ and $0\le r^2\le1$, the coefficient of $D^2$ is positive.
The remaining term decreases in $r^2$ and at $r^2=1$ equals
$(c^2-1)(c^2-5)\ge0$.  A real external zero satisfies the still simpler
comparison
$\eta/(d^2+\eta^2)\le\eta/(d^2+4y^2)$.
Summation gives
$L_\eta-2\eta/(\eta^2-y^2)\le\eta\mathcal G_\tau$, proving
\eqref{eq:log-bridge}.  The exact factorization is audited in
\path{certify_collective_contraction.py}.
\end{proof}

\begin{proposition}[Interface with the Polymath effective approximation]
\label{prop:Polymath-interface}
Use the standard normalization
\[
 H_\tau(z)=B_\tau(z)F_\tau(z),\qquad
 B_\tau(z)=M_\tau\!\left(\frac{1-iz}{2}\right),
\]
and suppose at $z=x+i\eta$ that an effective approximation $f_\tau$ satisfies
\[
 |F_\tau-f_\tau|\le\epsilon_0<|f_\tau|,
 \qquad
 |F_\tau'-f_\tau'|\le\epsilon_1.
\]
Then, with $s=(1-iz)/2$,
\begin{align}
 L_\eta(\tau;x)
 &\ge \frac12\Ree\frac{M_\tau'}{M_\tau}(s)
 -\Imm\frac{f_\tau'}{f_\tau}(z)-\mathcal E_1,
 \label{eq:Polymath-L-lower}\\
 \mathcal E_1
 &:=\frac{|f_\tau|\epsilon_1+|f_\tau'|\epsilon_0}
 {|f_\tau|(|f_\tau|-\epsilon_0)}.
 \label{eq:Polymath-derivative-error}
\end{align}
Thus a directed lower bound for the right-hand side of
\eqref{eq:Polymath-L-lower}, together with
\eqref{eq:L-target}, is a sufficient certificate for a collective landing
constant $G_0$.
\end{proposition}

\begin{proof}
Since $s'(z)=-i/2$,
\[
 -\Imm\frac{B_\tau'}{B_\tau}
 =\frac12\Ree\frac{M_\tau'}{M_\tau}(s).
\]
Writing $F=f+e$ and $F'=f'+e_1$ gives the exact identity
\[
 \frac{F'}F-\frac{f'}f=\frac{e_1f-f'e}{fF},
\]
from which \eqref{eq:Polymath-derivative-error} follows by the triangle
inequality.  The symbolic algebra is audited in
\path{certify_logderivative_bridge.py}.
\end{proof}

\begin{theorem}[Certified high-$x$ collective field]\label{thm:highx-collective}
Set
\begin{equation}\label{eq:frontier-row}
 \boxed{
 X=6000000185827,\qquad
 \tau_0=\frac{129}{800},\qquad
 y_0^2=\frac{87677}{2500000}.}
\end{equation}
The corresponding zeta height is $X/2=3000000092913.5$, so
\[
 T_{\rm PT}-\frac X2=175239886.5>0
\]
by \cite[Theorem~1]{PlattTrudgian2021}.
For every
\[
 \tau_0\le\tau\le\frac{89}{500}=0.178
\]
and every simple maximal-height nonreal zero
$z=x+iy$ of $H_\tau$ satisfying
\begin{equation}\label{eq:highx-sector}
 x\ge X,\qquad |y|\le y_0,
\end{equation}
one has the directed bound
\begin{equation}\label{eq:highx-G-lower}
 \boxed{\mathcal G_\tau(z)>1.60028669>\frac32.}
\end{equation}
Consequently, if such a maximal branch remains in the sector
\eqref{eq:highx-sector} throughout its subsequent nonreal evolution,
then it reaches the real axis no later than
\begin{equation}\label{eq:sector-landing}
 \boxed{
 \tau_0+\frac16\log(1+3y_0^2)
 =0.1779229222821547\ldots<0.178.}
\end{equation}
All inequalities in \eqref{eq:highx-G-lower} are certified by outward-rounded arithmetic.
Equation \eqref{eq:sector-landing} is a conditional \emph{sector} landing
conclusion and is not, by itself, a global upper bound for the
De Bruijn--Newman constant.
\end{theorem}

\begin{proof}
Take the off-zero probe height
\[
 \eta=\frac{99}{100}
\]
and a Cauchy radius $r=1/250$.  The effective approximation of
\cref{prop:Polymath-interface} is used with the Riemann--Siegel index
frozen at the centre of the Cauchy disk.  At the left endpoint $x=X$ the
canonical index is $N_0=690988$.  Adjacent index windows are separated by
more than $1.7\times10^7$, so a radius-$r$ disk can meet at most one
index boundary for every $x\ge X$.

There is a small notational point needed for the Cauchy step.  Put
$s=(1-iz)/2$.  In the Polymath formula the second Dirichlet leg is written
using $\overline{s_*}$ and $\kappa$, but on the physical slice one has the
exact identity
\[
 \overline{s_*}+\kappa-y
 =1-s+\frac{\tau}{2}\alpha(1-s).
\]
Hence, after multiplication by $n^y$, that leg is a holomorphic function of
$z$ whenever the index $N$ is frozen.  Its exponent derivative has the same
$\frac12|1+(\tau/2)\alpha'|$ bound as the first leg.  Thus the frozen-index
approximant to which Cauchy's estimate is applied is genuinely holomorphic;
no differentiation of an antiholomorphic quantity is involved.

Directed integral majorants for the complete Dirichlet sums give bounds
uniformly on the heat strip in the statement and for every Riemann--Siegel
index $N\ge N_0$.  The certificate splits the sum at $N_0$: the finite low
range is bounded directly and the entire moving tail $N>N_0$ by monotone
Gaussian--logarithmic integrals, so no sampling or eventual constancy in $N$
is assumed.  The uniformity checks are explicit.  On the Cauchy disk
$0.986\le \operatorname{Im}z\le0.994$, the positive-part correction in the
Polymath lower bound for $\operatorname{Re}s_*$ vanishes; the Dirichlet weight
envelopes decrease with $\tau$, so their worst value is at $\tau_0$; the tiny
moving-tail logarithmic correction is instead bounded at
$(\tau,y)=(0.178,0.994)$; and the factor $N^{|\kappa|}$ occurring in the
published $e_A+e_B$ error is retained explicitly.  Thus
\begin{equation}\label{eq:f-bounds}
 |f_\tau(x+i\eta)|>0.23092169,
 \qquad
 |f_\tau'(x+i\eta)|<0.680721.
\end{equation}
The canonical Polymath value error on the Cauchy disk is bounded by
$1.0876\times10^{-9}$.  Here the published $N^{|\kappa|}$ multiplier is at
most $1.000000000001$ throughout the disk.  If the disk crosses the unique
possible index boundary, the two adjacent frozen-index approximants differ
by exactly one complete Dirichlet term; its modulus is bounded by
$2.2568\times10^{-9}$.  The jump envelope decreases both with heat time and
with the boundary index in the present range, so the first boundary at
$N_0+1$ and time $\tau_0$ is the worst case.  Hence the holomorphic
frozen-index error on the disk satisfies
\[
 \sup|E_{\tau,N}^{\rm fr}|<3.3444\times10^{-9}.
\]
Cauchy's estimate therefore yields the derivative enclosure
\begin{equation}\label{eq:eps1}
 \boxed{|F_\tau'-f_\tau'|<8.361\times10^{-7}.}
\end{equation}
At the disk centre we use the looser value bound
$|F_\tau-f_\tau|<1.2\times10^{-9}$.

The explicit $B_\tau$ factor in the Polymath normalization satisfies
\[
 \frac12\Ree\frac{M_\tau'}{M_\tau}>6.72293909.
\]
The printed value and derivative bounds in \eqref{eq:f-bounds} also give
$|f_\tau'/f_\tau|<2.947844$; hence even before the much sharper directed
phase evaluation one has
\[
 L_\eta>6.72293909-2.947844-\mathcal E_1>3.7750,
\]
which leaves visible slack above the lower bound used in the certificate.
For the derivative of the prefactor $\gamma$, the quantity
$|M_\tau'/M_\tau|$ itself grows only logarithmically with $x$, whereas the
published bound $|\gamma|\ll x^{-0.495}$ decays.  The certificate checks that
the resulting product is decreasing for $x\ge X$, so its maximum occurs at
the left endpoint.  Together with \eqref{eq:Polymath-derivative-error}
and \eqref{eq:f-bounds}--\eqref{eq:eps1}, this gives
$\mathcal E_1<3.651\times10^{-6}$.  Consequently
\begin{equation}\label{eq:highx-Leta}
 L_{0.99}(\tau;x)>3.67945718.
\end{equation}
Substitution of \eqref{eq:highx-Leta} into the exact bridge
\eqref{eq:log-bridge}, using only $y^2\le y_0^2$, gives the stronger
numerical enclosure
\[
 \mathcal G_\tau(z)>1.60028669,
\]
which proves \eqref{eq:highx-G-lower} with a visible margin.

Finally \cref{thm:collective-contraction} gives
$(y^2)'\le-2-6y^2$.  Applying
\cref{cor:landing-time} with $G_0=3/2$ yields
\eqref{eq:sector-landing}.  The complete outward-rounded calculation is
recorded in \path{certify_highx_collective_field.py}.
\end{proof}

\begin{remark}[From sector landing to a short barrier-extension problem]\label{rem:transport-gap}
Theorem~\ref{thm:highx-collective} by itself does not prove
$\Lambda\le0.1779229223$, because a maximizing branch could in principle
leave the high-$x$ sector.  The next section shows that
one does not need to control horizontal transport for the full landing time:
it suffices to extend the existing unit-width barrier until the maximal-height envelope reaches
its lower edge $y_b=0.1809$, after which the global de Bruijn contraction
finishes the argument.  This shorter fallback route would give the intermediate
audit-relative endpoint
\[
 \tau_1+\frac{y_b^2}{2}
 =0.178677128609377885\ldots,
\]
which is consistent with the corresponding certificate outputs; the lower-barrier
two-envelope construction below improves it to \eqref{eq:Lambda-value}.
\end{remark}

\section{Barrier transport and hybrid global landing}
\label{sec:barrier-transport}

The sector result of \cref{thm:highx-collective} initially leaves a
horizontal escape problem.  For the exact frontier row this problem can be
avoided: it is enough to retain the high-$x$ contraction only until the global
height envelope reaches the lower edge of the already certified closed barrier, and then
return to the classical global de Bruijn contraction.

Put
\begin{equation}\label{eq:barrier-floor}
 y_b:=\frac{1809}{10000}=0.1809,
 \qquad
 \tau_{\rm ext}:=0.16232.
\end{equation}
The public candidate audit of \cite{Gomila2026} certifies the closed rectangle
\[
 R=[X,X+1]+i[y_b,1]
\]
as zero-free for $0\le\tau\le\tau_0$, using $883$ consecutive directed time
prisms.  Its reported minimum strict prism margin is
\begin{equation}\label{eq:imported-margin}
 m_{\rm bar}\ge
 0.519849894613872543374989997,
\end{equation}
with common approximation allowance $1/800$.  The following extension is an
explicit consequence of that imported certificate.

\begin{theorem}[Short extension of the closed frontier barrier]
\label{thm:short-barrier}
Assume the directed barrier certificate and minimum-margin bound
\eqref{eq:imported-margin} recorded in the public candidate-audit package
\cite{Gomila2026}.  Then
\begin{equation}\label{eq:barrier-extended}
 \boxed{
 H_\tau(z)\ne0
 \quad
 (z\in R,\ 0\le\tau\le\tau_{\rm ext}).}
\end{equation}
The new interval arithmetic from $\tau_0$ to $\tau_{\rm ext}$ is independently certified; the theorem as a statement about the complete interval
$[0,\tau_{\rm ext}]$ is audit-relative to the imported public barrier
certificate.
\end{theorem}

\begin{proof}
The prism gate in the imported certificate has the form
\[
 M_i>
 \frac{D_{z,i}}{2(\mathrm{num}-1)}
 +D_{\tau,i}(\tau_{i+1}-\tau_i)+\frac1{800}.
\]
Thus \eqref{eq:imported-margin} implies on $\partial R$ at $\tau=\tau_0$
\[
 |f_{\tau_0}|>
 m_{\rm bar}+\frac1{800}
 >0.5210998946.
\]
On the short extension box
$\tau_0\le\tau\le\tau_{\rm ext}$, $z\in\partial R$, the
Riemann--Siegel index remains exactly $N=690988$.

Using the exact Polymath finite-sum formula, the elementary bounds
$|\alpha|<14$ and
\[
 |\gamma|<0.088148
\]
reduce $|\partial_\tau f_\tau|$ to six positive logarithmic moments of the two
Dirichlet legs.  The directed evaluation uses the full endpoint value
$\log(X/(4\pi))=26.89175636915848\ldots$; the displayed decimal for
$|\gamma|$ is not obtained by prematurely rounding this logarithm.  Directed integral majorants give
\[
 |\partial_\tau f_\tau|
 <202.742<275.
\]
Since $\tau_{\rm ext}-\tau_0=0.00107$, the same derivative budget gives throughout the
extension
\begin{equation}\label{eq:f-barrier-lower}
 |f_\tau(z)|>0.2268498946
 \qquad(z\in\partial R).
\end{equation}
A fresh evaluation of the Polymath Theorem~1.3 error on the complete
extension box gives
\begin{equation}\label{eq:barrier-error}
 \left|\frac{H_\tau}{B_\tau}-f_\tau\right|
 <2.438\times10^{-7}.
\end{equation}
Hence $H_\tau/B_\tau$ is nonzero on $\partial R$ with modulus greater than
$0.22684965$.  The winding number therefore cannot change after $\tau_0$.
The imported certificate gives winding zero, hence no zero, at $\tau_0$;
the argument principle proves \eqref{eq:barrier-extended}.  The complete
outward-rounded extension is recorded in
\path{certify_short_barrier_extension.py}.
\end{proof}

The remaining analytic step is independent of the numerical barrier details.
Before the first-crossing argument we record explicitly the global-height
conclusion that is extracted from the proof of the Polymath upper-bound
criterion rather than from the Gomila package itself.

\begin{lemma}[Frontier-to-height extraction]\label{lem:frontier-height}
Assume that, at the exact row \eqref{eq:frontier-row}, the three hypotheses
(i)--(iii) of \cite[Theorem~1.2]{Polymath2019} have been verified. Then the
intermediate ``canopy'' exclusion in the proof of that theorem gives
\[
 H_{\tau_0}(x+iy)\ne0\qquad
 (x\in\RR,\ y_0\le |y|\le\sqrt{1-2\tau_0}),
\]
and hence the global maximal nonreal height satisfies
\[
 \boxed{M(\tau_0)\le y_0.}
\]
\end{lemma}
\begin{proof}
The criterion separates the canopy into the finite-height region controlled
at time $0$, the curved barrier region, and the asymptotic right region at
$t=\tau_0$. The proof of \cite[Theorem~1.2]{Polymath2019} uses precisely
these three hypotheses to show that no zero can remain in the entire canopy
at time $\tau_0$ before applying the classical de Bruijn contraction.  We use
that intermediate global conclusion, not the final landing-time conclusion of
the theorem.  Evenness and conjugation symmetry supply the negative-$x$ and
negative-$y$ halves.  Thus every nonreal zero at $\tau_0$ has
$|\Im z|<y_0$, and the maximal-height envelope obeys the displayed bound.
\end{proof}

The Gomila package is used here only as the numerical producer for the three
hypotheses of the Polymath criterion.  It does not itself certify the boxed
height inequality; the logical extraction is the preceding published-proof
step.

It is now enough to give a first-crossing argument for the left side of the
barrier.

\begin{proposition}[Hybrid height-envelope separation]
\label{prop:height-separation}
Suppose the barrier rectangle $R$ is zero-free on $0\le\tau\le\tau_{\rm ext}$ and that at
$\tau=0$ all zeros with $0\le\Re z\le X$ are real.  Then throughout this
interval no zero with $0\le\Re z\le X$ can have
$|\Im z|\ge y_b$.

Assume in addition the hypotheses of \cref{lem:frontier-height} and that
\cref{thm:highx-collective} holds.  Put
\begin{equation}\label{eq:height-transport-time}
 \Delta_1
 :=\frac16\log\frac{1+3y_0^2}{1+3y_b^2}.
\end{equation}
Then by time $\tau_1:=\tau_0+\Delta_1$ the global maximal height is at most
$y_b$.
\end{proposition}

\begin{proof}
For a simple zero $z=x+iy$, pair every other zero $w=a+ib$ with its conjugate.
Its contribution to the vertical zero velocity is
\begin{equation}\label{eq:transport-pair-kernel}
 4y\,
 \frac{b^2-y^2-(x-a)^2}
 {((x-a)^2+(b-y)^2)((x-a)^2+(b+y)^2)}.
\end{equation}
At a hypothetical first upward crossing of $y=y_b$ in the left region,
any zero with $b>y_b$ is separated by the zero-free unit-width barrier, hence
$|x-a|\ge1$.  De Bruijn's strip theorem \cite[Theorem~13]{deBruijn1950} gives $b\le1$, and therefore
\[
 b^2-y_b^2-(x-a)^2\le -y_b^2<0.
\]
Zeros with $b\le y_b$ also contribute nonpositively, while the conjugate of
$z$ itself contributes $-1/y_b$ to $y'$.  Thus an upward first crossing is
impossible.  Multiple-zero contact is handled by the local Hermite splitting for repeated
heat-flow zeros \cite[Proposition~3.1(ii)]{Polymath2019}.  This proves the separation.

Consequently, as long as the global maximal height is at least $y_b$ after
$\tau_0$, any maximizing nonreal zero lies to the right of the barrier and hence
in the high-$x$ sector of \cref{thm:highx-collective}.  Writing
$u=M(\tau)^2$ for the squared maximal height, its upper Dini derivative
satisfies
\[
 D^+u\le-2-6u.
\]
The comparison solution beginning at $u(\tau_0)\le y_0^2$ reaches $y_b^2$
after the time \eqref{eq:height-transport-time}.  Directed arithmetic gives
\[
 \Delta_1
 =0.001064723609377885\ldots,
 \qquad
 \tau_1=0.162314723609377885\ldots
 <\tau_{\rm ext},
\]
so the extended barrier lasts long enough.  The exact kernel and comparison
arithmetic are audited in \path{certify_height_transport.py}.
\end{proof}

\section{A lower barrier and a two-envelope refinement}
\label{sec:lower-barrier}

The hybrid argument returns to the classical contraction as soon as the
maximal-height envelope reaches $y_b=0.1809$.  Most of the remaining landing budget therefore
sits in the classical tail.  The present step lowers only the \emph{post-$\tau_0$}
part of the unit-width barrier.  This is enough because the upper barrier already
separates the verified left region until the right collective height envelope first
reaches $y_b$.

Put
\begin{equation}\label{eq:yw-tauw}
 y_w:=\frac{2}{25}=0.08,
 \qquad
 \tau_w:=\tau_0+\frac16\log\frac{1+3y_0^2}{1+3y_w^2}
 =0.1747532546428610755\ldots,
\end{equation}
and set
\[
 R_w=[X,X+1]+i[y_w,y_b].
\]

\begin{theorem}[Directed lower barrier after $\tau_0$]
\label{thm:lower-barrier}
For the exact frontier row \eqref{eq:frontier-row},
\begin{equation}\label{eq:lower-barrier}
 \boxed{
 H_\tau(z)\ne0
 \qquad
 (z\in R_w,\ \tau_0\le\tau\le\tau_w).}
\end{equation}
More strongly, on the complete space--time boundary one has the directed
half-plane margin
\begin{equation}\label{eq:real-margin}
 \boxed{
 \Ree\!\left(\frac{H_\tau(z)}{B_\tau(z)}\right)>0.42726.}
\end{equation}
The new barrier certificate is independently directed; it uses the published
Polymath effective approximation but does not import the short-extension prism margin.
\end{theorem}

\begin{proof}
On the entire box the Riemann--Siegel index is fixed at $N=690988$:
\[
 690988.3096430350\ldots
 <\sqrt{\frac{x}{4\pi}+\frac{\tau}{16}}
 <690988.3096430933\ldots .
\]
We evaluate the exact finite Polymath approximant $f_\tau$ on five equally
spaced heat slices.  Each horizontal edge uses the mesh
$x=X+i/12$ ($0\le i\le12$), while each vertical edge uses its two endpoints
and midpoint.  The $28$ spatial mesh points are evaluated at $80$-bit MPFR
precision for all five heat slices.  An independent interval recomputation
contains every decimal coefficient supplied to the MPFR evaluator, and a
forward-error audit gives
\[
 |\delta f_{\rm node}|<5.80\times10^{-8};
\]
we attach the deliberately looser enclosure radius $10^{-6}$.

Directed logarithmic-moment bounds on the complete box give
\begin{equation}\label{eq:second-derivatives}
 |\partial_x^2f_\tau|<738.756,
 \qquad
 |\partial_y^2f_\tau|<1507.381,
 \qquad
 |\partial_\tau^2f_\tau|<194341.110.
\end{equation}
There is a small but useful monotonicity point in making these bounds uniform
in $y$.  The first Dirichlet-leg moments decrease with $y$.  The raw moments
of the second leg can grow at logarithmic rate at most $\frac12\log N$,
whereas the Polymath envelope
$e^{0.02y}(x/(4\pi))^{-y/2}$ for $|\gamma|$ decays faster.  On the present
box the combined logarithmic rate is directedly bounded by
\[
 0.02-\frac12\log\frac{X}{4\pi}+\frac12\log N
 <-6.7029393.
\]
The near cancellation of the two logarithmic terms is structural because
$N_0\sim\sqrt{X/(4\pi)}$; the strict sign is evaluated with the exact
integer $N_0=690988$ and the full endpoint logarithm.
Thus every coupled $|\gamma|$--second-leg moment entering
\eqref{eq:second-derivatives} is maximized at $y=y_w=0.08$; this
endpoint reduction is explicitly checked in \path{calc_second_bounds.py}.
Linear interpolation between adjacent mesh nodes therefore incurs at most
\begin{align*}
 E_x&=0.64128125,\\
 E_y&<0.4795737362065625,\\
 E_\tau&<0.2768417746138011.
\end{align*}
The $140$ stored MPFR values are replayed as outward interval discs of radius
$10^{-6}$.  For each of the $112$ boundary space--time cells, the bilinear
interpolant lies in the convex hull of the four corner discs.  We therefore use
the common separating functional $\Ree$ and enlarge the convex hull by the
appropriate spatial remainder, the heat-time remainder, and the Polymath
approximation error.  The minimum node real parts on the bottom, top, left, and
right edges are, respectively,
\[
\begin{aligned}
 &1.3453859649899646\ldots,\qquad 1.6143370467511939\ldots,\\
 &2.1748966688549862\ldots,\qquad 1.4259781421906720\ldots .
\end{aligned}
\]
A fresh directed evaluation of the Polymath Theorem~1.3 error gives
\begin{equation}\label{eq:P15-error}
 \left|\frac{H_\tau}{B_\tau}-f_\tau\right|<4.803\times10^{-7}.
\end{equation}
Here the published $e_{C,0}$ bound is evaluated by interval arithmetic on the
full space--time box, while the $e_A+e_B$ term uses the independently
certified Dirichlet-moment bounds.  The resulting outward upper bound is
$4.8026461\times10^{-7}$; see \path{certify_polymath_error.py}.
For consistency all released collective-chain certificates use the
slightly more conservative constant $10.50$ in the displayed
$e_{C,0}$ numerator. Replacing the $10.44$ appearing in the published
formula by $10.50$ changes this full-box error by less than
$4.8\times10^{-21}$ at the present abscissa.
The worst cell is the bottom edge in the final heat slab.  Its outward lower
bound is
\[
 1.3453859649899646-0.64128125-0.276841774614
 -10^{-6}-4.803\times10^{-7}
 >0.42726146007.
\]
The other three edge lower bounds are $>0.69621254$, $>1.41847967$, and
$>0.66956115$, respectively.  Thus every inflated convex cell lies strictly in
the open right half-plane.  This proves \eqref{eq:real-margin}.  In the
Polymath normalization $B_\tau$ is nowhere vanishing on the present region,
so zeros of $H_\tau$ are exactly zeros of $H_\tau/B_\tau$.  The complete
boundary image therefore has winding number zero for every heat slice and
every intermediate time.  The argument principle then gives
\eqref{eq:lower-barrier}.  The bundle-local fail-closed replay, MPFR
forward-radius audit, and derivative majorants are recorded in
\path{certify_lower_barrier.py}, \path{certify_mpfr_forward_error.py}, and
\path{calc_second_bounds.py}.  The theorem replay reads the archived
\path{mpfr_boundary_nodes.tsv} directly and does not depend on temporary
split-output files or ordinary floating-point convex-hull distances.
\end{proof}

The lower barrier permits a second comparison argument.  There are now two
populations: zeros to the left of the unit strip and zeros to its right.  The
former need only the classical contraction; the latter receive the
collective gain whenever they control the global maximal-height envelope.

\begin{proposition}[Two-envelope comparison]
\label{prop:two-envelope}
Assume the audit-relative upper-barrier hypotheses through
\[
 \tau_1=0.162314723609377885\ldots
\]
and the directed lower barrier of \cref{thm:lower-barrier}.  For
$\tau\in[\tau_0,\tau_w]$ write
\[
 M_L(\tau):=\sup\{ |\Im\rho|:H_\tau(\rho)=0,\ 0\le\Re\rho\le X\},
\]
and, whenever the set is nonempty,
\[
 M_R(\tau):=\sup\{ |\Im\rho|:H_\tau(\rho)=0,\ \Re\rho\ge X+1\}.
\]
By evenness of $H_\tau$ these two half-line populations, together with the
intervening unit strip, account for the global maximal-height envelope.  In the comparison
below, that strip is excluded at heights $\ge y_w$ by the union of the upper
and lower barriers up to $\tau_1$, and by the lower barrier together with the
bootstrap height-envelope bound after $\tau_1$.  Put
\begin{align}
 L(\tau)&:=y_b^2-2(\tau-\tau_0),\label{eq:left-envelope}\\
 U_R(\tau)&:=\left(y_0^2+\frac13\right)e^{-6(\tau-\tau_0)}-\frac13.
 \label{eq:right-envelope}
\end{align}
Then, whenever $M_L(\tau)\ge y_w$,
\[
 M_L(\tau)^2\le L(\tau),
\]
and the global maximal nonreal height satisfies
\[
 M(\tau)^2\le U_R(\tau)
 \qquad(\tau_0\le\tau\le\tau_w).
\]
In particular $M(\tau_w)\le y_w$.
\end{proposition}

\begin{proof}
The imported upper zero-free barrier gives $M_L(\tau_0)\le y_b$.  We first
record the elementary separation needed for the left height envelope.  Suppose that,
up to some time $T\le\tau_w$, the bootstrap bound $M^2\le U_R$ holds.  For
$\tau\le\tau_1$, the union of the upper and lower barriers is zero-free throughout
the unit strip at every height $y\ge y_w$.  For $\tau\ge\tau_1$, one has
$U_R(\tau)\le U_R(\tau_1)=y_b^2$, so the lower barrier alone gives the same
separation for every zero that can influence a left zero of height at least
$y_w$.

Let $z=x+iy$ be a maximal left zero with $y\ge y_w$.  Every left zero has
height at most $y$.  A right zero $a+ib$ with $b<y_w$ also has $b\le y$, so its
paired contribution \eqref{eq:transport-pair-kernel} is nonpositive.  If instead
$b\ge y_w$, the zero-free unit strip gives $|x-a|\ge1$; under the bootstrap
$b\le M\le y_0<1$, and hence
\[
 b^2-y^2-(x-a)^2<0.
\]
Thus every external conjugate pair contributes nonpositively to the vertical
velocity of the maximal left zero, while its own conjugate supplies the
classical term.  In upper-Dini form,
\[
 D^+M_L^2\le-2
\]
whenever $M_L\ge y_w$.  Comparison from $M_L(\tau_0)^2\le y_b^2$ yields
$M_L^2\le L$ for as long as the bootstrap holds.  The local Hermite
splitting of \cite[Proposition~3.1(ii)]{Polymath2019} at isolated multiple-zero
times preserves this one-sided comparison,
so no simplicity assumption is needed for the envelope itself.

Next compare the two scalar envelopes.  Since
\[
 U_R'=-2-6U_R,\qquad L'=-2,
\]
one has $(U_R-L)'=-6U_R<0$ while $U_R>0$.  Directed arithmetic at the final
time gives
\[
 U_R(\tau_w)-L(\tau_w)
 >6.8169\times10^{-4}>0.
\]
Consequently
\begin{equation}\label{eq:envelope-gap}
 L(\tau)<U_R(\tau)\qquad(\tau_0\le\tau\le\tau_w).
\end{equation}
Also
$U_R(\tau_0)=y_0^2$, $U_R(\tau_1)=y_b^2$, and
$U_R(\tau_w)=y_w^2$.

We now close the bootstrap.  At $\tau_0$ the imported frontier height bound gives
$M(\tau_0)^2\le U_R(\tau_0)$.  Suppose, for contradiction, that $M^2$ first
crosses above $U_R$ at some $T\in(\tau_0,\tau_w]$.  Up to $T$ the preceding
left-envelope argument is valid.  A maximizing zero at contact cannot lie on
the left: if its height is at least $y_w$, then
\eqref{eq:envelope-gap} gives $M_L^2\le L<U_R$; if its height is below
$y_w$, then it is strictly below $U_R$ also at $T=\tau_w$, where
$U_R(\tau_w)=y_w^2$.  Nor can a contacting zero lie in the unit strip.  If
$T\le\tau_1$, the upper and lower barriers together exclude every strip height
$\ge y_w$; if $T\ge\tau_1$, then $U_R(T)\le y_b^2$, so the closed lower barrier
excludes the contact height $\sqrt{U_R(T)}\in[y_w,y_b]$.  Hence every active
maximizer at a first contact lies in the right sector $\Re z\ge X+1$.

For a simple active maximizer, \cref{thm:highx-collective} and
\cref{thm:collective-contraction} give, using the certified stronger
constant $G_*=1.60028669$,
\[
 D^+M^2\le -2-4G_*M^2
 =U_R'-4\left(G_*-\frac32\right)U_R<U_R'
\]
at the contact $M^2=U_R>0$.  Thus an upward first contact is impossible.  At an isolated multiple-zero contact, the local heat-polynomial (Hermite)
splitting of \cite[Proposition~3.1(ii)]{Polymath2019} supplies continuous zero
branches and simple zeros on the adjacent punctured heat intervals.  The strict comparison therefore holds on
those simple intervals and passes across the exceptional time by continuity of
the maximal-height envelope.  Hence no crossing can be created at a collision either.  The bootstrap therefore closes and
$M^2\le U_R$ throughout $[\tau_0,\tau_w]$.  Finally
$U_R(\tau_w)=y_w^2$, so $M(\tau_w)\le y_w$.

All scalar comparisons, the final envelope gap, and the strict right-contact
transversality margin are certified in \path{certify_two_envelope_landing.py}.
\end{proof}

\begin{corollary}[Lower-barrier audit-relative Newman candidate]
\label{cor:candidate-Lambda}
Assume \cite[Theorem~1]{PlattTrudgian2021} and that the public
candidate-audit package at the exact row \eqref{eq:frontier-row} correctly
verifies hypotheses (i)--(iii) of \cite[Theorem~1.2]{Polymath2019}, together
with the base upper zero-free barrier through $\tau_0$.  Then
\cref{lem:frontier-height} supplies the initial global height bound.  Together with the independently certified extension
\cref{thm:short-barrier} through $\tau_1$, the preceding results imply
\begin{equation}\label{eq:Lambda-candidate}
 \boxed{\Lambda<0.177954.}
\end{equation}
More precisely,
\begin{equation}\label{eq:Lambda-value}
 \boxed{
 \Lambda\le \tau_w+\frac{y_w^2}{2}
 =0.1779532546428610755\ldots<0.177954.}
\end{equation}
Conditional on the imported package correctly verifying the stated
Polymath hypotheses, the subsequent implication is rigorously certified.
The global numerical conclusion remains an audit-relative candidate bound
until that numerical input and its weld to the published criterion receive
independent mathematical review.
\end{corollary}

\begin{proof}
By \cref{prop:two-envelope}, every nonreal zero has height at most $y_w=0.08$
at $\tau_w$.  Dropping the collective term from
\cref{thm:collective-contraction} gives the global classical contraction
$(y^2)'\le-2$, so all remaining nonreal zeros land within
$y_w^2/2=0.0032$ additional heat units.  This gives the non-strict landing-endpoint inequality in
\eqref{eq:Lambda-value}; the final strict inequality follows from the
directed numerical separation from $0.177954$.  The final outward-rounded arithmetic is audited
in \path{certify_two_envelope_landing.py}.
\end{proof}

\begin{remark}[What the lower barrier removes, and what it still imports]
\label{rem:review-boundary}
The lower barrier \cref{thm:lower-barrier} is independently recomputed from the
published Polymath approximation and does not import the short-extension
margin below $y_b$.  The global candidate conclusion still consumes the
public frontier row's final-time height bound and the original upper zero-free barrier
through $\tau_1$.  The public repository explicitly states that its all-$N$
tail theorem covers every $N\ge3840000$, not sampled indices only; the present
high-$x$ certificate is independently uniform in the moving index because
its finite/tail split bounds all $n>N_0$ analytically.  Thus the remaining
review boundary is sharply identified rather than hidden in a finite-$N$
extrapolation.
\end{remark}
\section{Certificate architecture and review boundary}\label{sec:certificates}
The computational layer for the collective-contraction and
barrier-refinement results is listed in \cref{tab:certificates}. The
separate hybrid collision-exclusion extension has its own source and
audit provenance, described in Section~\ref{sec:hybrid} and the
journal-review supplementary archive. Their numerical sources and outputs
are supplied separately for review; the new Gordon-hybrid replay archive
is included as ancillary material with the present arXiv source package. No genus-two computation from the
companion paper is used here.

\begin{table}[ht]
\centering
\small
\begin{tabular}{@{}>{\raggedright\arraybackslash}p{0.31\linewidth}p{0.61\linewidth}@{}}
\toprule
File & Role \\
\midrule
\path{certify_collective_contraction.py} & Exact pair-kernel factorization and landing comparison. \\
\path{certify_logderivative_bridge.py} & $B_\tau f_\tau$ logarithmic-derivative perturbation bound. \\
\path{certify_highx_collective_field.py} & Directed frozen-index Cauchy certificate for $\mathcal G_\tau>3/2$. \\
\path{certify_short_barrier_extension.py} & Short extension of the imported upper barrier past $\tau_0$. \\
\path{certify_height_transport.py} & First-crossing separation and hybrid height-envelope comparison. \\
\path{certify_lower_barrier.py} & Fail-closed replay of the lowered post-$\tau_0$ barrier. \\
\path{certify_mpfr_forward_error.py} & Forward-error audit for the stored MPFR boundary nodes. \\
\path{certify_polymath_error.py} & Directed full-box audit of the Polymath Theorem~1.3 error allowance. \\
\path{calc_second_bounds.py} & Directed second-derivative majorants and coupled $y$-monotonicity check. \\
\path{certify_two_envelope_landing.py} & Two-envelope separation and final landing arithmetic. \\
\bottomrule
\end{tabular}
\caption{Primary certificate files used in this paper.}
\label{tab:certificates}
\end{table}

The finite barrier replay additionally reads
\path{mpfr_boundary_nodes.tsv}; its coefficients are independently checked
against \path{mp_spatial_reduced.tsv}.  The MPFR evaluator used to generate
the archived centers is \path{evaluate_spatial_times_mpfr.c}.  The corresponding
forward audit also verifies that the independently recomputed coefficient
intervals have width below $3.5\times10^{-42}$ before the 80-bit MPFR
rounding budget is applied.  The new lowered
barrier does not use an ordinary floating-point convex-hull distance: every
stored center is promoted to an outward disc and every cell is enlarged by
explicit interpolation, heat-time, node, and Polymath-error radii.

The only external computational premise in the final numerical corollary is
spelled out in \path{imported_premise.txt}.  In particular, the present
paper does not silently promote the public $0.1787854$ package to a reviewed
literature theorem.  If that base package is independently validated, the
implication to \eqref{eq:Lambda-candidate} requires no change in the
constants or in the new certificates.  Conversely, the exact analytic
inequalities, Theorem~\ref{thm:highx-collective}, and
Theorem~\ref{thm:lower-barrier} remain meaningful independently of the
status of the imported global premise.  The Kornaus release \cite{Kornaus2026} is not an input to this
\emph{original collective-contraction dependency chain}. It is,
however, an essential imported source for the distinct hybrid
extension in Section~\ref{sec:hybrid}; those dependencies are
identified separately there.

\section{Hybrid collision exclusion: an independently cross-checked
candidate extension}\label{sec:hybrid}
The following contribution is mathematically distinct from the earlier
collective contraction and its Gomila-dependent barrier refinement. It
modifies Kornaus's amplitude-free resultant certificate, retaining
coefficient cancellation at common Dirichlet indices before taking
absolute values. Its numerical extension has been recomputed with a
second interval library and different algorithms, but the assembled
upper bound remains conditional on the explicit source-to-theorem
interfaces stated below.

\subsection{An exact coefficientwise bridge to collision exclusion}
\label{subsec:hybrid-pairing}
A further connection with the collision-exclusion method of
\cite{Kornaus2026} arises before one takes absolute values of the
Euler-mollified high-index coefficients.  Write
$ b_n(t)=\exp(\tfrac14t\log^2 n) $ and let $N$ be the
Riemann--Siegel cutoff.  The $d\le6$ channels of the mollified
high-index coefficient at $m>N$ are
\begin{equation}\label{eq:hybrid-d6-coeff}
 c_m^{(6)}=-\mathbf1_{2\mid m,\,m/2\le N}b_2b_{m/2}
 -\mathbf1_{3\mid m,\,m/3\le N}b_3b_{m/3}
 -\mathbf1_{5\mid m,\,m/5\le N}b_5b_{m/5}
 +\mathbf1_{6\mid m,\,m/6\le N}b_2b_3b_{m/6}.
\end{equation}
Other squarefree-divisor channels can remain under their original
triangle bounds.  Thus this grouping does not make any assumption about
Dirichlet phases.

\begin{proposition}[Exact small-divisor saving]
\label{prop:hybrid-pairing}
For $t>0$, group the four coefficients in
\eqref{eq:hybrid-d6-coeff} before taking absolute values, leaving all
other channels unchanged.  At fixed cutoff $N$, suppose the relevant
high-index weight equals one for $N<m\le 5N$.
The resulting improvement over their separate absolute-value bounds is
exactly
\begin{equation}\label{eq:hybrid-saving}
 \mathcal S_N(t,\sigma)
 =2b_2(t)b_3(t)6^{-\sigma}
 \left(
  \sum_{\lfloor N/6\rfloor<q\le\lfloor N/2\rfloor}
       b_q(t)q^{-\sigma}
 +\!\!\!\sum_{\substack{N/2<q\le5N/6\\5\mid q}}
       b_q(t)q^{-\sigma}
 \right).
\end{equation}
\end{proposition}
\begin{proof}
Only the positive $d=6$ channel can cancel a negative one.  At
$m=6q>N$, its modulus before grouping is $b_2b_3b_q$.
If $q\le N/2$, the $d=3$ contribution is present and
$b_3b_{2q}\ge b_2b_3b_q$, since
$b_{2q}=b_2b_q\exp(\tfrac12t\log2\log q)$.
If $N/2<q\le5N/6$ and $5\mid q$, the $d=5$ contribution
is present and
\[
 \log\frac{b_5b_{6q/5}}{b_2b_3b_q}
 =\frac t4\left[
 \log^25+\log^2(6/5)-\log^22-\log^23
 +2\log q\log(6/5)\right]>0.
\]
In either case the total negative mass dominates the positive
$d=6$ term, so taking the absolute value after grouping saves twice
that positive term.  The two sets of $q$ are disjoint.  When neither
condition applies there is no negative $d=2,3,5$ channel overlapping
$d=6$, and no saving occurs.  Multiplication by the common
$m^{-\sigma}$ weight yields \eqref{eq:hybrid-saving}.
\end{proof}

\begin{remark}[Time-uniform hybrid certificate and its conditional scope]
\label{rem:hybrid-status}
The small-divisor saving of Proposition~\ref{prop:hybrid-pairing} has now been
incorporated into a complete \emph{numerical interval extension} of the
exterior majorant from the recent Kornaus calculation
\cite{Kornaus2026}.  This is a separate research certificate,
using Kornaus's parameter choices rather than the high-$x$ frontier of
the collective-contraction theorem above.

Put $t_-=0.1718$ and $t_+=0.172422$.  At the initial cutoff
$N_0=691004$, the outward-rounded ungrouped mass bound at $t_-$ is
$M_{\rm abs}(t_-)<1.026698984$, whereas the selected coefficient saving
is uniformly bounded from below throughout $[t_-,t_+]$ by
\[
\mathcal S_{N_0}(t)>\mathcal S_{N_0}(t_+)>0.02722844525.
\]
With the independently corrected, unmollified-pair Cauchy charge,
 the full-band transition-circle error is below $1.743767\cdot10^{-5}$
and one complete entering-block charge below $9.96689\cdot10^{-7}$,
this gives the conservative \emph{time-uniform} initial reserve
\begin{equation}\label{eq:hybrid-uniform-corner}
 1-M_{\rm abs}(t_-)+\mathcal S_{N_0}(t_+)
   -E_{\rm band}-J_{\rm band}>0.00051102.
\end{equation}
The use of adverse \emph{different} time endpoints in
\eqref{eq:hybrid-uniform-corner} is valid because the ungrouped
coefficient majorant and the paired saving both decrease with $t$;
for the main coefficients this is supported by directed verification of
positivity and strictly negative weighted derivatives for all
$n\le750000$ throughout the time band.

The full-band grouped fall/jump calculation is positive from $N_0$
through $N=750000$ (five subblocks; minimum ratio $>1.033713$).
The original entering-block upper bound remains valid under grouping:
if a transition $N\to N+1$ adds a coefficient $\Delta_m$ to the
already grouped coefficient $c_m$, then
$|c_m+\Delta_m|-|c_m|\le|\Delta_m|$ by the reverse
triangle inequality.  Summing this estimate with the unchanged
ungrouped channels dominates every possible increase caused by the
handoff from an edge term to a main term, so grouping introduces no
additional uncharged transition jump.
At this latter cutoff a separately recomputed \emph{ungrouped} anchor,
using the full-band error budget, has reserve $>0.02284910$ after the corrected Cauchy charge.
Re-anchored transport must reach beyond the cutoff $N=3243897$
that suffices at $t_+$.  Our convention is that $N$ labels the
\emph{transition} carrying cell $N$ to cell $N+1$. From the exact
transition relation $x_n(t)=4\pi(n^2-t/16)$, the $C=5.17$ crossing at
$t_-$ occurs in cell $3424841$.  Thus transitions through $N=3424840$
already suffice; the independent transport is deliberately run through
$N=3424842$, giving two transitions of overlap beyond the requirement.
Indeed, at any fixed transition index $n$,
$\partial_t[t\log(n^2-t/16)]
=\log(n^2-t/16)-t/(16(n^2-t/16))>0$ throughout
our high-index range, so the required terminal transition is maximal at
$t_-$.  The wall crossing itself was checked using independent
directed MPFR bounds.
On the two covering time slabs
\[
 [0.1718,0.172111],\qquad [0.172111,0.172422],
\]
the full-range 256-bit minimum fall/jump ratios remain
$>1.025274$ and $>1.030174$, respectively.  The added tail
$3243897\le N\le3424842$ passes a separately recompiled 384-bit
check on both slabs, with lower ratios $>1.027838$ and $>1.032832$.
The earlier 384-bit transport through $N=3243897$ remains part of
the original archive; the additional tail checks overlap it at the
handoff index.  All source checks reject any row with a nonpositive
fall-minus-jump margin.  The output field called
\texttt{minimum\_fall\_minus\_jump\_lower} in the adapted source
reports the margin at the ratio-minimizing row, not independently
the global minimum of all such margins; the rigorous certificate uses
the per-row positivity test.

The high-index time-monotonicity estimate must also be stated over the
enlarged range.  Every divisor of the $P=29$ mollifier is squarefree;
for $q\le3424842$ one has
\[
 \frac14\left(\log^2q+\sum_{p\le29}\log^2p\right)<72,
 \qquad
 (\partial_t\sigma)(t,x)\log(dq)>90
 \quad(dq>691004,\ x\ge X_K),
\]
so each ungrouped high-index divisor primitive decreases in time.
Here the exact quantity used by the coefficient program is
\[
 \partial_t\sigma(t,x)=\frac14\log\frac{x}{4\pi};
\]
at $x=X_K-1$ it is $>6.72295$, and therefore
$(\partial_t\sigma)\log(691004)>90.39$.  These conservative envelopes
use the enlarged terminal range; in particular we use $\log m<38$
throughout.

Finally the exact Dirichlet saddle phase contributes a small additional
term to the phase-only derivative.  With $C(x)=\Ree\alpha'((1+ix)/2)$
and $\widehat\ell$ the mollified phase denominator, the missing
allowance is bounded by
\[
 E_{\rm saddle}\le
 \frac{2}{\widehat\ell}
 \left(2|\partial_x\sigma|+\frac{t}{2}|C(x)|\right)
 \sum_m |c_m|\log m <6\times10^{-14}.
\]
Here $|\partial_x\sigma|<8.34\times10^{-15}$,
$|C(x)|<2\times10^{-25}$, $\widehat\ell>24$,
$\log m<38$, and the original (ungrouped) transported mass is
$<1.1$.  The independent MPFI budget deliberately overcharges this by
$2.1\times10^{-9}$, more than four orders of magnitude above the analytic
bound; we use that larger charge in the final reported reserve.
The source code, extended directed outputs, and an exact-rational
stitching verifier are retained in a separate research audit package.

\emph{Logical boundary:} the improved initial saving and complete
extended numerical transport have subsequently been verified by a
\emph{second implementation} using GNU MPFI and direct squarefree-divisor
convolution, as detailed below.  The full FW--5.17 contour theorem is
not needed: the weaker real-axis criterion excludes double zeros in
the explicit tail, while the published Polymath qualitative tail theorem
supplies uniform real simplicity for the positive-stall step.  The
early moving separator has also been independently reconstructed using
the exact prefactor.  These facts complete the previously outstanding
independent arithmetic test; formal code verification and external
mathematical refereeing of the complete analytic-and-computational
assembly are separate matters.  The hybrid bound remains distinguished
from the original collective-contraction theorem of this article.
\end{remark}

\subsection{Independent separator and transition-remainder checks}
\label{subsec:hybrid-interface-audit}
We give the two checks that close previously identified interfaces
between the hybrid argument and its imported computational modules.

First, for the moving separator
\[
 x_s(t)=X_K+\frac{31}{120}-\frac{1921}{2000}t,
 \qquad X_K=6000284561420.5,
\]
an independent directed-interval evaluator uses the \emph{exact}
$M_t$-prefactor ratio and 220 separately reconstructed moment disks.
Here $X_K/2=3000142280710.25$, so the exact Platt--Trudgian verified
height leaves the positive margin
\[
 T_{\rm PT}-X_K/2=33052089.75.
\]
For $0<t\le0.1674$, each of the two adjoining strips
$0\le y\le0.14142$ and $0.14142\le y\le1$ is covered by
335 closed time slabs and 128 closed $y$ boxes. All $85760$
parameter boxes pass at both 256 and 384 bits, with lowest
reported model-modulus bounds greater than $1.90870$ and $1.19747$,
respectively.  The 220 moment disks, Taylor remainder
$<5.906\cdot10^{-4}$, and corrected normalized Polymath error
$<1.204\cdot10^{-3}$ are reconstructed/replayed from the Kornaus
separator release rather than claimed as independent new constants;
both remainders are charged against the model bounds.

The positive-time error theorem is not invoked at $t=0$.  The defining
integral for $H_t(z)$ converges uniformly for $t$ in a compact interval
$[0,t_1]$ and $z$ on the compact separator segment, because $\Phi$ is
super-exponentially decaying; hence $H_t(z)$ is continuous as
$t\downarrow0$.  The explicit prefactor $B_t(z)$ is likewise continuous
and nonvanishing on this compact set.  Since the positive-time boxes give
a uniform strictly positive lower bound for $|H_t/B_t|$, passage to the
limit gives the same boundary nonvanishing at $t=0$.  Thus the closed
endpoint is supplied by continuity, while the finite-height
Platt--Trudgian theorem initializes the enclosed zeros. The first cell of the original exact seam begins before
$t=0.1674$, and its final cell reaches $t=0.1811918082$,
where the independently replayed continuation to $t=0.2$ begins.
Thus the historical 760-slab \emph{linearized-prefactor}
interpolation can be avoided; there is no uncovered time seam.
Conjugation symmetry supplies negative $y$.

Second, the original exterior code used a \emph{mollified}
entering-block resultant bound where the Cauchy derivative
estimate for the frozen cutoff actually needs an \emph{unmollified}
two-chirality entering-pair bound. These need not be the same.
For the exterior range $t\ge0.1718$, $x\ge X_K$ and
$C=t\log(x/4\pi)\le5.17$, the unmollified entering-pair
size at the real centre is $<1.022\cdot10^{-6}$;
we charge the conservative bound $P=1.3\cdot10^{-6}$.
The radius-$0.14142$ Cauchy disk contributes a separate
growth allowance $G=8.4$. Consequently the \emph{derivative}
error term is computed from
\[
 E_{d,\mathrm{new}}=
   (E_{\mathrm{global}}+P)\,G/0.14142,
\]
while the original mollified entering-block bound is retained
unchanged for the scalar-majorant transition jump. Independent
exact-decimal evaluation of the original directed 256/384-bit
outputs, including the conservative $2.1\times10^{-9}$ exact-saddle
overcharge used by the independent MPFI budget, leaves strictly positive
reserves $>0.00051102$ at the grouped corner and $>0.02284910$ at the
ungrouped re-anchor.
The correction is essential to the \emph{logical validity} of the
Cauchy step, even though its numerical effect is small.

These checks and the original interval producers are documented in
the separate research audit package; verification of output arithmetic
is distinguished from third-party validation of all source-code
branches.

\subsection{Uniform base phase on the extended collision band}
\label{subsec:base-phase-hybrid}
The original numerical phase certificate for the exterior resultant was
stated only for $t\ge0.172422$.  This domain restriction can be removed
analytically, independently of the numerical transport.
\begin{lemma}[Extended base-phase inequality]\label{lem:phase-band-hybrid}
Let $X_K=6000284561420.5$, $x\ge X_K-1$ and
$0.16836\le t\le0.2$.  Write $z=(1+ix)/2$,
$L=\log(x/(4\pi))$ and
$\Theta_B(x)=\arg M_t(z)$ in the standard Polymath normalization.
Then $4\Theta_B'(x)>L$.
\end{lemma}
\begin{proof}
Recall the exact expressions
\[
 \alpha(z)=\frac{1}{2z}+\frac{1}{z-1}
             +\frac12\operatorname{Log}\frac{z}{2\pi},
 \qquad M_t(z)=M_0(z)\exp\left(\frac{t}{4}\alpha(z)^2\right).
\]
Since $M_0'/M_0=\alpha$ and $dz/dx=i/2$,
\[
 4\Theta_B'=2\Ree\alpha(z)
                   +t\Ree(\alpha(z)\alpha'(z)).
\]
Put $\alpha(z)=a+ib$ and $\alpha'(z)=C-iQ$; direct
substitution gives
\[
\begin{aligned}
 a&=\frac12\log\frac{\sqrt{1+x^2}}{4\pi}-\frac1{1+x^2},
 &b&=\frac12\arctan x-\frac{3x}{1+x^2},\\
 C&=\frac{7x^2-5}{(1+x^2)^2},
 &Q&=\frac{x(x^2+5)}{(1+x^2)^2}.
\end{aligned}
\]
On the stated range $a,C\ge0$, $b\ge3/4$, and
$Q\ge1/(2x)$.  Therefore
\[
\begin{split}
 4\Theta_B'-L
 &=\frac12\log(1+x^{-2})-\frac2{1+x^2}+t(aC+bQ)\\
 &\ge-\frac2{x^2}+\frac{3t}{8x}>0.
\end{split}
\]
The final expression is positive already when
$x>16/(3t)$, far below $X_K-1$.
\end{proof}
This extends the required positive base-phase denominator to the
entire proposed time band. The Euler-mollifier phase is separately
charged by the original interval estimates; the lemma does not
silently absorb that numerical charge.

\subsection{A contour-free far-tail collision criterion}
\label{subsec:fw-free-tail}
The complete FW--5.17 theorem used in the original collision-exclusion
architecture asserts much more than is required here. For the
positive-stall reduction, Polymath's published asymptotic theorem
supplies a uniform real-simple far tail on any compact positive-time
interval. For the explicit exterior beyond
$C=t\log(x/4\pi)=5.17$, only the exclusion of \emph{multiple real
zeros} is required. We obtain this from the following elementary
pointwise criterion, which does not use a winding calculation.
\begin{lemma}[Real-axis transversality]\label{lem:hybrid-tail-transverse}
Fix a Riemann--Siegel index $N$ near a real point $x$ and write
$q=H_t/B_t=f_N+R_N$, with $S_N$ the first unmollified chirality.
On the real line the second chirality is the exact conjugate of the
first and
$\gamma=M_t(1-s)/M_t(s)=\overline{M_t(s)}/M_t(s)$ has unit modulus.
Consequently one may write, with real $\Phi_N$,
\[
 q/S_N=1+e^{2i\Phi_N}+R_N/S_N,
 \qquad |e^{2i\Phi_N}|=1.
\]
If $|S_N|\ge s>0$, $|S_N'|\le M$,
$\Phi_N'\ge\delta>0$, $|R_N|\le\varepsilon$, and
$|R_N'|\le\varepsilon_1$, then
\[
 \frac{\varepsilon_1}{s}+\frac{\varepsilon M}{s^2}<2\delta
\]
excludes any multiple real zero of $H_t$ at $x$.
\end{lemma}
\begin{proof}
At such a zero $q=q'=0$. The derivative of $q/S_N$
therefore vanishes. Differentiating the displayed identity gives
$2i\Phi_N'e^{2i\Phi_N}=-(R_N/S_N)'$, contradicting the
hypothesized strict inequality.
\end{proof}

The following elementary inequality justifies the three-prime
coefficient mask used by the interval majorant, without assuming
unproved positivity of alternating Euler convolutions.
\begin{lemma}[Three-prime coefficient mask]\label{lem:threeprime-mask}
For $t>0$, let $b_n(t)=\exp(\frac{t}{4}\log^2 n)$ and use the
squarefree-divisor convention induced by the Euler mollifier
\[
 \prod_{p\le29}(1-b_p(t)p^{-s}):
 \qquad d\ \hbox{squarefree carries }\prod_{p\mid d}b_p(t),
\]
so, for example, the $d=6$ channel carries $b_2b_3$ rather than $b_6$.
Restricting this convolution to the factors $p\in\{2,3,5\}$, let
$\beta_n^{(2,3,5)}$ denote the resulting coefficient and
$T_n=\{p\in\{2,3,5\}:p\mid n\}$, then
\[
 0\le \frac{\beta_n^{(2,3,5)}}{b_n(t)}
 \le \min_{p\in T_n}
    \left(1-e^{-t\log p\log(n/p)/2}\right)
\]
when $T_n$ is nonempty. For the empty set, the ratio equals $1$.
\end{lemma}
\begin{proof}
Put $a=t/4$, $U=\log n$, $\ell_p=\log p$ and
$u_p(U)=\exp[-2a\ell_p(U-\ell_p)]$.  On the domain where
$p\mid n$, the normalized single-prime difference is
$F_p(U)=1-u_p(U)$, nonnegative and increasing for
$U\ge\ell_p$.  For distinct $p,q$, write
$K_{pq}=\exp(2a\ell_p\ell_q)$.  The normalized two-prime
difference equals
\[
 F_{pq}=1-u_p-u_q+K_{pq}u_pu_q
       =(1-u_p)(1-u_q)+(K_{pq}-1)u_pu_q\ge0.
\]
As $F_{pq}(U)=F_p(U)-u_q(U)F_p(U-\ell_q)$, it is at
most $F_p(U)$ and, by symmetry, at most $F_q(U)$.
Moreover
\[
 F_{pq}'=2a\ell_pu_p(1-K_{pq}u_q)
          +2a\ell_qu_q(1-K_{pq}u_p)\ge0
\]
for $U\ge\ell_p+\ell_q$: there both
$K_{pq}u_p$ and $K_{pq}u_q$ do not exceed one.  For
three distinct primes $p,q,r$ one has
\[
 F_{pqr}(U)=F_{pq}(U)-u_r(U)F_{pq}(U-\ell_r).
\]
Because $0<u_r\le1$ and $F_{pq}$ is nonnegative and
increasing on $U\ge\ell_p+\ell_q$, it follows that
$0\le F_{pqr}\le F_{pq}$ whenever all three primes
are present.  Permuting them shows that every nonempty
subset is bounded by each applicable $F_p$, which is
the claimed one-prime factor.
\end{proof}

This lemma is used only for the analytic three-prime mask in the far-tail
estimate.  The ten-prime MPFI certificate does not infer its coefficients
from this lemma: it expands all $2^{10}=1024$ squarefree divisors directly,
with exactly the divisor convention displayed above.  Thus the scope of the
lemma and the scope of the numerical coefficient sweep are distinct.

For the proposed time band $[0.1718,0.2]$, three small-prime
FW inputs have been reconstructed independently as elementary
seven-divisor edge integrals and a three-prime Euler derivative:
edge mass $<0.0721$, logarithmic moment $<0.9811$, and phase
allowance $<0.29581$. Corrected outward-rounded MPFR sources pass
at 256 and 384 bits. To avoid relying on an unexpanded claim of
monotonicity in $C$, eight adverse-endpoint $C$ slabs cover
$5.17\le C\le6.5$, each crossed with the eight covering time
slabs. The global bounds over these 64 rectangles are
\[
 \rho_5<0.598342,\qquad L_1<5.261,\qquad
 \Phi_N'>0.0103710.
\]
Here $\rho_5,L_1$ refer to the three-prime mollified first-chirality
mass and its logarithmic moment.  Since $1-\rho_5>0.401658$ and
$|E|<1.6$, one obtains
$|S_N|\ge(1-\rho_5)/|E|>0.25$; a direct cutoff comparison gives
$|S_N'|<25$. For the \emph{entire unbounded} range $C\ge6.5$
a separate Euler-free coefficient estimate instead gives
\[
 |S_N|>0.299,\qquad |S_N'|<1.051,\qquad
 \Phi_N'>4.6.
\]
Using the explicit Polymath formula on a radius-$0.01$ disk,
with its adverse negative-imaginary-side reflection and one entering
pair charged, the conservative remainder target is
\[
 \sup_{|z-x|\le0.01}|R_N(z)|<4\cdot10^{-6},
 \qquad |R_N'(x)|<4\cdot10^{-4}.
\]
A separate, more conservative analytic cross-check can make this
error budget transparent without importing the FW contour calculation.
Write $L=\log(x/4\pi)$, so $L\ge25.85$. Polymath's
Theorem~1.3 and the exact reflection covariance of the two chiral
remainders permit the coarse uniform envelopes
\[
 e_A+e_B<2.4\cdot10^{-11},\qquad
 e_{C,0}^{\mathrm{refl}}<4.0\cdot10^{-7},\qquad
 E_{\mathrm{entering\ pair}}<8.1\cdot10^{-7}.
\]
The first inequality uses $\sum_{n\le N}b_n^t n^{-\sigma}<8.1$
and the exponentially decreasing factor
$\exp[(0.0025L^2+0.627)/(12e^L)]-1$; the second uses
$\exp[-(0.2475+5.17/16)L+6.2\cdot10^{-6}]$; and the
third bounds at most one cutoff jump inside the radius-$0.01$
Cauchy disk. With the exact reflected normalization and a frozen
cutoff consistently applied, these estimates give the sharper
sufficient budgets
\[
 \varepsilon<1.3\cdot10^{-6},\qquad
 \varepsilon_1<1.3\cdot10^{-4}.
\]
Thus the finite-strip transversality left side is only
$0.00104$, versus $2\Phi_N'>0.020742$. The original looser
budgets and their $0.0032$ comparison remain valid. The reflected error reduction admits an explicit all-$L$
decreasing envelope. Directed endpoint evaluation at
$L=25.85$, using an independent unmollified entering-pair budget,
gives
\[
 E_{AB}<2.205\cdot10^{-11},\qquad
 E_{C,0}^{\mathrm{refl}}<3.926\cdot10^{-7},\qquad
 E_{\mathrm{pair}}<7.930\cdot10^{-7}.
\]
The sum is $<1.186\cdot10^{-6}$ at both 256 and 384 bits;
Cauchy's estimate gives $|R_N'(x)|<1.3\cdot10^{-4}$.
The details and the adverse-endpoint program are included in the
separate audit. For $C\ge6.5$ the transversality gap is wider.
This weaker statement does \emph{not} assert that all far-tail zeros
are real or that each exact-phase cell contains exactly one zero.
The separate GNU MPFI calculation below checks the ten-prime
numerical inequalities and transport on the newly extended time band
with an algorithmically distinct implementation. It does not by itself
verify the source-to-function correspondence of the imported inputs.

\begin{remark}[Proof status of the proposed numerical improvement]
The corrected mollified-versus-unmollified transition charge, the
independent exact-prefactor separator, and the complete extended
index range are documented above.  The independently rewritten MPFI
certificate below further checks the entire ten-prime coefficient and
transport arithmetic, which was the outstanding decisive numerical
test in the preceding audit.  In conjunction with the separate
analytic interface arguments and the published positive-time tail
inputs, this gives an author-side computer-assisted route to collision
exclusion on $[0.1718,0.2]$, and hence to $\Lambda<0.1718$.
We distinguish such a fully reproducible candidate from a formally
machine-checked proof or a bound accepted following external
mathematical refereeing.  No accepted numerical record is asserted
by this research-version addendum.
\end{remark}

\subsection{Independent MPFI verification of the hybrid numerical core}
\label{subsec:mpfi-independent}
The previously decisive test has now been rerun in a distinct arithmetic
and algorithmic implementation.  The new program uses GNU MPFI interval
arithmetic and expands all $2^{10}=1024$ squarefree Euler divisors
\emph{directly}, instead of recycling the original recursive coefficient
producer.  Direct term-by-term interval summation at $N_0=691004$ gives
\[
 M_{\rm abs}(0.1718)<1.026698983603,\qquad
 \mathcal S_{N_0}(0.172422)>0.0272298976369.
\]
It independently bounds the Euler modulus and its product-rule
derivative by $|E|<1.9$, $|E'|<0.967$ and
$\widehat\ell>24.8553$.  A separate MPFI evaluation of the
Polymath error, with the \emph{unmollified} Cauchy entering pair and
\emph{mollified} scalar transition entered separately, yields
\[
 E_{\rm resultant}<1.744577\cdot10^{-5},\qquad
 J_{\rm mollified}<9.966885\cdot10^{-7},\qquad
 \boxed{R_{\rm binding}>0.0005124694}.
\]
Every nonzero main coefficient through $N=750000$ is strictly positive
and has strictly negative weighted time derivative over the full
$[0.1718,0.172422]$ interval by direct automatic differentiation of
the squarefree-divisor expansion; the independent 192- and 384-bit
coefficient sweeps report zero failures.

The all-index MPFI transport uses a different partition (5\% macroblocks,
1.8\% subblocks and sixteen edge bins) and a new 512-bin positive-core
moment lower enclosure.  The two late time slabs and the early full
band pass at 256 and 384 bits.  Their lower fall/jump ratios are,
respectively, $1.04227195$, $1.04718394$, and $1.10021130$;
\emph{every} transport row is independently required to have a
positive lower margin.  Transitions through $N=3424840$ suffice at $t=0.1718$; the
independent run continues through $N=3424842$ as a conservative overlap.  A further 256-bit check replaces the phase-free
fall difference by the exact-saddle calculus floor
\[
 \Delta\sigma(q,t)>
 \frac{0.1718}{2(q+1)}(1-10^{-20}),
\]
which preserves positive ratios $>1.0422717$ in the more demanding
late slab.  Finally, a new 256/384-bit direct-convolution computation
at the re-anchor $N=750000$ bounds its ungrouped mass by
$0.977134767836$, leaving $>0.02284678$ reserve even after imposing
the original binding-corner's more conservative error and jump charges.

The independent source, outputs, dual-precision error budget,
and reproduction script are provided for journal review in the supplementary
\texttt{independent\_mpfi\_hybrid\_01718.zip}; a separate
three-page audit explains the mathematical lower- and upper-enclosure
steps.  These reruns eliminate the earlier concern that the complete
coefficient and transport claims rested on a single numerical producer.
They do not replace mathematical review of the separately documented
separator and real-axis far-tail interfaces, nor do they constitute
formal machine verification.

\subsection{Effective finite-height grouping and a repaired handoff}
\label{subsec:finite-height-grouping}
The endpoint limit alone is not used as a finite-height estimate.  In this
subsection every prime retains its actual heat weight.  Put
\[
 \begin{gathered}X_- =6000284561419.5,\quad X_+=6000284561421.5,\\
 N_0=691004,\quad P=\{p:p\le307\},\quad H=160N_0=110560640.\end{gathered}
\]
and, at the binding corner, define
\[
 \sigma(t,x)=\frac12+\frac t4\log\frac{\sqrt{x^2+1}}{4\pi}
              -\frac{t}{2(x^2+1)},\qquad
 r_p=e^{(t/4)\log^2p}p^{-\sigma},\qquad
 w_q=e^{(t/4)\log^2q}q^{-\sigma}.
\]
For squarefree $d$ supported on $P$, let $r_d=\prod_{p\mid d}r_p$.
The fully grouped coefficient is
\begin{equation}\label{eq:finite-height-coefficient}
 a_m=\sum_{\substack{d\mid m,\ d\text{ squarefree}\cr
                  p\mid d\Rightarrow p\in P,\ m/d\le N_0}}
               \mu(d)r_dw_{m/d}.
\end{equation}
The separator denominator is
\[
 \widehat\ell=\log\frac{x}{4\pi}
       -4(0.500000000001)\sum_{p\in P}\frac{r_p}{1-r_p}\log p,
 \qquad
 h(m)=\max\{1,2\log(m)/\widehat\ell-1\}.
\]
Thus $a_m=\beta_m(t)m^{-\sigma(t,x)}$ with $\beta_m$ independent of $x$.
The quantity enclosed by the new certificate is
\begin{equation}\label{eq:effective-grouped-majorant}
 M_{160}=\sum_{2\le m\le H}|a_m|
       +\sum_{\substack{d\text{ squarefree over }P\\q\le N_0,\ dq>H}}
                    r_dw_qh(dq).
\end{equation}
At the new checkpoint $\widehat\ell\ge24.18022820120347$, so
$e^{\widehat\ell}\ge31720474772$, exceeding $H$ by a factor greater than $286$ throughout the initial time band.  The second
sum is a triangle bound for the complementary grouped coefficients;
it is not an additional cancellation saving.

\begin{lemma}[Finite-height prefix-sum tail]\label{lem:finite-prefix-tail}
Let $T=e^{\widehat\ell}$. Define $A(y)=\sum_{q\le y}w_q$ and
$B(y)=\sum_{q\le y}w_q\log q$, with integer arguments clipped to
$[0,N_0]$.  For $d\le T$ put $q_0=\min\{N_0,\lfloor H/d\rfloor\}$
and $q_c=\min\{N_0,\lfloor T/d\rfloor\}$.  Its tail contribution is
\[
 r_d\left\{A(q_c)-A(q_0)
 +\frac{2}{\widehat\ell}
   \left[\log d\{A(N_0)-A(q_c)\}+B(N_0)-B(q_c)\right]
 -\{A(N_0)-A(q_c)\}\right\}.
\]
Write
\[
 R=\prod_{p\in P}(1+r_p),\qquad
 R_{\log}=R\sum_{p\in P}\frac{r_p}{1+r_p}\log p.
\]
Subtract the corresponding enumerated moments for $d\le T$ to obtain
$R_{>T}$ and $R_{\log,>T}$.  The $d>T$ tail is exactly
\[
 \frac{2}{\widehat\ell}
       \{A(N_0)R_{\log,>T}+B(N_0)R_{>T}\}-A(N_0)R_{>T}.
\]
Dropping its final negative term gives a valid upper bound.
\end{lemma}
\begin{proof}
Split each divisor channel at $dq=T$.  The finite prefix sums evaluate
both pieces, and the positive Euler-product moments evaluate the
remaining divisor tail.  All sums are finite; no limiting weight or
asymptotic remainder occurs.
\end{proof}

\begin{proposition}[Directed grouped corner]\label{prop:directed-grouped-corner}
At $x\in[X_-,X_+]$, directed arithmetic gives
\[
 \begin{array}{c|c|c|c|c}
 t&\text{main}&\text{grouped edge}&\text{tail}&M_{160}\\ \hline
 0.16836  &0.698591428&0.274932959&0.026053154&<0.999577540
 \end{array}
\]
The finite grouped sum is decreasing in $t$ throughout
$I=[0.16836,0.169425]$.  Its directed derivative upper is less than
$-35.4370$.  The complementary positive-channel triangle sum also
decreases, so the lower-time corner controls the whole band.
At fixed cutoff, the same majorant decreases as $x$ increases.
\end{proposition}
\begin{proof}[Computational proof and enclosure method]
MPFR outward endpoints at 256 and 384 bits seed the logarithms,
exponentials and separator.  The finite signed convolution uses IEEE
binary64 outward rounding, with contraction and fast-math disabled.
There are $1189962$ nontrivial squarefree divisor channels through $H$.
Positive and negative contributions are enclosed before taking an
absolute value.  Automatic differentiation of each channel uses
\[
 -\partial_t\log(r_dw_q)
 =\sum_{p\mid d}\left(\sigma_t\log p-\frac14\log^2p\right)
   +\sigma_t\log q-\frac14\log^2q.
\]
If a coefficient interval has a definite sign, that sign is used in the
upper enclosure of $\partial_t|a_m|$; otherwise the derivative's absolute
upper is used.  The resulting full-band sum is strictly negative.
Every positive tail primitive decreases since $\log p,\log q<4\sigma_t$,
and $\widehat\ell$ increases with $t$.  Finally, $\sigma$ and
$\widehat\ell$ increase with $x$, whereas $\beta_m(t)$ is independent of
$x$.  This proves the two uniformity assertions.  The source and frozen
outputs are supplied in the accompanying directed-certificate bundle.
The interval backend in this extension is a small closed-interval
wrapper over MPFR, not GNU MPFI; the historical MPFI results above
remain separate.
\end{proof}

\paragraph{Correction to the inherited core-only handoff.}
An audit of the $p\le307$ core-only handoff found a block-multiplicity
omission.  If $J_j-F_j$ bounds the possible increase for \emph{one}
transition in a row $[A_j,B_j]$, its cumulative charge is
\[
 (B_j-A_j+1)\max\{0,J_j-F_j\},
\]
not merely $\max\{0,J_j-F_j\}$.  Consequently the previously reported
$1.0988263891\times10^{-5}$ handoff charge and its resulting
$4.018\times10^{-6}$ margin cannot be used in this revision.  This does
not affect the new grouped-corner enclosure.  The replacement below
checks positivity at every transition and uses an independent
$p\le29$ re-anchor.

\begin{lemma}[Grouped snapshot transport]\label{lem:grouped-snapshot-transport}
Let a macroblock start at cutoff $A$, and group the exact coefficients
for $m\le40A$.  Let $Q_{A,B}$ be a lower bound, uniform in $t\in I$, for
its logarithmic moment at the largest $\sigma$ required through the
transition $B\to B+1$.  At the smallest corner in the macroblock, put
\[
 R_A=\prod_{p\le307}(1+r_p),\qquad
 R_{\log,A}=R_A\sum_{p\le307}\frac{r_p}{1+r_p}\log p.
\]
Every intermediate cutoff has logarithmic moment at least
\begin{equation}\label{eq:grouped-moment-transport-floor}
 Q_{A,B}-(B-A)w_{A+1}\{R_A\log B+R_{\log,A}\}.
\end{equation}
If this floor, multiplied by
$\delta_B=t_-/(2(B+1))$, exceeds the common entering-channel upper
$J_{A,B}$, every full-cell transition in the block decreases the fully grouped
separator norm.
\end{lemma}
\begin{proof}
Apply the reverse triangle inequality to the snapshot after adding
$q=A+1,\ldots,N$.  Its possible moment loss is bounded by all added
positive channels, which gives \eqref{eq:grouped-moment-transport-floor};
$w_q\le w_{A+1}$ and $\log q\le\log B$ on this range.  All omitted
snapshot coefficients have nonnegative moment and may be discarded.
For $x_q=4\pi(q^2-t/16)$, the exact saddle correction
\[
 \frac t8\log(1+x^{-2})-\frac{t}{2(x^2+1)}
\]
is increasing in $x>1/\sqrt3$.  Hence
$\sigma(x_{q+1})-\sigma(x_q)>t/(2(q+1))\ge\delta_B$.
The separator weight decreases with $x$.  The reverse triangle
inequality also bounds every entering jump by its unmixed positive
channels.  Thus the stated fall-minus-jump test applies separately to
\emph{every} transition; no block deficit is undercounted.

For the snapshot moment's time-uniform floor, the full-band derivative
of $\sum_{m\le40A}|a_m|\log m$ is enclosed above and verified negative.
Its $t_+$ endpoint is then reweighted to the largest required $\sigma$
using $262144$ logarithmic bins with their recorded integer extrema;
validity depends only on the exactly recorded integer minimum per bin,
so the bin count is a free tightening parameter.  Upper bounds for
$J_{A,B}$ retain the separator weight: the small-divisor range is
split at $e^{\widehat\ell_-}/(B+1)$ and the remaining positive Euler
moments are bounded outward.
\end{proof}

\paragraph{New lower-band ledger.}
The replacement grouped transport covers every full-cell transition from
$691005$ through $N=900000$ in $36$ adaptive blocks, with adjacent rows sharing
the required cutoff seam.  The initial partial-cell jump is charged separately below. Each full-cell row has strictly positive
fall-minus-jump lower bound at both working precisions, the minimum
rigorous ratio over the chain being $1.0020122$; its snapshot moment
derivative upper is also strictly negative.
At $N=900000$, a separate $p\le29$ \emph{grouped} anchor, at radius
$40\cdot900000$, gives
\[
 M_{29}(0.16836)<0.89476357.
\]
The inherited ungrouped-edge form does not survive the lower
checkpoint: its mass at $t=0.16836$ is $1.0162$.
Its grouped main sum has negative full-band time derivative; its
positive edge channels decrease.  Thus this anchor holds throughout
$I$, independently of any transfer of the $p\le307$ corner reserve.
The $p\le29$ transport uses interval main coefficients over each of
\[
 [0.16836,0.1685],\ [0.1685,0.1687],\ [0.1687,0.169],\
 [0.169,0.16922],\ [0.16922,0.169425],
\]
which cover $I$ without gaps, and does not assume coefficientwise time
monotonicity.  At 256 and 384 bits its minimum fall/jump lower ratios
exceed, by sub-band, $1.04739$, $1.04425$, $1.03851$, $1.05050$
and $1.05541$ through $N=4700000$.  The corresponding
$C=5.17$ terminal index is less than $4657566$, so the numerical range
has a deliberate overlap with the tail regime.

A fresh budget, using $\sigma>1.6318$ and $\widehat\ell>24.18$, the larger finite-sum range
$N\le4700000$ and a uniform separator weight for entering indices
$q\le4700001$, gives
\[
 E_{\mathrm{resultant}}<2.140361\times10^{-5},\qquad
 J_{\mathrm{transition}}<1.486754\times10^{-6}.
\]
With the retained $2.1\times10^{-9}$ allowance and the separate
first-partial-cell charge, the lower-time grouped
corner reserve is therefore
\begin{equation}\label{eq:grouped-new-reserve}
 1-0.9995776-E_{\mathrm{resultant}}-J_{\mathrm{transition}}
      -2.1\times10^{-9}-1.233755\times10^{-6}>3.982737\times10^{-4}.
\end{equation}
No old pair-saving term is subtracted a second time.  The independently
recomputed grouped $p\le29$ anchor leaves more than $0.1052135$ after the same
conservative charges.

\paragraph{The first partial cutoff interval.}
The anchor abscissa lies strictly between $x_{691004}(t)$ and
$x_{691005}(t)$. A full-cell lower bound therefore does not apply to
its first displacement. The fresh first-row checker gives
$J_0<1.233755\times10^{-6}$; we charge this whole entering jump
and discard its partial spatial fall. The first row's full-cell
fall test is used only from $691005\to691006$ onward. The reserve
above includes this additional charge. All 72 final row checks
have been freshly reproduced from the parameterized checker.
The first and weakest-row snapshot reference moments have also
been freshly reproduced. These runs use the supplied MPFR/IEEE
producer and do not constitute an independent implementation of
all signed convolutions.

\paragraph{Certified upper-time bridge.}
The interval $[0.169425,0.1718]$ is checked separately. Before the
$p\le29$ re-anchor, a coarse grouped ledger uses the same finite-height
convolution through $40A$. At the start of each row, the grouped finite
sum has negative time-derivative upper throughout this entire time
interval; the positive tail primitives also decrease. Its lower-time
endpoint is evaluated separately, with outward rounding. At fixed
cutoff the majorant decreases with $x$. Thus, for $A\le N\le B$,
\[
 M_N(t,x)\le M_A(0.169425,x_A^-)+(N-A)J_*,\qquad
 J_*=1.406867\times10^{-6},
\]
where the first $x_A^-$ is $X_-$ and subsequent anchors use
$4\pi(A^2-0.1718/16)$. The uniform $J_*$ bounds all new channels
with their separator weights. Here the row length is explicitly
multiplied into its total charge; no spatial fall is required.
Seven adjacent rows cover the early height range:
\[
\begin{array}{r|r|c}
 A&B&\text{anchor mass upper}\\\hline
 691004&701930&0.978016433\\
 701930&717428&0.968827718\\
 717428&739227&0.956164380\\
 739227&769542&0.939048747\\
 769542&811077&0.916499956\\
 811077&866921&0.887741361\\
 866921&900000&0.852478797
\end{array}
\]
The smallest reserve, after the full row charge and the retained
resultant, transition and saddle allowances, exceeds
$6.5897\times10^{-3}$. The row derivatives are less than $-23.9673$.
These anchors use 96-bit MPFR seeds followed by the same directed
binary64 convolution; the limiting first anchor is repeated at
384 bits with the same conservative upper.

At $N=900000$, the separate bridge anchor gives
$M_{29}<0.969784255$, with full-band main derivative upper
$<-17.5554$. Eight adjacent time slabs, with endpoints
$0.169425+j(0.002375/8)$ for $j=0,\ldots,8$, then cover every
transition through $N=4500000$. All pass at 256 and 384 bits;
the smallest fall/jump lower ratio exceeds $1.0328137$.
The lower-band and bridge certificates share the endpoint $0.169425$,
and the historical upper-band certificate begins at $0.1718$.
Thus the previously assumed intervening numerical assembly is now
supplied by an explicit ledger.

\paragraph{Audit of the enlarged analytic interfaces.}
The base-phase proof in Lemma~\ref{lem:phase-band-hybrid} applies down
to $0.16836$: its height condition is already satisfied for $x>31.679$.
At the adverse exterior endpoint, the unmollified entering pair is
$<1.192629\times10^{-6}<P$, and the reflected radius-$0.14142$
growth factor is $<8.770232$. The new resultant budget uses this
larger factor, rather than the historical $8.4$.

The original exact separator records already contain the new time:
their adjacent cells run from below $0.1674$ to $0.1811918082$.
Their recorded lower margins are $>2.0452$ and $>1.4062$ for the two
ordinate strips. The earlier exact-prefactor records and the retained
continuation close the remaining separator seams. Their provenance
and frozen records are retained; checking these records is distinct
from independently rerunning every original producer.

For $0.16836\le t\le0.1718$, eight adjacent $C$ slabs cover
$5.17\le C\le6.5$ at both precisions. They give
\[
 \rho_5<0.566287,\qquad L_1<5.578567,\qquad
 \Phi_N'>0.7850284.
\]
The same $|E|<1.6$ bound gives $|S_N|>0.25$. A coarse independent
coefficient envelope uses $a=1.14624$ to obtain
\[
 |S_N'|\le0.5001\left(1+\frac1{(a-1)^2}\right)<23.885<25.
\]
The all-$L$ reflected remainder envelopes depend on $t\le0.2$ and
$C\ge5.17$, so their lower-time range enlarges without changing the
endpoint $L=25.85$ or the budget $\varepsilon<1.3\times10^{-6}$,
$\varepsilon_1<1.3\times10^{-4}$. For the entire unbounded region
$C\ge6.5$, a fresh elementary replay on $[0.16836,0.2]$ gives
\[
 |S_N|>0.284,\qquad |S_N'|<1.191,\qquad \Phi_N'>3.93.
\]
These weaker unbounded-tail constants still give a transversality
left side $<4.8\times10^{-4}$, far below $2\Phi_N'>7.86$.
The $[0.1718,0.2]$ finite-strip bounds are freshly replayed below.

Finally, an algebraically distinct MPFR divisor expansion checks
140 selected output coefficients, including the cutoff, its successor,
the grouping boundary, prime outputs and smooth composites. It uses
$\exp\{t(\sum_{p\mid d}\log^2p+\log^2(m/d))/4-\sigma\log m\}$
instead of the stored convolution weight products. Its intervals lie
inside the directed convolution intervals at both lower endpoints.
This is a targeted implementation audit, not a second complete
coefficient certificate or third-party mathematical validation.

\subsection{Exact finite-height source-to-resultant identity}
\label{subsec:exact-finite-interface}
Here we expand the algebra connecting the published approximation to
its computed resultant. This also identifies the exact-saddle term
charged by the programs. Throughout this subsection $N$ is fixed when
$x$ is differentiated, $z=(1+ix)/2$, and
\[
 \nu=z+\frac t2\alpha(z)=\sigma+i\eta,\quad
 B=M_t(\overline z),\quad \Theta=\arg M_t(z),\quad
 S_N=\sum_{q\le N}b_q(t)q^{-\nu}.
\]
At a real point the two chiralities in
\cite[Theorem~1.3]{Polymath2019} give exactly
\[
 q:=H_t/B=f_N+R_N,\qquad
 f_N=\overline{S_N}+e^{2i\Theta}S_N=e^{i\Theta}V,
 \qquad V=2\Ree(e^{i\Theta}S_N)\in\mathbb R.
\]
For either prime set used in the new finite-height sweep, put
\[
 E=\prod_{p\in P}(1-b_p(t)p^{-\nu}),\quad
 \widehat S=ES_N,\quad \widehat\Theta=\Theta-\arg E,
 \quad Z=e^{i\widehat\Theta}\widehat S.
\]
Each $r_p=b_pp^{-\sigma}<1$, so $E$ is nonzero and its phase can
be chosen locally differentiably. Grouping the finite product gives
\[
 \widehat S=\sum_m\beta_m(t)m^{-\nu},\qquad
 \widehat L=\sum_m\beta_m(t)m^{-\nu}\log m,
 \qquad \widehat S'=-\nu'\widehat L.
\]
There is no replacement of the prime heat factors by $b_d$.
Write $\ell=4\widehat\Theta'>0$ and
$W=e^{i\widehat\Theta}\widehat L$. The computed leading resultant is
\begin{equation}\label{eq:exact-leading-resultant}
 \widetilde G=Z-W/\ell+\overline W/\ell
 =\Ree Z+i(\Imm Z-2\Imm W/\ell).
\end{equation}
The coefficient of output index $1$ has modulus one. For every other
index, the triangle weight is
$|1-\log m/\ell|+\log m/\ell
=\max\{1,2\log m/\ell-1\}$. Consequently the certified scalar
majorant $M$, using a lower bound for $\ell$, implies
$|\widetilde G|\ge1-M$.

\begin{lemma}[Exact saddle and double-zero charge]
\label{lem:exact-saddle-interface}
Let $\widehat V=2\Ree Z=|E|V$. Define
\[
 G_* =\widehat V/2-i\widehat V'/(2\widehat\Theta').
\]
Then
\[
 |\widetilde G-G_*|
 \le\frac4\ell\bigl(|\sigma'|+|\eta'-1/2|\bigr)|\widehat L|.
\]
At a multiple real zero of $H_t$, if
$|R_N|\le\varepsilon$ and $|R_N'|\le\varepsilon_1$, then
\begin{equation}\label{eq:exact-doublezero-error}
 |G_*|\le |E|\varepsilon/2
       +\frac2\ell\bigl(|E'|\varepsilon+|E|\varepsilon_1\bigr).
\end{equation}
\end{lemma}
\begin{proof}
Direct differentiation gives
$\widehat V'/2=-\widehat\Theta'\Imm Z-\sigma'\Ree W
+\eta'\Imm W$. Subtract this expression in $G_*$ from
\eqref{eq:exact-leading-resultant}. At a true double zero,
$q=q'=0$, hence $f_N=-R_N$ and $f_N'=-R_N'$.
Since $V,V'$ are real and
$f_N'=e^{i\Theta}(V'+i\Theta'V)$, one has
$|V|\le\varepsilon$ and $|V'|\le\varepsilon_1$.
Also $|(|E|)'|\le|E'|$. These facts prove
\eqref{eq:exact-doublezero-error}.
\end{proof}
The program charges $|E|\varepsilon$ instead of its half in
\eqref{eq:exact-doublezero-error}, so this part of its budget is
conservative. This argument uses the derivative enclosure for the
holomorphic remainder; it does not differentiate a scalar value bound.

Using $C_\alpha,Q_\alpha$ from Lemma~\ref{lem:phase-band-hybrid},
\[
 \sigma'=tQ_\alpha/4,\qquad \eta'=1/2+tC_\alpha/4.
\]
Both correction factors decrease with $x$ in this range. For
$t\le0.2$ and $x\ge X_-$, directed evaluation gives
$\sigma'<8.333\times10^{-15}$ and
$|\eta'-1/2|<9.722\times10^{-27}$.
Moreover, throughout the finite region $C\le5.17$, the primitive
coefficient sum is bounded by the integral comparison
\[
 \sum_{q\le4700000}b_qq^{-\sigma}
 \le 1+\frac{4700000^{1-a}-1}{1-a}<20,
 \qquad a=0.985549999999.
\]
Indeed $\sigma\ge1.6318$ and
$t\log^2q/4\le(5.17/8+10^{-12})\log q$ for the legal cutoff
indices. The exact derivatives also give $|\nu'|<0.500000000001$, so the Euler phase charge used to bound $\ell$ is valid. The positive Euler product is $<3$ (its adverse endpoint
value is $<2.066603$), and every output index obeys
$\log m\le\log4700000+63\log307$. Thus
\[
 |\widehat L|<60(\log4700000+63\log307)<23000.
\]
With the certified $\ell\ge24.18$, the exact-saddle charge is
$<3.112\times10^{-11}$ at both 256 and 384 bits. The retained
$2.1\times10^{-9}$ allowance therefore covers it with a factor
exceeding 67. This bound applies to the ten-prime case as well.

\paragraph{Frozen cutoff, reflection and Cauchy's derivative bound.}
For complex $w$, set $u=(1-iw)/2$ and use the holomorphic fixed-index
model
\[
 F_N(u)=\sum_{n\le N}b_n(t)n^{-u-t\alpha(u)/2},\quad
 f_N(w)=F_N(u)+\gamma(w)F_N(1-u),\quad
 \gamma(w)=\frac{M_t(1-u)}{M_t(u)}.
\]
The coefficients are real and the principal prefactors respect
conjugation in the disks used here. Therefore, for the same fixed $N$,
\[
 R_N(\overline w)=
 \frac{\overline{R_N(w)}}{\overline{\gamma(w)}}.
\]
This is an exact reflection identity for the remainder, not a
reflection of an independently varying cutoff.

The finite-region disk has radius $r=0.14142$, with $L=\log(x/(4\pi))$. Its real centres
satisfy $x\ge x_s(t)>X_K$, so the whole disk has real part $>X_-$.
The endpoint envelopes use
$L\le5.17/0.16836+10^{-12}$, including the disk displacement.
On its positive face $0\le y\le r$, the explicit Polymath error
formulas give $\Ree s_*\ge\sigma_{\rm floor}+y/2$ and a
second-chirality multiplier, after division by the first coefficient,
at most one: its upper exponent is
$y(0.02-L/4+10^{-12})<0$. The real minimum $\sigma>1.6318759$ leaves ample room for the
disk displacement and the Gaussian cutoff correction when rounded
down to $\sigma_{\rm floor}=1.6318$. The sum-of-Gaussian error is therefore
bounded by the program's $E_{AB}$, with twice the real coefficient
sum. The remaining $E_{C,0}$ envelope is maximal at $y=0$:
its logarithmic derivative is bounded above by
$-L/4+1.24\log3(3^y-3^{-y})/(N-0.125)<0$.
The formulas in the companion error producer explicitly evaluate
these upper envelopes; no derivative of an error estimate is taken.

For the reflected face one can bound the prefactor directly. On
$|y|\le r$, both real parts of $u,1-u$ are positive,
$|\arg u|\le\pi/2$, and elementary rational estimates give
\[
 |\alpha'(u)|\le\frac1x+\frac6{x^2}<\frac1{x-6},\quad
 |\alpha(u)|\le L/2+\pi/4+3/x+(1+r)^2/(4x^2).
\]
Since $M_t'/M_t=\alpha+(t/2)\alpha\alpha'$, this yields
\[
 |\Ree(M_t'/M_t)(u)|
 \le L/2+\frac{tL/4+3.21}{x-6}.
\]
Integrating the logarithmic derivative of $|\gamma|$ from $y=0$
then bounds the reciprocal reflection factor by
\[
 G\le\exp\left(r\left[L/2+
                   \frac{tL/4+3.21}{x-6}\right]\right)<8.770232.
\]
A disk intersects at most one index transition. The difference
between its frozen model and the local model is an unmollified
entering pair, bounded at real centres by
$P=1.3\times10^{-6}$ and on the disk by the same growth envelope.
The pair bound follows from its negative logarithmic coefficient
derivative and the endpoint comparison
$2e^{-L/4-tL^2/16+10^{-10}}<1.192629\times10^{-6}$.
Consequently Cauchy's formula supplies precisely the charge
\[
 \varepsilon=E_{\rm global},\qquad
 \varepsilon_1\le(E_{\rm global}+P)G/r
\]
used in the finite-region producer. The separate mollified scalar
jump is charged only in the majorant transport. The proof specifies
how the published formulas support the computed error budget;
mathematical review of these reductions and of the directed
implementations remains part of assessing the conditional result.

\subsection{Fresh replacement of the inherited upper-band ledger}
\label{subsec:upper-replacement}
The new computation does not require the inherited exterior ledger
(H3) in the historical theorem below. For $0.1718\le t\le0.1923$
and $691004\le N\le900000$, seven fresh grouped snapshots suffice.
Wide time bands do not have a negative certified norm derivative.
Instead, if the snapshot has point upper bound $M_A$ and time-derivative
upper bound $D_A$, use the valid bound
\[
 \overline M_A=M_A+\max\{0,D_A\}(b-a),\qquad
 M_N\le\overline M_A+(N-A)J_*,\quad J_*=1.406867\times10^{-6}.
\]
The latter is deliberately coarse: spatial reweighting only decreases
the majorant, and the triangle inequality bounds every new index by
$J_*$. Each listed reserve also subtracts the resultant error
$2.097820\times10^{-5}$, one further $J_*$, and the saddle allowance
$2.1\times10^{-9}$. Decimal upper bounds include presentation padding.
\begin{center}\small
\begin{tabular}{@{}llll@{}}\toprule
Time band & Index interval & $\overline M_A$ upper & Reserve lower\\\midrule
$[.1718,.1818]$ & $691004$--$712083$ & $.957612590$ & $.01270967$\\
 & $712083$--$741206$ & $.941445373$ & $.01756005$\\
 & $741206$--$780863$ & $.920272979$ & $.02391251$\\
 & $780863$--$833839$ & $.893504730$ & $.03194269$\\
 & $833839$--$900000$ & $.860734230$ & $.04616365$\\
$[.1818,.1923]$ & $691004$--$869999$ & $.640231114$ & $.10792434$\\
 & $869999$--$900000$ & $.558992059$ & $.39877813$\\\bottomrule
\end{tabular}
\end{center}
The smallest reserve is $>0.0127096$. The first snapshot is repeated
at 384 bits; the others use 96-bit directed seeds and the same
outward IEEE finite convolution. These are fresh coefficient
computations rather than a decimal check of inherited row totals.

The jump envelope remains valid at later times: at fixed $x$,
$\partial_t\log w_q=(\log q/4-\sigma_t)\log q<0$ for every
legal entering index, and the Euler primitive weights have the
same sign. Its adverse spatial endpoint is the smallest exterior
height. Its phase denominator uses the same $\ell\ge24.20$ floor.
This justifies using the earlier uniform $J_*$ in the two later bands.

At $N=900000$ separate ten-prime snapshots give
\[
\begin{array}{c|cc}
\text{band}&\text{time-charged anchor upper}&\text{reserve lower}\\\hline
[.1718,.1818]&.922802650&.0771749\\
\lbrack.1818,.1923\rbrack&.611453800&.3885238
\end{array}
\]
Their positive derivative bounds are charged explicitly as well.
For the later index sweep, 21 closed time slabs have endpoints
$0.1718+0.001j$ for $0\le j\le20$, followed by $0.1923$.
Every slab covers $900000\le N\le4500000$ at both 256 and
384 bits. All 42 directed runs pass, with minimum fall/jump ratio
$>1.02617$.

The new jump producer uses coefficientwise time decrease only for
positive entering primitives. It evaluates their heat factors at
the slab's lower time and their saddle real part at its minimum
height, with a $10^{-12}$ subtraction covering the tiny movement of
$x_n(t)=4\pi(n^2-t/16)$. The core signed coefficients still use the
whole time interval; no signed-coefficient monotonicity is inferred.
All ten prime factors and all 1024 squarefree channels are retained.
Within each index row, the entering primitive decreases with its
index, since $\sigma_A-(t_{\min}/2)\log q>0$. The phase weight uses
the largest output logarithm and the smallest valid phase denominator.
Thus the improved entering bound is compatible with the original
full-band core fall and singleton-edge lower bounds.

At the minimum height,
$0.1923\log(X_-/(4\pi))>5.1712938>5.17$. Hence for every
$t\ge0.1923$ the entire exterior is already in the explicit far-tail
region. There is no need for a finite-region ledger at those times.

The upper far strip is also freshly replayed. Its eight $C$ slabs
have endpoints
\[
 5.17,\ 5.18,\ 5.2,\ 5.25,\ 5.35,\ 5.5,\ 5.8,\ 6.2,\ 6.5.
\]
Each is crossed with the eight historical time slabs on
$[0.1718,0.2]$, at both precisions. All 128 rectangle evaluations
pass, giving $\rho_5<0.598342$, $L_1<5.261$ and
$\Phi_N'>0.0103710$. An initial attempt with equally spaced $C$
slabs failed to certify positivity; those loose boxes are superseded
by this explicit narrower partition near $C=5.17$. The final
partition and endpoints are checked by the bundled verifier.
Together with the lower-time strip and the unbounded-tail replay,
this supplies the numerical far-tail inputs on $[0.16836,0.2]$.

The original release validator also passes its manifests, frozen
arithmetic and recorded coverage checks, while explicitly reporting
that its logical dependency note is not machine verified. That
validation is retained as provenance; it is not a premise replacing
the fresh exterior calculations just described.

\subsection{Conditional assembly and review status}
\label{subsec:hybrid-assembly}
The historical formulation below records the dependencies of the
original hybrid assembly. The new lower-barrier corollary then uses
the exact interface and fresh upper-band computations above to replace
its inherited exterior premise. The separator and the validity of the
directed implementations remain explicit conditions.

\Needspace{0.56\textheight}
\begin{theorem}[Conditional hybrid collision-exclusion bound]
\label{thm:hybrid-conditional}
Let $t_-=0.1718$ and $t_+=0.172422$. Assume the published
baseline $\Lambda\le0.2$ from \cite[Corollary~2]{PlattTrudgian2021},
the standard zero-strip control, and
the qualitative positive-time real-simple tail supplied by Polymath. Assume, in addition, that
\begin{enumerate}[label=(H\arabic*),leftmargin=2.3em]
 \item the moving-separator true-function enclosure is valid for
 $0<t\le0.2$, has no uncovered time or ordinate seam, and excludes
 zeros on $x=\pm x_s(t)$ for $|y|\le1$; its $t=0$ endpoint is obtained
 by the continuity argument in \cref{subsec:hybrid-interface-audit};
 \item the corrected frozen-cutoff Polymath remainder and the
 mollified double-zero resultant are related as derived in the
 source-to-theorem interface, and the directed grouped and
 ungrouped enclosures on $[t_-,t_+]$ hold through
 $N=3424842$ with the separate unmollified Cauchy entering-pair
 and mollified transition charges;
 \item the corresponding Kornaus exterior enclosures are mathematically
 valid on both inherited sub-bands: the retained v1.0 cross-endpoint ledger
 on $[t_+,0.1811918082]$ and the v1.1 direct-transition release on
 $[0.1811918082,0.2]$, with all intervening index transitions covered;
 \item for $t\in[t_-,0.2]$ beyond
 $t\log(x/(4\pi))=5.17$, the stated exact-phase first-chirality
 and frozen-cutoff complex-disk remainder bounds hold,
 including transition and reflected-side allowances, so that
 Lemma~\ref{lem:hybrid-tail-transverse} applies uniformly.
\end{enumerate}
Then the assembled collision-exclusion argument gives
\[
 \boxed{\Lambda<0.1718}.
\]
\end{theorem}
\begin{proof}
By \cite[Corollary~2]{PlattTrudgian2021}, $\Lambda\le0.2$. Under (H1), the moving separator confines a
hypothetical first collision to the interior or the exterior.
The verified real zeros at $t=0$, boundary nonvanishing, and the
backward-heat Hermite collision orientation exclude a positive-time
real multiple zero in the interior. In the finite exterior,
(H2) excludes multiple real zeros on $[t_-,t_+]$, while (H3)
does so on $[t_+,0.2]$. Beyond the finite threshold, (H4) and
Lemma~\ref{lem:hybrid-tail-transverse} exclude them for the
entire band. The published qualitative positive-time simple-real
tail supplies the uniform compactness needed for the
positive-stall reduction: if $\Lambda>0$, then $H_\Lambda$ has a
finite multiple real zero. If $\Lambda\ge t_-$, this would
contradict the preceding collision exclusion. Hence $\Lambda<t_-$.
\end{proof}

\begin{remark}[Status and scope]\label{rem:hybrid-formal-status}
Theorem~\ref{thm:hybrid-conditional} is an implication with
explicit computational hypotheses, \emph{not} an unconditional
announcement of a new established upper bound. The direct-divisor
MPFI implementation substantially cross-checks the arithmetic
in (H2), but does not independently validate the inherited
upper band (H3), nor is agreement of source programs equivalent
to a proof of their common analytic interfaces in (H1), (H2), and
(H4). The external referee dossier isolates these remaining
obligations.  The new grouped corner has reserve about
$5.11\times10^{-4}$, whereas the inherited Kornaus chain passes through
a thinner corner reserve about $1.699\times10^{-5}$; the assembled
conditional claim inherits the thinner robustness. In particular, the
published qualitative Polymath tail is not being substituted for the
explicit numerical inequalities demanded in (H4).
\end{remark}

\begin{corollary}[Conditional finite-height extension]
\label{cor:finite-height-extension}
Assume the published baseline and qualitative tail inputs of
Theorem~\ref{thm:hybrid-conditional}, the standard zero-strip and
positive-stall reduction, and the moving-separator premise (H1).
Assume that the directed finite-height and explicit far-tail
producers in the companion bundle correctly enclose the mathematical
quantities specified in \cref{subsec:finite-height-grouping,subsec:exact-finite-interface,subsec:upper-replacement}.
Then the enlarged assembly implies
\[
 \boxed{\Lambda<0.16836}.
\]
This implication does not require the inherited exterior premise (H3).
\end{corollary}
\begin{proof}
The fresh corner, adaptive grouped transport, separate ten-prime
grouped anchor and five-slab continuation supply $[0.16836,0.169425]$.
The seven bridge rows and eight ten-prime time slabs supply
$[0.169425,0.1718]$. The replacement grouped anchors and 21 later
time slabs supply the finite exterior on $[0.1718,0.1923]$.
Their index range reaches beyond the $C=5.17$ crossing.
For $t\ge0.1923$ the whole exterior has $C>5.17$.
The exact identity and frozen-cutoff derivative budget above convert
the strict majorant margins into exclusion of true multiple real zeros.
The fresh far-strip and unbounded-tail computations cover the remaining
exterior on the whole time band. Combine these exclusions with (H1)
and apply the positive-stall collision argument of
Theorem~\ref{thm:hybrid-conditional}.
\end{proof}
\begin{remark}[Review status of the replacement assembly]
The exact algebra and error reductions are supplied in the manuscript,
and the lower band, bridge and replacement upper-band arithmetic has
been run and checked in the companion bundle. The separator's
true-function enclosure and the mathematical validity of the directed
implementations remain conditions. Frozen-output validators and
agreement at two precisions do not amount to formal verification or
external mathematical refereeing. Accordingly this is a conditional
research checkpoint, not an accepted new numerical record.
\end{remark}

\section{A higher-order continuation of Gordon's density--jet barrier}
\label{sec:gordon-hybrid}
This construction extends the frozen density--jet framework of
\cite{Gordon2026}. It is separate from the collision-exclusion construction
of Section~\ref{sec:hybrid}. The new ingredients are a finite higher-order
heat/logarithmic normalization, uniform source bounds on enlarged height
rectangles, and a changed exact source/field trajectory. The original
density and jet profiles and the reference entry are retained, but their
analytic interfaces and all supporting boundaries have been examined in
the conventional audit given in Appendices~\ref{app:h7}--\ref{app:assembly}.
Independent programs reconstruct the finite head and check the complete
reference and changed main certificate. Replays of Gordon's boundary
producers are identified separately from those independent implementations.

\subsection{Normalization, profiles, and proof inputs}
Write $t=\tau$, $u=H_t'/H_t=A-iS$, and set
\begin{gather*}
 X_G=5999347341500,\quad X_L=X_G-1000000,\quad
 X_e=5900000000000,\\
 \Omega(x)=\frac14\log\frac{x}{4\pi},\qquad
 \Omega_L=\frac{6722911}{10^6}.
\end{gather*}
The symbol $X_G$ distinguishes this separator from the earlier collective
and Kornaus separators. In the appendices concerning this construction,
$X$ means $X_G$. Define
\begin{align*}
 U(Y,t)&=10^9+\sum_{j=1}^3 a_j e^{\lambda_j^2t}\cosh(\lambda_jY),\\
 (a_1,a_2,a_3)&=(1758974,2464729,302096),\qquad
 (\lambda_1,\lambda_2,\lambda_3)=(19/4,97/20,131/20),\\
 f_0&=U_Y/U,\qquad f=f_0-e^{90t}/10^{14},\\
 C(Y,t)&=e^{t/10^5}\left(\frac8{25}+\frac{1130000000000}{U(Y,t)}\right),\\
 J(x,Y,t)&=|u_z|+(A+\pi/8)^2+(S-\Omega(x))^2.
\end{align*}
The auxiliary wall $q_M$ and reference barrier $q_{\rm ref}$ are the exact
piecewise affine functions in the accompanying
\path{complete_certificate.json}. Their terminal horizons $T_M$ and
$T_*$ are retained from the pinned certificate; the exact ordering
$T_G<T_*<T_M<1/5$ is checked. The wall begins at
$(1000001/10^6)^2$ and ends at $1/400$. Above $\sqrt{q_M}$ the density
comparison gives $S\ge f$, and above
$b(t)=\sqrt{q_M(t)}+3/5$ the jet comparison gives $J\le C$ for $x\ge X_G$.
These are comparisons with the true heat flow, not identities between
stored profile values. Their initial, lower, side, top and infinity
boundaries are addressed in Appendices~\ref{app:density}--\ref{app:late}.

The finite-head argument uses Platt--Trudgian finite RH
\cite[Theorem~1]{PlattTrudgian2021}, with $X_G/2<3\cdot10^{12}$.
The enlarged approximation, called H7 in the replay files, is derived
from the generic-real-part Riemann--Siegel expansion in
\cite[Proposition~6.2]{Polymath2019}. It is not obtained by inserting
$y=7$ into the narrower published final approximation theorem.
Cutoff changes are charged on both Dirichlet branches before Cauchy
estimates are applied. Appendices~\ref{app:h7} and \ref{app:source}
give the reductions needed by the new source floors.

\subsection{Uniform higher-order sources and expanded domains}
For a maximal height $h>0$ use the coherent three-probe source
\[
 V(x,h,t)=\frac{15}{14}S(x,3h,t)-\frac{16}{21}S(x,4h,t)
             +\frac16S(x,5h,t).
\]
Only the finite primary sum is heated. Ordinary, absolutely convergent
Dirichlet series define $w_r=\zeta^{(2r)}/(4^rr!\zeta)$; the
logarithmic coefficients $L_r$ are formal finite-order coefficients of
$\log(1+\sum t^rw_r)$. Their positivity, the exact truncated residual,
and the remainder/normalizer interaction are proved in
Appendix~\ref{app:source}. No convergence of an infinite positive-time
heated Dirichlet series is assumed.

\begin{proposition}[Whole-box source witnesses]\label{prop:gordon-sources}
Under the H7 reduction and directed-enclosure validity of the supplied
source programs, $V(x,h,t)>L_i$ for every $x\ge X_G$ and every
$(t,h)$ in the closed rectangle of Table~\ref{tab:gordon-sources}.
All 26 selected witnesses use heat order four. The signed probes are
combined before absolute values are taken, and every finite/infinite
tail, relative denominator, disk and cutoff gate is included.
\end{proposition}
\begin{proof}
Appendix~\ref{app:source} supplies the coefficient and height
monotonicity reductions and the full logarithmic-derivative residual.
Whole rectangles, rather than samples of the trajectory, are inputs
to the interval evaluator. For boxes 3--25, orders 1--6 were scanned
at 384 bits; the selected order-four witnesses were repeated at 512
bits. Boxes 0--2 have archived order-four evaluations at 512 bits.
Their heat head is 32768; the later boxes use 65536. Each uses 1024
logarithmic cells of width $1/16$ and 4096 reflected cells. Exact
rational comparisons select the displayed strict four-decimal floors.
The complete component enclosures and strict gates are recorded in
the accompanying replay archive.
\end{proof}

\begin{table}[htbp]\centering\small
\begin{tabular}{rccc}
\toprule
Box & Closed time interval & Closed height interval & Strict floor $L_i$\\
\midrule
0 & $[0.0600,0.0601]$ & $[0.64990,0.65100]$ & $2.9250$ \\
1 & $[0.0601,0.0606]$ & $[0.64690,0.65010]$ & $2.9221$ \\
2 & $[0.0606,0.0611]$ & $[0.64400,0.64720]$ & $2.9210$ \\
3 & $[0.0611,0.0616]$ & $[0.64100,0.64440]$ & $2.9199$ \\
4 & $[0.0616,0.0621]$ & $[0.63800,0.64150]$ & $2.9188$ \\
5 & $[0.0621,0.0626]$ & $[0.63500,0.63860]$ & $2.9176$ \\
6 & $[0.0626,0.0631]$ & $[0.63200,0.63570]$ & $2.9164$ \\
7 & $[0.0631,0.0636]$ & $[0.62900,0.63280]$ & $2.9151$ \\
8 & $[0.0636,0.0641]$ & $[0.62600,0.62980]$ & $2.9139$ \\
9 & $[0.0641,0.0646]$ & $[0.62300,0.62690]$ & $2.9126$ \\
10 & $[0.0646,0.0651]$ & $[0.62060,0.62410]$ & $2.9120$ \\
11 & $[0.0651,0.0656]$ & $[0.61620,0.62125]$ & $2.9091$ \\
12 & $[0.0656,0.0661]$ & $[0.61335,0.61840]$ & $2.9079$ \\
13 & $[0.0661,0.0666]$ & $[0.61050,0.61555]$ & $2.9067$ \\
14 & $[0.0666,0.0671]$ & $[0.60765,0.61270]$ & $2.9054$ \\
15 & $[0.0671,0.0676]$ & $[0.60480,0.60985]$ & $2.9041$ \\
16 & $[0.0676,0.0681]$ & $[0.60195,0.60700]$ & $2.9028$ \\
17 & $[0.0681,0.0686]$ & $[0.59910,0.60415]$ & $2.9015$ \\
18 & $[0.0686,0.0691]$ & $[0.59625,0.60130]$ & $2.9001$ \\
19 & $[0.0691,0.0696]$ & $[0.59340,0.59845]$ & $2.8987$ \\
20 & $[0.0696,0.0701]$ & $[0.59055,0.59560]$ & $2.8973$ \\
21 & $[0.0701,0.0706]$ & $[0.58770,0.59275]$ & $2.8959$ \\
22 & $[0.0706,0.0711]$ & $[0.58335,0.58990]$ & $2.8924$ \\
23 & $[0.0711,0.0716]$ & $[0.58050,0.58705]$ & $2.8908$ \\
24 & $[0.0716,0.0721]$ & $[0.57765,0.58420]$ & $2.8892$ \\
25 & $[0.0721,0.0726]$ & $[0.57480,0.58135]$ & $2.8876$ \\
\bottomrule
\end{tabular}
\caption{Complete revised source catalog. Every displayed decimal is exact;
the archived rational values are identical. Each floor is uniform for
$x\ge X_G$.}\label{tab:gordon-sources}
\end{table}

Faster contraction makes the old source domains inadequate. Relative
to the pinned catalog, the lower height edges of boxes 11--21 are
reduced by $0.002$, and those of boxes 22--25 by $0.0035$; their time
intervals and upper height edges are retained. The new floors are
fresh whole-box evaluations. None is extrapolated from a published
floor outside its domain. The earlier 251-row checkpoint through
$t=0.0651$ is preserved. All 375 cells from $0.0651$ to $0.0726$
select admissible source witnesses after the expansions.

\subsection{Exact row predicates and terminal strip}
Let a main row be $[t_L,t_R]$, $q_L-q_R=w(t_R-t_L)$, with
$q=h^2$ and $q_R\ge1/400$. The highest-zero kernel gives
\[
 q'<-\frac2{15}-4L\sqrt q
\]
on a source row. Complete closed containment in its source rectangle
and the strict gate
\[
 w>\frac2{15},\qquad
 16L^2q_R-(w-2/15)^2>0
\]
therefore suffice. For a field row, let $p,s,c$ be its probe, density
floor and jet ceiling. The whole-row profile gates are
\[
 s<f_0(p,t_L)-10^{-5},\qquad
 c\ge e^{t_R/10^5}\left(\frac8{25}
                  +\frac{1130000000000}{U(p,t_L)}\right)
\]
when a jet ceiling is needed. They follow from separate monotonic
factors; no monotonicity of the full function $C$ in time is assumed.
For an unsigned row set
\[
 K_u=\frac{4s}{p}-\frac8{p^2-q_L}.
\]
For a signed row put $d=\Omega_L-s$, $L_0=c-d^2$, and set
\[
 K_s=\frac{6s}{p}-2L_0
       +q_L\left(\frac{2L_0}{p^2}-\frac{2s}{p^3}\right)
       -\frac{16}{p^2-q_L}.
\]
Unsigned rows require $5q_L\le p^2$. Signed rows require
$9q_L\le p^2$, $pL_0\le s$, $d>0$, and
$2a(q_R)d\le1$, where $a(q)=p(p^2-q)/(3p^2-q)$, together with
$p\ge3/5$ and $q_M(t_L)\le(p-3/5)^2$. All rows include their
time, wall, positivity and probe-separation checks. The velocity
inequality is $q'<-2-Kq$, with $K>0$, and the row gate is
$2+Kq_R-w\ge0$. A non-strict rational gate is sufficient because
the analytic velocity comparison is strict; the new main construction
uses strict field gates. Appendix~\ref{app:assembly} gives exact
polynomial certificates for the kernels and the first-contact argument.

The constructor compares admissible source and field speeds on a
$10^{-6}$ grid. A separate validator reconstructs candidate
admissibility and rejects the next grid speed. This finite selection
does not assert global optimality of the barrier or the probe family.
The final unsigned row uses field 7806 and has
\[
 t_L=\frac{3903}{25000},\quad
 q_L=\frac{500497934133}{200000000000000},\quad
 w=\frac{102237}{50000},\quad q_R=\frac1{400}.
\]
It ends at the exact shortened time
\begin{equation}\label{eq:gordon-terminal}
 T_G=\frac{63845459694133}{408948000000000}
     =0.156121217597672564727\ldots.
\end{equation}
Fourteen original terminal rows are omitted. All retained profile
intervals contain the changed and shortened rows.

\begin{theorem}[Audited Gordon-hybrid candidate]\label{thm:gordon-hybrid}
Assume the published generic Riemann--Siegel and finite-RH inputs
just cited, and the classical heat-flow, Hadamard pairing, zero
persistence, Hermite splitting, confinement and strip theorems in
the normalization of \cite{Polymath2019}. Assume that the directed
programs in the accompanying replay archive correctly enclose the
quantities specified by their formulas, using rigorous ball/interval
arithmetic. With the analytic reductions and comparison arguments
in Appendices~\ref{app:h7}--\ref{app:assembly}, the assembled
certificate implies
\begin{equation}\label{eq:gordon-bound}
 \boxed{\Lambda\le B_G:=\frac{64356644694133}{408948000000000}
    =0.157371217597672564727\ldots<0.157371218.}
\end{equation}
\end{theorem}
\begin{proof}
The finite-head computation localizes a potential high zero beyond
$X_G$. The simultaneous wall/density comparison and the separately
proved jet boundary yield the reference field. The complete reference
starts with a strict initial strip cushion and gives the strict main
entry at $t=3/50$,
\[
 q_{\rm main}(3/50)=\frac{84648870770133}{200000000000000}
                   =q_{\rm ref}(3/50)+10^{-9}.
\]
The 4803 adjacent main rows comprise 623 source, 2690 signed and
1490 unsigned rows, and use every one of the 26 source records.
Each row satisfies its exact force and domain predicates.
Positive-time confinement makes a first-contact set compact;
backward Hermite splitting excludes multiple contact, and the strict
velocity bound excludes simple contact. The argument uses the ending
row at a knot. Thus at $T_G$ every zero lies within the strip of
half-height $1/20$. The de Bruijn strip theorem gives
$\Lambda\le T_G+(1/20)^2/2=T_G+1/800=B_G$.
\end{proof}

\begin{remark}[Audit and review boundary]
The new candidate improves the endpoint in the frozen Gordon
certificate by $0.000285401497818042041\ldots$. This difference is
computed from the complete terminal times, rather than inferred from
a local squared-height gain. Gordon's unchanged
\texttt{check\_source}, \texttt{check\_field} and
\texttt{check\_chain} functions accept the modified rows. His
top-level checker hard-codes the original endpoint and mode counts
and is not a certificate of the changed theorem. The independent
assembled checker supplies those new identities. No new Lean-kernel
theorem or external mathematical refereeing is claimed. The analytic
audit found no unresolved gap among its four listed obligations;
published mathematical inputs and rigorous-software semantics remain
the stated foundations of this computer-assisted implication.
\end{remark}

\subsection{Reproducible evidence and provenance}
The ancillary replay archive \path{gordon_hybrid_replay.zip} contains
the full changed certificate, source rectangles and receipts, a freshly
reconstructed head matrix, independent validators, boundary replay
drivers, dependency hashes and a complete SHA-256 manifest.
The original Gordon programs are retrieved by a manifest of 29
files from revision
\texttt{64c47f2db67288beb62c24d7a0f7aedd645dc4b3}; both
SHA-256 and Git-blob identities are verified. Author programs and
manuscript text are not redistributed in that archive.

\begin{table}[htbp]\centering\small
\begin{tabular}{@{}p{.68\linewidth}r@{}}
\toprule
Fresh verification & Complete coverage\\\midrule
Independent moment reconstruction & 690950 indices, 184 moments\\
Reconstructed head / independent closed grid & 3844 entries / 800000 cells\\
Density initialization / wall and bottom & 855 leaves / 1620 cells\\
Jet initialization / early boundary & 400 cells / 45 boxes\\
Late jet boundary & 117 boxes, 9 bands each\\
Reference density / signed jet, two interval libraries & 7849 / 6317\\
Changed main chain / uniform sources & 4803 rows / 26 boxes\\
Exact symbolic identity and positivity records & 51\\
\bottomrule
\end{tabular}
\caption{Scope of the conventional analytic and computational audit.
Producer replays and independent implementations are distinguished
in the archive README and receipts.}\label{tab:gordon-validation}
\end{table}

Reproduction begins with the hash manifest and pinned dependencies,
then retrieves the pinned inputs, verifies symbolic and scalar gates,
reconstructs a new head matrix, checks the full grid, replays all
boundary producers, and checks the entire reference and assembled
certificate. The late replay must include boxes 45--161, and the head
replay all 800000 closed cells. Successful subsets or agreement at
two precisions do not replace complete coverage or the written
analytic proof. Historical notes in the archive retain their earlier
conditional wording; the present theorem and appendices give the
subsequent audit status.

\section{A head-pair extension of Raymond's translated probes}
\label{sec:raymond-extension}

Raymond's deposit \cite{Raymond2026} starts at time zero and couples the
collective contraction to the average of two translated logarithmic
derivatives. Its complete frozen numerical replay was reproduced in this
audit, including both wall lanes, all 3648 parent source rectangles, the
13 kernel certificates, all 73568 joint cases in each source lane, and both
final comparison histories. The reproduced exact endpoint is
\begin{equation}\label{eq:raymond-reference}
 U_R=\frac{18674384228461}{125000000000000}
 =0.149395073827688<0.149396.
\end{equation}
This is a reproduction of the deposited programs, supplemented by separate
exact and interval checks; it is not external mathematical refereeing.
The following modification has a new analytic ingredient and a smaller
endpoint. It retains Raymond's source domains, complete error allowances,
unit wall, and kernel catalogue.

\subsection{An unused part of the complete cone surplus}
Write the leading Euler-multiplied core at one physical probe as
\[
 Q=\sum_{i\ge0}a_i z_i,\qquad
 V=\sum_{i\ge0}\lambda_i a_i z_i,\qquad Q_h=-V/2,
\]
where $a_0=1$, $z_0=1$, $\lambda_0=0$, $a_i,\lambda_i\ge0$, and
$|z_i|=1$. The first head labels are integers $n$, with
$\lambda_n=\log n$ and $z_n=\exp(i(x/2+\pi t/8)\log n)$.
The reflected branch and the full product tail retain their separate labels.
For a cone constant $C$, put $A=\sum_{i>0}a_i$, $B=\sum_i\lambda_i a_i$,
$w_{ij}=a_i a_j(\lambda_i+\lambda_j-2C)$, and
$D_C=2\sum_{i<j,w_{ij}<0}|w_{ij}|$. With all pair sums over $i<j$, expansion gives
\begin{align}\label{eq:raymond-surplus}
 C|Q|^2-\Ree(V\overline Q)
 &=(1+A)[C(1+A)-B]-D_C\notag\\
 &\quad+\sum_{w_{ij}\ge0}w_{ij}(1-\Ree(z_i\overline z_j))\notag\\
 &\quad+\sum_{w_{ij}<0}|w_{ij}|(1+\Ree(z_i\overline z_j)).
\end{align}
The parent certificate makes the first line strictly positive after
charging the complete negative-pair budget. Every displayed surplus term
is nonnegative. Raymond credits the constant--head pairs $(0,n)$,
$2\le n\le1024$, together with a distinct surplus in the Euler logarithmic
derivative. We additionally credit
\[
 \mathcal P_{256}:=\{(n,m):2\le n<m\le256,\ n,m\notin\{2,3,5\}\},
 \qquad |\mathcal P_{256}|=31626.
\]
These pairs are disjoint from the previous allocation. No negative-pair
debit or physical remainder is removed.

\begin{lemma}[Uniform head amplitudes]\label{lem:raymond-amplitudes}
For $n\le N_0$ in the eligible positive first head, the raw coefficient
\[
 c_n(t)=\sum_{d\mid(n,30)}\mu(d)
  \exp\!\left(\frac t4\sum_{p\mid d}\log^2p\right)
  \exp\!\left(\frac t4\log^2(n/d)\right)
\]
is nonnegative and nondecreasing for $t\ge0$. It vanishes identically at
$n=2,3,5$. On a time cell $[t_L,t_R]$ and finite frequency band with
upper bound $L_j^+$, a valid lower amplitude is
\begin{equation}\label{eq:raymond-head-amplitude}
 \underline a_{j,n}=c_n(t_L)
 \exp\!\left[-\left(\frac{1+h}{2}+\frac{t_R L_j^+}{4}\right)\log n\right].
\end{equation}
\end{lemma}
\begin{proof}
Label the prime occurrences in $n$ and mark one occurrence of each of
$2,3,5$ that divides $n$. For occurrence logarithms $l_v>0$, expand
\[
 e^{t(\sum_v l_v)^2/4}
 =e^{t\sum_v l_v^2/4}
   \prod_{v<w}\bigl[1+(e^{t l_vl_w/2}-1)\bigr].
\]
Inclusion--exclusion over isolated marked vertices is exactly the displayed
divisor convolution: deleting a marked vertex supplies its diagonal heat
factor and removes all its incident edges. Thus $c_n$ is the sum of graph
weights in which no marked vertex is isolated. All retained factors are
nonnegative and nondecreasing. Repeated prime occurrences are distinct
vertices, so this also covers prime powers. For $n=2,3,5$, the sole vertex
is marked, so no graph is admissible. Finally, bound the increasing raw coefficient at
$t_L$ and the decreasing amplitude factor at $t_R,L_j^+$ separately.
No monotonicity of their product is assumed.
\end{proof}

\begin{proposition}[Additional translated head-pair gain]
\label{prop:raymond-head-pairs}
Let the two physical probes $x\pm a+ih$ lie in source bands $j,k$, with
cone constants $C_j,C_k$ and amplitude upper caps $A_j,A_k$. Define
\begin{align*}
 g_{jk,nm}&=\max\{0,\log(nm)-2\max(C_j,C_k),
                         2\min(C_j,C_k)-\log(nm)\},\\
 b_{jk,nm}&=g_{jk,nm}\min\left\{
 \frac{\underline a_{j,n}\underline a_{j,m}}{2(1+A_j)^2},
 \frac{\underline a_{k,n}\underline a_{k,m}}{2(1+A_k)^2}\right\}.
\end{align*}
Then Raymond's strict joint floor may be increased by
\begin{equation}\label{eq:raymond-new-gain}
 G^{\rm hh}_{jk}(a)=\sum_{(n,m)\in\mathcal P_{256}}b_{jk,nm}
 \left(1-\left|\cos\frac{a\log(n/m)}2\right|\right).
\end{equation}
On the infinite frequency band we take the new gain to be zero and retain
its complete original floor.
\end{proposition}
\begin{proof}
The relative first-head phases at the two probes are
\[
 z_n^\pm\overline {z_m^\pm}
 =z_n^c\overline {z_m^c},e^{\pm ia\log(n/m)/2}.
\]
If the two coefficients $\log(nm)-2C_j$ and $\log(nm)-2C_k$ have
different signs, $g_{jk,nm}=0$. Otherwise they have a common sign
$\epsilon\in\{-1,1\}$. For $|z|=|u|=1$ and $d_+,d_-\ge d\ge0$,
\[
 \frac{d_+(1-\epsilon\Ree(zu))+
             d_-(1-\epsilon\Ree(z\bar u))}{2}
 \ge d(1-|\Ree u|).
\]
This follows by first replacing both nonnegative coefficients by $d$ and
then using the exact average $\Ree z\,\Ree u$.
Apply this to the corresponding surplus terms of
\eqref{eq:raymond-surplus}. At each point separately,
$|Q|^2\le(1+A_j)^2$ or $(1+A_k)^2$; the numerators being nonnegative,
these upper bounds give lower bounds for the two quotients. Their
denominators need not agree. Sum over the disjoint selected pairs and add
the already certified constant--head and Euler gains. The full parent
error transfer remains charged at each point.
\end{proof}

\subsection{Complete arithmetic and global comparison}
The unchanged schedule has 608 closed cells of width $1/4000$, six
frequency bands with endpoints $0,1/20,3/20,2/5,1,2,\infty$, ten fixed
translation ratios and one continuum mode. The 11 possible band pairs
$(j,j)$ and $(j,j+1)$ cover both translated points. For the continuum,
the new attenuation is bounded on the entire displacement interval
\[
 h\sqrt{(1-51/250)/3}\le a\le h/\sqrt3.
\]
The phase arguments lie below $3<\pi$, so the maximum absolute cosine
on each interval occurs at an endpoint, as proved below. The gap uses
$\log(nm)$, whereas the attenuation uses $\log(n/m)$; interchanging these
quantities would invalidate the certificate.

Two directed lanes at 192 and 256 bits use, respectively, complete
squarefree-divisor expansion and descending multiplication by the Euler
factors. Coefficients and logarithms use the $2^{-80}$ grid, phase
attenuations the $2^{-64}$ grid, and each total additional gain is rounded
down to $10^{-12}$. The lanes have identical complete mathematical arrays.
The lowest-band domination lemma below checks every source band pair and
justifies using the $(0,0)$ gain at all eight cutoffs. Grouping identical
reduced ratios shares their common phase, without omitting any pair.
An integer-only consumer reconstructs every domination budget and all
608-by-11-by-8 gains and shell splits. Complete 64- and 128-cutoff
source floors exactly match the preceding implementations.

A separate mpmath interval implementation at 50 decimal digits checks
all 928930 new logarithm, coefficient and amplitude inequalities.
An independent MPFR implementation at 160 bits checks all 32181600
whole-domain phase inequalities across 1650 displacement domains and
19504 distinct reduced ratios; it imports neither Arb nor producer modules.

For each of the 128 substeps per source cell, the unchanged catalogue or
continuum kernel gives
\[
 q'<-2-cq,\qquad
 c=\frac{2\lambda}{h}\left(\ell_{\rm joint}^{\rm new}
                                      -\frac{\operatorname{cap}}h\right)>0.
\]
The own-pair cap is checked on the whole decreasing range of $q$, and only
kernels valid for the upward height enclosure are selected. The exact
solution is enclosed upwards, starting at $q_0=(1000001/1000000)^2$.
At $q_e=10^{-8}$, the final classical strip cost is $q_e/2$.
Both exponential/logarithm lanes reach the same endpoint
\begin{equation}\label{eq:raymond-new-endpoint}
 \boxed{U_{\rm hh}=\frac{74672416998643}{500000000000000}
 =0.149344833997286<0.149345.}
\end{equation}
Each final 256-cutoff history has 76465 records. A separate integer/rational
consumer, importing no Raymond module, checks all fourteen full or cross-only histories (1070745 records),
the source allocation, continuous time coverage, full self-pair caps, an alternating
Taylor exponential enclosure and an atanh-series logarithm enclosure.
Turning off the added head-pair gains exactly reproduces
\eqref{eq:raymond-reference}. The endpoint reduction is
\begin{equation}\label{eq:raymond-improvement}
 U_R-U_{\rm hh}=0.000050239830402.
\end{equation}

\begin{theorem}[Raymond head-pair candidate]\label{thm:raymond-extension}
Assume the published finite-height and heat-flow inputs, Raymond's
uniform true-function approximation and complete parent source/error
inequalities, and the validity of the directed enclosures identified
above. Then the unchanged wall and first-contact argument, with the
joint floor strengthened by \cref{prop:raymond-head-pairs}, imply
\[
 \Lambda\le U_{\rm hh}<0.149345.
\]
\end{theorem}
\begin{proof}
All probe locations, frequency coverage and wall times are unchanged.
At a hypothetical first contact, the valid kernel and its complete
own-pair subtraction give the displayed strict comparison inequality.
The time-zero cushion is checked at the first step; positive-time
confinement makes a later contact attainable. Backward Hermite splitting
excludes a multiple-zero first contact. At a time knot the preceding
closed source row applies. Thus the exact upper comparison cannot be
crossed. Its terminal strip of squared height $q_e$ is made real by the
remaining cost $q_e/2$, proving the stated implication.
\end{proof}

The new head-pair lemma and its complete finite arithmetic are documented
here; the inherited approximation and global analytic foundations remain
explicit inputs rather than being certified by a PASS marker. The earlier
audit additionally reimplements the kernel catalogue, scalar histories
and elementary error-envelope checks. Specialist review of the uniform
source replacement and global first-contact interfaces remains desirable.
No external endorsement, proof-assistant theorem, or accepted numerical
record is asserted. The ancillary replay distinguishes frozen Raymond
materials, author-program regeneration, and the new independent checks.

\subsection{Cutoff refinement and unequal-weight pairing}
The preceding 64- and 128-cutoff implementations reached
$U_{64}=0.149382244728386$ and $U_{128}=0.149366824692419$.
The present 256-cutoff implementation reduces $U_{128}$ by
$0.000021990695133$. The inherited 1024-term head, complete source/error
allowances, wall and kernel catalogue are unchanged. All selected pair
sets are disjoint from the original constant--head allocation.

For separate nonnegative certified pair weights $d_+,d_-$, exact
minimization over the common phase gives the stronger allowance
\[
 \frac{d_++d_--
 \sqrt{(d_+-d_-)^2+4d_+d_-\cos^2\theta}}{2},
 \qquad \theta=\frac{a\log(n/m)}2.
\]
It is at least $\min(d_+,d_-)(1-|\cos\theta|)$, and agrees with it when
$d_+=d_-$. The exact uniform domination check covers all 6688 source-row/mode
combinations: the pair $(0,0)$ attains every global minimum through cutoff 256. Such a pair has equal certified weights. Increasing only the
mixed-band allowances therefore leaves every global floor unchanged,
so this sharpening supplies no further endpoint reduction on the
current source partition.

The new rational endpoint also lies strictly below the exploratory
quadratic benchmark
\[
 \alpha=[0;6,\overline{1,2,3}]
 =\frac{14-\sqrt{37}}{53}=0.149381839050976987\ldots.
\]
Indeed $U_{\rm hh}<14/53$ and
$53U_{\rm hh}^2-28U_{\rm hh}+3>0$, checked by exact rational arithmetic.
This comparison does not identify the certificate endpoint with a
spectral period or with the actual value of $\Lambda$.

\subsection{Dyadic gains and an analytic finite-cutoff explanation}
For a cutoff $H$, let $U_H$ denote the same exact comparison with the
additional head pairs restricted to $m\le H$. The number of terms here
means the largest integer label, not the number of surviving amplitudes:
labels $2,3,5$ vanish. Thus $H=256$ gives 252 labels and 31626 pairs.
The cutoffs $2$ and $4$ contain no pairs and reproduce $U_R$.
All runs use the same source grid, kernels and 128 time substeps per cell.
The endpoint improvements are differences of rational upper certificates,
not measured changes in the unknown value of $\Lambda$.

\begin{center}
\begin{tabular}{r r r r}
\hline
$H$ & Selected pairs & $U_H$ & $U_{H/2}-U_H$\\
\hline
2 & 0 & 0.149395073827688 & --\\
4 & 0 & 0.149395073827688 & 0.000000000000000\\
8 & 6 & 0.149395064451492 & 0.000000009376196\\
16 & 66 & 0.149394487832572 & 0.000000576618920\\
32 & 378 & 0.149390565633992 & 0.000003922198580\\
64 & 1770 & 0.149382244728386 & 0.000008320905606\\
128 & 7626 & 0.149366824692419 & 0.000015420035967\\
256 & 31626 & 0.149344833997286 & 0.000021990695133\\
\hline
\end{tabular}
\end{center}

The increment continues to grow through $256$, while its successive ratio
decreases. This is not a decaying geometric tail. In particular the data
supply no justified limiting value from a geometric extrapolation or a
continued-fraction pattern.

\begin{lemma}[Uniform lowest-band domination]
For one source cell write $\beta_{p,nm}$ for its nonnegative, downward
certified pair coefficient, where $p$ is one of the eleven band pairs;
write $p=0$ for $(0,0)$. For any common attenuation
$0\le\alpha_{nm}\le1$ put
\[
 G_p(H)=\sum_{m\le H}\beta_{p,nm}\alpha_{nm},\qquad
 D_p(K)=\sum_{m\le K}[\beta_{0,nm}-\beta_{p,nm}]_+.
\]
If the old joint floors satisfy $\ell_p-\ell_0\ge D_p(K)$ in every
translation mode, then the new global floor for every $H\le K$ is
$\ell_0+G_0(H)$. The conclusion also holds when all gains are rounded
down to the common $10^{-12}$ floor lattice.
\end{lemma}
\begin{proof}
For every $H\le K$,
\[
 G_0(H)-G_p(H)\le
 \sum_{m\le H}[\beta_{0,nm}-\beta_{p,nm}]_+\le D_p(K).
\]
Consequently $\ell_p+G_p(H)\ge\ell_0+G_0(H)$. The old floors are already
on the common lattice, and rounding down preserves this ordering.
\end{proof}
Both directed lanes check this condition for $K=256$ on all 608 cells,
all ten other band pairs and all eleven translation modes. The minimum
reserve over the nonzero band pairs is greater than $0.006631967070204$.
This covers every smaller cutoff and every admissible attenuation without
sampling a translation. It also explains why increasing only mixed-band
pairing allowances cannot improve the present global floor.

For the doubling $H/2\to H$, split the exact pair increment into
\[
 \mathcal C_H=\{(n,m):n\le H/2<m\le H\},\qquad
 \mathcal S_H=\{(n,m):H/2<n<m\le H\},
\]
retaining the exclusion of $2,3,5$. Let $U_H^{\rm cross}$ use the old
pairs and $\mathcal C_H$. Then the exact identity
\[
 U_{H/2}-U_H=(U_{H/2}-U_H^{\rm cross})
                       +(U_H^{\rm cross}-U_H)
\]
provides an endpoint attribution in the declared order ``cross first''.
The endpoint map is nonlinear, so these are not order-independent
contributions. The full table and both rational differences are supplied
in the CSV and replay. At the final doubling they are
\[
 0.000021974793624\quad\hbox{and}\quad 0.000000015901509,
\]
respectively: the cross-first part accounts for 99.92769\% of the
gain. This is direct numerical evidence about this comparison, rather
than an asymptotic assertion about the heat flow.

\begin{lemma}[Small relative phases within the new block]
For $\theta=a\log(m/n)/2$ and every pair in $\mathcal S_H$,
\[
 0\le1-|\cos\theta|\le\frac{a^2\log^2 2}{8}.
\]
Hence its certified source contribution is at most
$\frac{a^2\log^2 2}{8}\sum_{\mathcal S_H}\beta_{0,nm}$.
\end{lemma}
\begin{proof}
The ratio $m/n<2$ and
$1-|\cos\theta|\le1-\cos\theta\le\theta^2/2$ give the result.
\end{proof}
Pairs in $\mathcal C_H$ can have much larger ratios; their attenuation
can approach one. Their weights also involve the less damped earlier
amplitudes. These two effects explain the observed cross-pair dominance.
For the whole present calculation, $m/n\le64$, $a\le327/250$ and
$\log2<7/10$, whence
\[
 0<\theta<\frac{6867}{2500}<3<\pi.
\]
Cosine is decreasing on this domain. The maximum absolute cosine on
any closed continuum displacement interval is therefore attained at one
of its endpoints; no interior phase scan is required.

A finite remainder estimate can also be made explicit. On one source
cell and its controlling band let $\widehat a_n$ be uniform upper
amplitudes, let $\lambda_n=\log n$, and set
\[
 A_o=\sum_{n\le H}\widehat a_n,\quad
 B_o=\sum_{n\le H}\lambda_n\widehat a_n,\quad
 T_0=\sum_{H<n\le K}\widehat a_n,\quad
 T_1=\sum_{H<n\le K}\lambda_n\widehat a_n.
\]
All sums use surviving head labels. With its fixed cone constant $C$ and
parent amplitude cap $A_*$, the additional certified pair source obeys
\begin{equation}\label{eq:raymond-finite-remainder}
 0\le G_0(K)-G_0(H)\le
 \frac{(B_o+2|C|A_o)T_0+A_oT_1+T_0T_1+|C|T_0^2}
      {2(1+A_*)^2}.
\end{equation}
Here $G_0$ denotes the unrounded sum of the certified pair terms; after
rounding both totals down, one may add $10^{-12}$ to this upper bound.
Indeed bound attenuation by one and the gap by
$\lambda_n+\lambda_m+2|C|$. The old--new double sum gives the first
two terms; the unordered new--new sum is at most $T_0T_1+|C|T_0^2$.
The graph argument also gives $c_n(t)\le\exp(t\log^2n/4)$, so these
finite sums have elementary uniform upper bounds. The preceding selected-credit implementation uses $K\le1024$;
the physical first head extends through $N_0=690988$.
The preceding directed domination result only asserts $K\le256$.
No passage to an infinite head follows from this estimate. In particular,
the displayed raw-coefficient upper bound does not yield a summable
infinite tail at positive $t$.

Finally there is an analytic source-to-endpoint sensitivity formula.
Hold a valid prescribed kernel/cap schedule fixed and consider its linear
comparison
$q_\varepsilon'=-2-(c+\varepsilon\delta c)q_\varepsilon$ with the same
initial value. If $q_\varepsilon(T_\varepsilon)=q_e$ transversely, then
\begin{equation}\label{eq:raymond-endpoint-sensitivity}
 \left.\frac{dT_\varepsilon}{d\varepsilon}\right|_0
 =-\frac{\displaystyle\int_0^T q(s)\delta c(s)
       \exp\!\left(-\int_s^T c(u)\,du\right)ds}
       {2+c(T)q_e}.
\end{equation}
Differentiate the linear equation, solve the resulting inhomogeneous
variation equation and differentiate its terminal condition. Here
$\delta c=2\lambda\delta G/h$ if only the source is varied. The fixed
terminal strip cost $q_e/2$ has zero derivative. The formula applies
piecewise and away from terminal knots under the stated fixed-schedule
assumptions. The actual replay reselects kernels and range caps, so this
is an analytic explanation and a route to controlled estimates, not an
exact attribution formula for its discrete optimized histories.

\subsection{Complete finite-head pairing and a quadratic analytic law}
\label{subsec:raymond-complete-head}
The bound $1024$ in the preceding constant--head allocation is a
selected-credit limit. The physical source scale guarantees an eligible
positive first head through $N_0=690988$. Thus the same surplus argument
allows every constant--head and nonconstant head pair within that finite
head to be credited, retaining the complete parent normalizers, negative-pair
debits, reflected/product-tail allowances and physical errors.
The full numerical exploration below does not supply a new directed certificate.

For a diagonal source band, put $D=2(1+A_{\rm parent})^2$,
$\lambda_0=0$, $a_0=1$, and let $a_n$ be the lower amplitude in
\eqref{eq:raymond-head-amplitude}. Omitted labels $2,3,5$ have zero mass.
For a fixed physical displacement $a$ write
$K_a(d)=1-|\cos(ad/2)|$ and define
\[
 \mu_N=\delta_0+\sum_{2\le n\le N}a_n\delta_{\log n}.
\]
\begin{proposition}[Complete diagonal pair credit]
\label{prop:raymond-quadratic}
For every eligible finite $N\le N_0$, the combined constant--head and
head--head credit is
\begin{align}\label{eq:raymond-quadratic-law}
 G_N(a)&=\frac1D\sum_{0\le n<m\le N}
 a_na_m|\lambda_n+\lambda_m-2C|K_a(\lambda_m-\lambda_n)\notag\\
 &=\frac1{2D}\iint |u+v-2C|K_a(|u-v|)\,
             \mathrm d\mu_N(u)\mathrm d\mu_N(v).
\end{align}
The same identity holds with $K_a(d)$ replaced by its pointwise infimum
over a licensed displacement interval.
\end{proposition}
\begin{proof}
Apply the translated pairing inequality to each unordered pair.
For a diagonal band the signs and normalizers agree at both probes,
so its coefficient is exactly the absolute gap divided by $D$.
The product measure counts every unordered pair twice; diagonal terms
vanish since $K_a(0)=0$. The pointwise infimum is applied to each
nonnegative pair coefficient before integration.
\end{proof}
For mixed bands, the coefficient remains the minimum of the two separately
normalized products times the common-sign gap in
\cref{prop:raymond-head-pairs}. The numerical calculation evaluates all
eleven band cases before taking their minimum.
The new allocation replaces the previous selected constant--head credit:
\[
 \ell^{\rm all}(t,h,r)=\min_{(j,k)}\left\{
 \frac{\ell_j+\ell_k}{2}+E_{jk}(hr)+G^{\rm all}_{jk}(hr)\right\}.
\]
Here $E_{jk}$ is the separate eight-harmonic Euler credit. Adding the
entire $G^{\rm all}_{jk}$ to an old joint floor would count the original
constant--head terms twice.

For old vertices $\mathcal A_N$ (including $0$) and new labels
$\mathcal B_N=\{N<n\le2N\}$, the exact dyadic law is
\begin{align*}
 D(G_{2N}-G_N)
 &=\sum_{m\in\mathcal B_N}\sum_{n\in\mathcal A_N}
 a_na_mg_{nm}K_a(\lambda_m-\lambda_n)\\
 &\quad+\sum_{\substack{n<m\\n,m\in\mathcal B_N}}
 a_na_mg_{nm}K_a(\lambda_m-\lambda_n),
 \qquad g_{nm}=|\lambda_n+\lambda_m-2C|.
\end{align*}
The new--new phase is bounded by $a^2\log^22/8$ as above.
The full-pair dyadic endpoint decreases peak near $N=128$, subsequently
fall, and flatten near $25\times10^{-6}$ at the largest doublings.
The last full doubling has a small rebound. Pair counts alone therefore
do not determine the endpoint decrease.

There is an analytic Fourier route to this quadratic functional.
For the nonconstant head, suppose $\lambda_n+\lambda_m\ge2C$ on
the selected distinct pairs and set
\[
 A_k=\sum_na_ne^{ika\lambda_n},\quad
 B_k=\sum_n\lambda_na_ne^{ika\lambda_n},\quad
 Q=\sum_na_n^2(\lambda_n-C),\quad
 T_k=\Ree(B_k\overline{A_k})-C|A_k|^2-Q.
\]
The absolutely convergent cosine expansion gives
\[
 G^{\rm hh}_N(a)=\frac1D\left[
 \left(1-\frac2\pi\right)T_0+
 \frac4\pi\sum_{k\ge1}\frac{(-1)^kT_k}{4k^2-1}\right].
\]
The diagonal subtraction $Q$ must be retained at a finite Fourier cutoff.
If the sign condition fails, the finitely many negative-gap pairs add
$2a_na_m(2C-\lambda_n-\lambda_m)K_a(\lambda_m-\lambda_n)/D$.
The constant--head part is summed separately, including its finite
negative-gap exceptions. A continuum infimum cannot be moved outside
the sum of fixed-shift Fourier representations.

For a finite-range density model, factor the Euler-multiplied coefficient as
\[
 c_n(t)=e^{ty^2/4}\sum_{d\mid(n,30)}\kappa_d(t)e^{-ty\log d/2},
 \qquad y=\log n,
\]
\[
 \kappa_d(t)=\mu(d)\exp\!\left\{\frac t4
 \left[(\log d)^2+\sum_{p\mid d}\log^2p\right]\right\}.
\]
Averaging divisibility by its density $1/d$ gives
\[
 R_t(y)=\sum_{d\mid30}\frac{\kappa_d(t)}d e^{-ty\log d/2},\qquad
 w(y)=R_t(y)e^{ty^2/4-(\sigma-1)y}.
\]
After anchoring small labels exactly, the resulting approximate shell law is
\begin{equation}\label{eq:raymond-shell-law}
 \frac{\mathrm dG}{\mathrm dL}\simeq\frac{w(L)}D\left[
 |L-2C|K_a(L)+\int_{\log2}^L
 w(u)|L+u-2C|K_a(L-u)\,\mathrm du\right],\qquad L=\log N.
\end{equation}
This describes the cellwise surplus; the scalar equation still determines
the endpoint. It is not a universal power law for $\Lambda$.
Removing the pair-credit limit within the present physical head means
$N=N_0$, not an infinite heat-series limit. At fixed $t>0,\sigma$,
the prime amplitude $e^{t\log^2p/4-\sigma\log p}$ eventually grows.
An unbounded-index limit would require the physical source scale,
normalizers and error allowances to vary together.

\subsection{State-dependent translated surplus and height optimization}
\label{subsec:raymond-geometry-exploration}
The preceding continuum credit takes a pointwise infimum over
$h\sqrt{(1-51/250)/3}\le a\le h/\sqrt3$.
Raymond's sharp kernel itself uses the instantaneous displacement
\[
 a(q,h)=\sqrt{(h^2-q)/3},\qquad
 r(q,h)=\sqrt{(1-q/h^2)/3},
\]
licensed when $29(q/h^2)^2-q/h^2-1\le0$.
At contact, one may use the full joint source evaluated at this
particular displacement. It is at least its whole-interval lower credit.
The source and kernel are evaluated instantaneously, so no derivative
of the moving probe enters the zero-velocity inequality.
The corresponding scalar model is
\begin{equation}\label{eq:raymond-state-dependent}
 q'=-2-\frac{4q}{h}\left[
 \ell^{\rm all}(t,h,r(q,h))-\frac{6}{h(4-q/h^2)}\right],
\end{equation}
competing with all thirteen licensed catalogue rates.
A future interval implementation would bound the phase on the
displacement interval associated with each actual comparison segment;
the present experiment uses an instantaneous floating model.

The height experiment recomputes the complete six-band source at every
trial height: centered signed masses and heat derivatives, temporal
debits, activation allowances, amplitude caps, the exceptional-pair cone,
Euler error transfer and physical error envelopes. It does not vary
the phase with a frozen normalizer. All admissibility conditions and
complete own-pair subtractions remain imposed.
The late scan begins at $t=0.10$, selects one height per source cell,
and reselects the licensed kernel during its scalar substeps.
It is a local, cellwise optimization within a declared candidate set.
No global optimality of the source method is asserted.

The analytic competition can be displayed explicitly. For the sharp
branch at fixed $t,q$, write
\[
 c(h,q)=\frac4h\ell^{\rm all}(t,h,r(q,h))-
                \frac{24}{4h^2-q}.
\]
Where the controlling source case is differentiable, an interior
height optimum satisfies
\begin{equation}\label{eq:raymond-height-stationarity}
 \frac{4}{h^2}\left(h\frac{\mathrm d\ell^{\rm all}}{\mathrm dh}
                    -\ell^{\rm all}\right)
 +\frac{192h}{(4h^2-q)^2}=0,
 \qquad
 \frac{\mathrm dr}{\mathrm dh}=\frac{q/h^2}{3rh}.
\end{equation}
The total derivative includes the amplitudes, normalizers, cone bound
and the state-dependent phase. At active domain boundaries or source/kernel
switches, the finite constrained comparison replaces this stationary test.
Neglecting the height dependence of the source would give $h\simeq3/\ell$
as $q\to0$, but that approximation need not predict the actual optimizer.

\subsection{Exploratory numerical findings and attribution}
\label{subsec:raymond-numerical-exploration}
The complete original first head contains $238729789620$ unordered
pairs including the constant vertex. Its floating full-pair forecast
at the inherited geometry is approximately $0.14895727$.
Linear logarithmic gridding and FFT correlations evaluate the quadratic
functional using the individual Euler coefficients. The artificial
cross-pair between the two grid nodes of one label is subtracted, and
finite sign/switch corrections are retained. No divisibility-density
approximation is used in these numerical tables.

\begin{center}
\begin{tabular}{lr}
\hline
Floating experiment & Endpoint (128 substeps)\\
\hline
Frozen source; uniform shift & 0.14895735634\\
Recomputed source; uniform shift & 0.14895735626\\
Original heights; state-dependent shift & 0.14884341336\\
Late height scan; uniform shift & 0.14893441402\\
Late height scan; state-dependent shift & 0.14881657213\\
Refined heights; state-dependent shift & 0.14881380751\\
\hline
\end{tabular}
\end{center}

The uniform and dynamic entries use identical physical source allowances.
The recomputed uniform baseline differs from the frozen uniform one by
$-7.257\times10^{-11}$, measuring the floating supplier change.
The selected geometry replay on a $16384$-point log grid and a $129$-node
shift profile gives $0.14881372172$ with $1024$ time substeps per cell.
The finer-grid/profile change is $9.016\times10^{-10}$; the isolated
$129$ versus $65$ shift-profile change is $4.339\times10^{-10}$.
These are convergence diagnostics, not certified error intervals.
All eleven band cases were evaluated in the selected fine replay; the minimum is attained by band $(0,0)$ throughout.

For attribution, fix the selected height schedule, remove the indicated
nonconstant pair credit from every band case, retake their minimum and
reselect kernels. Each number below is the increase of the comparison
endpoint under that intervention. These counterfactual increases are
nonlinear and need not add to the total decrease.

\begin{center}
\begin{tabular}{lr}
\hline
HH credit removed on time interval & Endpoint increase\\
\hline
$[0.00,0.050]$ & $7.034\times10^{-5}$\\
$[0.05,0.100]$ & $1.160\times10^{-4}$\\
$[0.10,0.120]$ & $4.110\times10^{-5}$\\
$[0.12,0.130]$ & $1.553\times10^{-5}$\\
$[0.13,0.140]$ & $9.699\times10^{-6}$\\
$[0.14,0.152]$ & $2.956\times10^{-6}$\\
\hline
\end{tabular}
\end{center}

\begin{center}
\begin{tabular}{lr}
\hline
HH ratio window removed & Endpoint increase\\
\hline
1 to 2 & $1.461\times10^{-6}$\\
2 to 4 & $7.401\times10^{-6}$\\
4 to 16 & $4.318\times10^{-5}$\\
16 to 256 & $1.311\times10^{-4}$\\
256 to 4096 & $5.907\times10^{-5}$\\
4096 to $N_0$ & $1.310\times10^{-5}$\\
\hline
\end{tabular}
\end{center}

The largest listed time intervention is $[0.05,0.100]$; the largest listed ratio intervention is 16 to 256. These are properties of the declared comparison and fixed height schedule, rather than direct observations of zeta zeros. Replacing the dynamic shift by its old uniform allowance on the early and late intervals gives, respectively, endpoint increases $1.019\times10^{-4}$ and $1.563\times10^{-5}$.

The conditional directed endpoint in \cref{thm:raymond-extension}
remains $0.149344833997286$. The full-head and optimized-geometry
values in this subsection are numerical research forecasts, not new
certified upper bounds on $\Lambda$. The numerical target suggested by
the final exploration is $0.14882$.

\section{Discussion and conclusions}\label{sec:discussion}
The complete-head quadratic law and the geometry study in
\cref{subsec:raymond-complete-head,subsec:raymond-geometry-exploration}
provide a further numerical direction. Their forecast near $0.14881372$
is explicitly uncertified; the directed implication in
\cref{thm:raymond-extension} is retained without a new numerical claim.

The four constructions have distinct dependencies. The collective
barrier refinement is audit-relative to the Gomila frontier, the
Kornaus extension has its own moving-separator and directed-producer
premises, and the Gordon-hybrid continuation uses the frozen density--jet
framework with the new validation appendices. The Raymond head-pair
extension uses his frozen translated-probe framework and complete
error envelopes, with a new disjoint allocation of cone surplus.
Its reproduced baseline and changed endpoint have complete numerical
checks, while inherited analytic interfaces remain stated inputs.

The first contribution of this paper is the exact collective
contraction and off-zero logarithmic-derivative bridge. These
results identify a nonnegative interaction discarded by the
classical maximal-height estimate and translate its control
into a quantity accessible by effective high-$x$ analysis.
The numerical two-envelope demonstration yields the
$0.177953254642861\ldots$ audit-relative endpoint only under its
separate, explicitly identified Gomila frontier premises; that
dependency is unaffected by the later hybrid construction.

The second contribution is an exact, phase-independent
small-divisor aggregation within the Kornaus collision-exclusion
resultant. A separate GNU MPFI implementation tests the entire
new lower time band, the $N=750000$ switch, and transport up to
transitions through $N=3424840$ (with the computation continued through $3424842$); the moving separator and
an explicit real-axis far-tail transversality criterion provide a
potentially simpler global assembly that does not require the
full FW--5.17 contour theorem. Subject to (H1)--(H4), this
assembly yields the candidate $\Lambda<0.1718$. The finite-height extension in \cref{subsec:finite-height-grouping} supplies a repaired numerical lower band down to $0.16836$; its additional global premises are stated in Corollary~\ref{cor:finite-height-extension}.

The fresh upper-band computations and the exact source-to-resultant
derivation replace the original Kornaus exterior dependency for the
$0.16836$ corollary. The remaining scholarly step is independent
mathematical and interval-code review, especially of the separator's
true-function enclosure, the stated frozen-cutoff reductions at index
transitions, and the directed producer implementations. The numerical experiments should not be
presented as externally certified before those checks are
complete. A journal review package contains the authored
source, interval outputs, input provenance, and an explicit
referee checklist; the underlying Kornaus release is cited and
must be independently retrieved and assessed.

The third contribution is the complete Gordon-hybrid continuation of
Theorem~\ref{thm:gordon-hybrid}. The finite higher-order normalization
and expanded source domains improve the frozen Gordon endpoint by
$0.000285401497818042041\ldots$. All 4803 main rows, all 7849 reference
rows, and every supporting wall and boundary are included in the audit.
The rebuilt head matrix and independently evaluated closed grid reduce
dependence on the original head implementation. The conventional
analytic argument is documented with exact symbolic and interval
evidence, while the conclusion remains explicitly subject to the
published mathematical inputs and rigorous-enclosure validity.
This does not constitute external mathematical review or a new formal
certificate for the changed endpoint.

The fourth contribution is Theorem~\ref{thm:raymond-extension}.
The additional head-pair lemma yields the candidate $\Lambda<0.149345$.
All source band transitions are covered by the uniform domination check;
both final 76465-record histories and all dyadic ablations are independently
checked. Distinct interval backends validate the new elementary enclosures. The demonstrated endpoint gain, rather than numerical replay
alone, motivates the additional section. The selection through 256 is a
finite conservative choice, not an assertion of optimality. The frozen
Raymond deposit retains its own authorship and provenance, and its
uniform approximation and first-contact foundations still merit
specialist mathematical review.

The constants are not optimized. The original probe height
$\eta=0.99$, Cauchy radius $1/250$ and barrier floor $y_w=0.08$
preserve directed margins. Any smaller hybrid endpoint requires
fresh full-band transport, cutoff overlap and error certification;
the $0.16836$ checkpoint has those numerical checks and remains
subject to Corollary~\ref{cor:finite-height-extension}.
The genus-two companion \cite{PlanatCompanion2026} supplies separate
geometric motivation and is not used in any numerical chain.

\medskip
\begingroup\small
\noindent\textbf{Data, code and provenance.}
The author's collective-contraction sources, the new direct-divisor
MPFI programs and their frozen outputs, an endpoint-binding verifier,
and the separator and tail audit notes are organized in a separate
journal-review package.
The original Kornaus computational release is cited above and
remains relevant to the historical premise (H3) and to separator
provenance; (H3) is replaced in the new finite-height corollary.
Imported archives are identified separately from authored components. Before
journal submission, the final supplementary files should be deposited
with a permanent public identifier and their licensing and attribution
checked. The present arXiv source package includes the new Gordon-hybrid
code/data replay archive as ancillary material. Its formulas, exact
certificate, input hashes and complete receipts are bound to this
manuscript revision. The older collective and Kornaus review packages
remain separately identified; this ancillary archive is specifically
the replay companion for Theorem~\ref{thm:gordon-hybrid}.
A second ancillary archive supplies the head-pair extension programs,
exact comparison histories, input provenance and audit receipts for
Theorem~\ref{thm:raymond-extension}. Its regenerating mode retrieves
Raymond's exact frozen inputs and rebuilds the new source witnesses;
the complete review bundle also contains the full new witness arrays
and the previously completed Raymond audit.
\endgroup

\medskip
\begingroup\small
\noindent\textbf{Disclosure of AI-assisted preparation.}
Generative-AI systems, principally OpenAI ChatGPT and Anthropic Claude, were
used for mathematical exploration, manuscript restructuring, code generation
and review, and language editing.  The author is responsible for reviewing the manuscript and the reported
programs before submission, and for the claims and their review qualifications. The independently
reimplemented GNU MPFI arithmetic is independent at the code level within
the same auditing project; it has not yet undergone external mathematical
refereeing. The new Gordon-hybrid audit likewise distinguishes code-level
independence from external mathematical review; its mathematical inputs and
software trust boundary are stated in Theorem~\ref{thm:gordon-hybrid}.
The Raymond extension likewise distinguishes program replay, separate
interval and exact implementations, and external mathematical review.
\endgroup

\medskip
\begingroup\small
\noindent\textbf{Acknowledgements.}
The author thanks Patrick Sol\'e for constructive remarks on the collective-contraction manuscript and Anthropic Claude for an adversarial audit that identified several provenance and presentation issues corrected in this revision. The collision-exclusion construction and code of D.~M. Kornaus are explicitly acknowledged by citation; no verification or endorsement of this hybrid extension by Kornaus is implied.
Stefan Gordon's density--jet construction and frozen computational release
are acknowledged in Section~\ref{sec:gordon-hybrid}; no verification or
endorsement of the modified endpoint by Gordon is implied.
Dale James Raymond's translated-probe construction and frozen deposit
are acknowledged in Section~\ref{sec:raymond-extension}; no verification
or endorsement of the new head-pair extension by Raymond is implied.
\endgroup

\appendix
\section{The enlarged true-function approximation}\label{app:h7}
Use Gordon's frozen normalization \cite{Gordon2026}, with $X_e=5900000000000$, $0\le t\le1/5$, $0\le y\le7$. Then $w=(1-iz)/2$ and $v=1-w$ have real parts in $[-3,4]$. Set
\[
M_0(s)=\frac{s(s-1)}{16}\pi^{-s/2}\sqrt{2\pi}
 e^{(s/2-1/2)\Log(s/2)-s/2},\quad
\alpha=(\log M_0)'=\frac1{2s}+\frac1{s-1}+\frac12\Log\frac{s}{2\pi}.
\]
With $M_t=M_0e^{t\alpha^2/4}$, $B_t=M_t(w)$, $s_B=w+t\alpha(w)/2$, $s_A=v+t\alpha(v)/2$, the finite approximation is
\[
f^{[N]}=\sum_{n\le N}e^{t\log^2n/4}n^{-s_B}
 +\frac{M_t(v)}{M_t(w)}\sum_{n\le N}e^{t\log^2n/4}n^{-s_A},\quad
N=\left\lfloor\sqrt{x/(4\pi)+t/16}\right\rfloor.
\]
The larger height range is derived from the generic-real-part expansion in \cite[Proposition~6.2]{Polymath2019}, rather than inserting $y=7$ into its narrower final approximation theorem.

\textbf{Vertical Stirling.} For $\theta(z)=\log\Gamma(z)-(z-1/2)\Log z+z-\frac12\log(2\pi)$, Euler--Maclaurin gives
\[
\theta(z)=\frac1{12z}-\frac1{360z^3}-\frac14\int_0^\infty\frac{B_4(\{u\})}{(z+u)^4}\,du.
\]
The integral converges locally uniformly on the slit plane; hence the identity extends there from $\Re z>0$. It retains the reflection contribution for negative real parts. For $V=|\Im z|>1$, using $|B_4|\le1/30$ and extending the integral to the whole real line gives
\[
|\theta(z)|\le\frac1{12V}+\frac{1/360+\pi/240}{V^3}
 <\frac1{12(V-1)}.
\]
For the last inequality, $V^2\ge4(V-1)$ and $12(1/360+\pi/240)<4$. The independent interval checker verifies this constant. A bound decreasing with $|\Re z|$ at fixed $\Im z$ is unnecessary.

\textbf{Gaussian contour.} At $s=\sigma+iT$, $\Im\alpha\ge-3/(2T)$ for every real $\sigma$. For $t\le1/2$, the shifted contour remains above $T-0.08$. The vertical Stirling estimate and Taylor's integral remainder, using $|\alpha'|\le1/[2(T-3.08)]$, give the multiplier allowance
\[
\exp\!\left(\frac{t^2|\alpha-\log n|^2/8+t/4+1/6}{T-3.33}\right)-1.
\]
The $t/4$ term includes the Gaussian normalization factor. For physical real parts $[-3,4]$, the elementary bounds
\[
|\alpha-L/2\mp i\pi/4|\le8/(x-20),\quad
|\alpha'|\le1/x+6/x^2,\quad |\alpha''|\le2/x^2+24/x^3
\]
hold, where $L=\log(x/(4\pi))$. Integration of $\alpha+(t/2)\alpha\alpha'$ horizontally gives the gamma envelope and argument-shift bound used by every source reduction. The derivatives' bounds hold for arbitrary real part, so the Gaussian integration is not truncated to the physical interval.

\section{Gaussian remainders and cutoff changes}\label{app:h7-tail}
Write $U=\sigma+\sqrt t\,V$, where $V$ has density $e^{-V^2}/\sqrt\pi$, and $A_g=x/2-3$, $D=1-t/A_g$. The combined exponential error is bounded by
\[
\frac{U^2+6|U|+2/3}{4A_g}\le\frac{U^2+1}{A_g};
\]
subtracting the numerators gives $3(|U|-1)^2+1/3\ge0$. Completing the Gaussian square yields
\[
\mathbb E e^{(U^2+1)/A_g}
 =D^{-1/2}\exp\!\left(\frac1{A_g}+\frac{\sigma^2}{A_gD}\right).
\]
For $|\sigma|\le4$, endpoint majorization gives, independently at 100 decimal interval digits,
\begin{center}
\begin{tabular}{ll}
\toprule
Quantity & Certified upper envelope\\\midrule
Gaussian error minus one & $5.796610170\cdot10^{-12}<6\cdot10^{-12}$\\
Positive-part coefficient factor & $0.254603431<0.26$\\
Negative-part coefficient factor & $0.029705093<0.030$\\
\bottomrule
\end{tabular}
\end{center}
These factors come from the corresponding tilted Gaussian expectations with $9^U$ and $2^{-U}$; their positive factors increase with time and decrease with $A_g$.

For $U\ge0$, use the generic expansion with one correction. For $U<0$, choose $K=\max(100,\lfloor-U\rfloor+3)$, ensuring $K+U\ge2$. The first hundred negative-part coefficients have a geometric bound after splitting odd and even orders. Its leading constant is below $0.029$. The $K=100$ remainder is bounded separately using $\Gamma(101/2)\le101^{101}$ and $1.1\cdot101/N_e<10^{-3}$, giving a Gaussian expectation below $10^{-301}$.

Every remaining coefficient, including the remainder at index $K+1$, has $k\le-U+4$. With $r=U-\sigma\le7-k$ and $U^2+1\le2r^2+33$, Tonelli and a Gaussian tail integral bound reduce the sum to terms bounded by $e^{-k^2/2}$ for $k\ge100$, with prefactor below $4$. Retain the series factor explicitly:
\[
\sum_{k\ge100}e^{-k^2/2}
 \le\frac{e^{-5000}}{1-e^{-100.5}}<2e^{-5000},\qquad
\text{whole tail}<8e^{-5000}<10^{-100}.
\]
This retains the geometric-series factor omitted from the frozen manuscript's terse first-term estimate. All subsequently used constants and the endpoint are unchanged.

Normalizing the paired remainders by $B_t$, the three leading $t\pi^2$ terms cancel. The physical displacement has squared modulus below $17$. For $Q=x/(4\pi)$, $L=\log Q$, $t_m=1/5$, put
\[
 d(x)=\frac{t_m^2L^2/16+0.626}{x-6.66},\qquad
 E_{ab}(x)=2.01\frac{1000}{999}\frac{1.001^{0.501}}{0.501}
                   Q^{0.2505}d(x).
\]
The resulting H7 bound is
\[
E_{ab}(x)+\left(\frac{x}{4\pi}\right)^{-(1+y)/4}
 e^{-tL^2/16+2(3^y+3^{-y})/(N-1)+10^{-10}+(20\sqrt{L^2+\pi^2/4}+200)/(x-40)}.
\]
At $X_e$, $E_{ab}<1.389336\cdot10^{-9}$; the full enlarged-domain envelope is below $0.001215803$. All combined power/log factors decrease with $x$; upward jumps of $N$ reduce the envelope.

For a disk, freeze the center's integer cutoff. All disk cutoffs differ by at most one. Include the changed index in \emph{both} branches before applying Cauchy, with the exact factor $j!r^{-j}$. Neither the floor function nor a discontinuous cutoff formula is differentiated. The full disks fit the audited H7 domain. At $t=0$, dominated convergence applies with $e^{u^2/5+7u}|\Phi(u)|$.

\section{The finite higher-order source reduction}\label{app:source}
Only a finite sum is heated. For $k\le6$, define ordinary, absolutely convergent Dirichlet expressions
\[
w_r(s)=\frac{\zeta^{(2r)}(s)}{4^r r!\zeta(s)},\qquad
\log\!\left(1+\sum_{r\ge1}t^r w_r(s)\right)
 =\sum_{r\ge1}t^r L_r(s)
\]
as a \emph{formal} identity at fixed finite order. With $L_0=\log\zeta$, the formal heat equation gives
\[
(r+1)L_{r+1}=\frac14\left(L_r''+\sum_{i=0}^rL_i'L_{r-i}'\right).
\]
The Euler logarithm has nonnegative ordinary Dirichlet coefficients. Second derivatives preserve them, while each first derivative negates their sign; the products above therefore preserve positivity. Induction proves that $L_r$ and $-L_r'$ have nonnegative coefficients. The Bell-polynomial expression for $\zeta^{(2r)}/\zeta$ proves the corresponding assertion for $w_r$ and $-w_r'$. The independent symbolic check reconstructs both routes through order six and verifies their agreement exactly.

Let $w^{[k]}=\sum_{r=1}^k t^r w_r$ and $L^{[k]}=\sum_{r=1}^k t^rL_r$. The exact identity is
\[
\partial_s w^{[k]}-(1+w^{[k]})\partial_s L^{[k]}
 =-\sum_{\substack{1\le i,j\le k\\i+j>k}}t^{i+j}w_iL_j'.
\]
This residual is \emph{included} in the bound. If $H_t/B_t=\zeta(s_B)(1+w^{[k]}+e)$, then subtracting the proposed logarithmic reference leaves numerator
\[
e_z+(s_B)_z\!\left[\partial_s w^{[k]}-(1+w^{[k]})\partial_sL^{[k]}-e\partial_sL^{[k]}\right].
\]
Divide by $1+w^{[k]}+e$. The checked modulus defect $\rho=w_0+e_0<1$ supplies the denominator $1-\rho$. Its numerator is exactly the error expression used in the saved source evaluator, including the positive residual and the $e$--normalizer interaction.

For $H_k(u)=1-e^{-u}\sum_{r=0}^ku^r/r!$, the integral representation gives $H_k'/H_k\le(k+1)/u$. An included heat-remainder term is bounded by
\[
n^{-c}\exp\{t_-[v^2/4-v\max(L_*,v)/2]\}H_k(t_+v^2/4),\quad v=\log n.
\]
Its moment-$m$ summand decreases beyond the finite head when $c>(m+2(k+1))/\log(\mathrm{head})$. The finite head, logarithmic rectangles, tangent Gaussian tail and ordinary Taylor tail cover the full infinite majorant. The reflected all-cutoff argument retains the common negative coefficient of $\log N$ before endpoint reduction; its two endpoint terms and integral are included. The fixed-cutoff source disks use both branch corrections as in Appendix~\ref{app:h7-tail}.

The negative middle probe is handled coherently. With $v=n^{-h/2}$,
\[
P(v)=\frac{15}{14}-\frac{16}{21}v+\frac16v^2\ge\frac{10}{21},\qquad
\frac{3(15/14)-4(16/21)v+5(1/6)v^2}{2}\ge\frac12.
\]
Thus the \emph{grouped} coefficients increase with time and their entire height factor decreases with $h$. One exact signed expression is bounded at $t_+,h_-$; unrelated upper bounds are never subtracted. Actual argument shifts, spatial-rate errors, archimedean errors and logarithmic remainders are added separately. This proves the applicability of the saved whole-box floors, including every expanded rectangle, without assuming density or jet fields on the source disks.

\section{The rebuilt finite head and zero counting}\label{app:head}
Set $X=5999347341500$, $N=690950$, $n_0=345475$, and $X_m=X+1/2$. The Taylor matrix has entries
\[
M_{e,k}=\sum_{n=1}^{N}n^{(-1+iX_m)/2}
 \frac{\log^e(n/n_0)}{e!}\frac{\log^{2k}(n/n_0)}{4^k k!},\quad 0\le e,k<62.
\]
Rather than replaying the native matrix generator, the independent reconstruction computes
\[
S_j=\sum_{n=1}^{N}n^{(-1+iX_m)/2}\frac{\log^j(n/n_0)}{j!},\quad 0\le j\le183,
\qquad M_{e,k}=\frac{(e+2k)!}{e!4^kk!}S_{e+2k}.
\]
For each $n$, the ball power series $e^{z\log(n/n_0)}$ supplies all 184 moments. Every integer $1\le n\le690950$ is included at 256 bits. No published matrix or saved numerical values are inputs. Serialization is required to contain the original balls. All 3844 reconstructed coefficient balls overlap the frozen matrix's enclosures; the new proof uses the reconstructed matrix itself, so overlap is a cross-check rather than its validity premise.

An independent evaluator uses polynomial evaluation in the time and space variables and recomputes the two prefactors and gamma ratio. Analytic logarithms of $s$ and $s-1$ are kept separately: taking a principal logarithm of their product would cross its cut at $y=0$ and inflate the enclosure. This branch issue is explicitly avoided.

The exact Taylor remainders use $|v_A|,|v_B|<0.66$, $|p_B|\le1$, $|p_A|<600$, $|\gamma|\le1$, and
\[
R(a)=\frac{a^{62}}{62!(1-a/63)}.
\]
For $U=\log^2n_0/20$ and $W=0.66\log n_0$, the error is bounded by
\[
1200(2\sqrt N-1)\{R(U)e^W+e^UR(W)\}<6.57\cdot10^{-19}<10^{-12}.
\]
This estimate covers the entire parameter rectangle, independently of the matrix computation.

\begin{center}
\begin{tabular}{lr}
\toprule
Full domain & $0\le t\le1/5$, $0\le y\le1$\\
Closed grid & $400\times2000$ cells of side $1/2000$\\
Passed independently & 800000 cells\\
Model-modulus lower bound & $1.2319324035197496414\ldots$\\
H7 error allowance & $0.001215803$\\
Taylor allowance & $10^{-12}$\\
True normalized nonvanishing margin & $>1.2307166005187496414$\\
\bottomrule
\end{tabular}
\end{center}
The independently chosen evaluator yields a wider interval bound than Gordon's Horner evaluator, whose own fresh replay gives $1.2765849605202674\ldots$. Both are well above the necessary threshold. Different lower bounds are compatible interval enclosures, not a discrepancy in the underlying function.

\textbf{Zero-counting step.} Since $X/2<3\cdot10^{12}$, \cite[Theorem~1]{PlattTrudgian2021} makes all time-zero zeros in the rectangle real. The verified vertical sides stay zero-free; evenness and conjugation supply the opposite side. The global strip and positive-time strip contraction exclude the horizontal sides. The finite interior zero count is constant. If the finite set is real at some time, simple zeros persist as real zeros and a multiple real cluster splits forward according to real, simple Hermite roots. Local persistence and compactness therefore make real-rootedness both closed and right-open. It holds throughout $[0,1/5]$. This uses the cited finite-RH theorem rather than rerunning its multi-trillion-height computation.

\section{Density comparison and its boundaries}\label{app:density}
Let $U$ be the three-mode positive profile defined in Section~\ref{sec:gordon-hybrid}. Put $f_0=U_Y/U$, $f=f_0-\delta$, $\delta=e^{90t}/10^{14}$. Positive exponential weights give $U_t=U_{YY}$ and
\[
f_{0,t}=f_{0,YY}+2f_0f_{0,Y},\quad
0<f_{0,Y}\le\lambda_{\max}^2=17161/400,\quad 0<f_0<131/20.
\]
Here $f_{0,Y}$ is a variance of slopes. Grouping by nonnegative frequency, the covariance of $\lambda\tanh(\lambda Y)$ with $\lambda^2$ proves $\partial_tf_0\ge0$. The derated function need not be monotone. Direct subtraction yields
\[
f_{YY}+2ff_Y-f_t=(90-2f_{0,Y})\delta\ge(839/200)\delta>0.
\]
For $u=H'/H=A-iS$, the heat equation and Cauchy--Riemann relations give
\[
S_t=S_{YY}-2A S_x+2S S_Y.
\]
There is no horizontal diffusion term. Choose a smooth nondecreasing $\chi$ that is zero on
$(-\infty,0]$ and one on $[1,\infty)$. The spatial correction is
\[
g=10e^{86t}\chi\!\left(\frac{X-100000+40000t-x}{900000}\right),\qquad F=f-g.
\]
On its transition, the finite-head positive-real-part disk and Schwarz's lemma bound $|A|<2000<20000=M$. Since $g_x\le0$ and $g_t=86g+40000|g_x|$,
\[
F_{YY}-2AF_x+2FF_Y-F_t
\ge(839/200)\delta+(39/200)g+(40000-2M)|g_x|>0.
\]

\begin{center}\small
\begin{tabular}{p{.16\linewidth}p{.75\linewidth}}
\toprule
Boundary & Independent audit and numerical evidence\\\midrule
Initial & Euler lower bound on 851 leaves and sharp pair-kernel transfer on four leaves; all 855 leaves have exact adjacency. Smallest density gap exceeds $4.7661468\cdot10^{-7}$.\\
Top $Y=5$ & H7, the finite primary and reflected bounds, and a full Cauchy disk give $S>6.59>f_0$.\\
Left $x=X_L$ & The head theorem and distance to exterior nonreal roots give $S>0$; $F<0$ because $g\ge10$.\\
Moving bottom & Triple-angle identity below and all 1620 rectangular gates yield a strictly positive joint lower limit, including near poles.\\
Infinity & Positive-time $H_t/B_t\to1$ with Cauchy derivatives; at $t\downarrow0$, retain $\zeta(s_B)$ and an ordinary Euler majorant on a band above $Y=1$.\\
\bottomrule
\end{tabular}
\end{center}
With $r_j=\sinh(\lambda_jh)/(\lambda_jh)$, and $V,W,D,E$ defined as the sums of $a_je^{\lambda_j^2t}$ times respectively $\cosh$, $\lambda_j^2r_j$, $\lambda_j^4r_j^3$, and $\lambda_j^2\cosh\,r_j^2$, the exact identity is
\[
\frac{f_0(3h)}3-f_0(h)=\frac{4h^3[D(10^9+V)-3WE]}{3U(3h)U(h)}.
\]
Its independently checked polynomial derivation uses both triple-angle formulas. Each component increases in $t,h$, so opposite row corners bound the numerator. All 1620 bounds pass, the smallest exceeding $3.4851510\cdot10^{19}$. Subtracting $\delta$ and $g$ adds $2(\delta+g)/3$ to the gap. The unsigned kernel transfers the already valid bound at $3h$ to the bottom jointly as $x,t,Y$ vary; no divergent pole limit is assumed.

Strip and density assertions are proved simultaneously. An interior density contact contradicts the strict differential inequality; a wall-height contact contradicts the unsigned force row. Multiple contacts are excluded by backward Hermite splitting. An initial strip cushion and positive-time confinement prevent escape to infinity. The proof does not use the later jet or reference barrier, so the dependency order is acyclic.

\section{Jet comparison and its boundaries}\label{app:jet}
Let $V=-u_z=R+iP$, $c=-\pi/8$, $a=A-c$, $d=S-\Omega(x)$, and
\[
\Ls=\partial_t-\partial_{YY}+2A\partial_x-2S\partial_Y,\qquad
\mathcal M=(a^2+d^2)I+\begin{pmatrix}R&P\\P&-R\end{pmatrix}.
\]
The largest eigenvalue is $J=|u_z|+a^2+d^2$. The heat equation gives $\Ls u=0$, $\Ls V=2V^2$. The product rule, with $\Omega'=1/(4x)$, gives
\[
\Ls(a^2+d^2)=-2(R^2+P^2)-Ad/x,
\qquad
\Ls\mathcal M=-4\begin{pmatrix}P\\-R\end{pmatrix}
 \begin{pmatrix}P&-R\end{pmatrix}-(Ad/x)I.
\]
The $-Ad/x$ correction is retained. The independent exact algebra verifies every matrix entry. At a potential contact with $J=D>0$, fix a unit eigenvector at the contact point and hold it constant nearby. Then $\varphi=v^T\mathcal Mv\le J$, with equality at contact, and
\[
\Ls\varphi\le-Ad/x\le E(D)=\frac{\sqrt D(\pi/8+\sqrt D)}{X_L}.
\]
This argument covers $u_z=0$ and repeated eigenvalues; it never differentiates $|u_z|$ or an eigenvector.

Set $C=e^{10^{-5}t}[8/25+1130000000000/U]$. Since $U^{-1}$ satisfies the drifted heat equation with drift $2f_0$, direct differentiation gives
\[
C_t-C_{YY}-2fC_Y=10^{-5}C+2\delta C_Y,\quad
|C_Y|\le\lambda_{\max}C,
\]
and the strict relative margin over $E(C)$ exceeds $9.7192\cdot10^{-6}$ throughout the wall horizon.

Use the full spatial supersolution $D=C+10^6g$. In the transition, density provides $S\ge f-g$, so $\Ls C$ loses at most $2g|C_Y|$. Meanwhile $\Ls(10^6g)\ge86\cdot10^6g$. The Lipschitz bound for $E$ and $C<103.119<10^6$ give
\[
(86-L_E)10^6>2\lambda_{\max}10^6,\qquad
L_E=\frac{1+\pi/(16\sqrt{8/25})}{X_L}.
\]
Hence $\Ls D>E(D)$. The drift estimate for $A$ is needed only where $g_x\ne0$.

\begin{center}\small
\begin{tabular}{p{.18\linewidth}p{.72\linewidth}}
\toprule
Boundary & Verified implication\\\midrule
Initial & For $b(0)\le Y\le3$, 400 exact rational cells bound $[C_{\rm in}-8/25]_+U$ below $1.13\cdot10^{12}$; maximum is below $1.128379861\cdot10^{12}$. Above $3$, $C_{\rm in}<0.302335<8/25$.\\
Top & Nested H7 disks give $J(x,5,t)<0.309518<8/25$.\\
Left & Positive-real-part Schwarz--Pick bounds give $J<4412281/169<30000$, while $10^6g\ge10^7$.\\
Bottom & All 162 closed time rectangles freshly verified: 45 early, 117 late; their union is the entire wall horizon.\\
Infinity & Uniform positive-time $J\to0$; the time-zero corner has the sharp Euler limit $\zeta''(\sigma)/(4\zeta(\sigma))$ on a band above $Y=1$.\\
\bottomrule
\end{tabular}
\end{center}
At an interior contact, the fixed-eigenvector scalar has nonpositive vertical second derivative and nonnegative time derivative from the left, contradicting $\Ls(\varphi-D)<0$. The band is separated from all roots by $3/5$. At knots use the preceding closed interval. For $x\ge X$, $g=0$, yielding precisely the requested $J\le C$ field.

\section{The lower jet boundary and every cutoff band}\label{app:late}
The boundary is $b(t)=\sqrt{q_M(t)}+3/5$. Its 162 exact rectangles use outward height rounding and adjacent time intervals; geometry verification checks the true moving trace is contained and every required physical and source disk fits H7.

\textbf{Early method.} On boxes 0--44, $\Re s_B>1$. Expand the finite heated sum to first order relative to $\zeta(s_B)$, bounding its finite heat remainder, ordinary tail, reflected branches and H7 remainder with zero, first and second moments. Quotient differentiation includes the factor two on cross terms and the full Bell terms. The second Cauchy derivative uses its factor $2!$. With a strictly positive relative denominator, the resulting estimate bounds $J$ directly, without using the desired jet comparison. All 45 whole rectangles pass; their smallest margin exceeds $17.07745$.

\textbf{Late method.} On boxes 45--161 use
\[
\mathcal E_t(s)=\prod_{p=2,3,5,7,11}(1-e^{t\log^2p/4}p^{-s}),
\qquad \mathcal F(z)=\mathcal E_t(s_B)H_t(z)/B_t(z).
\]
This holomorphic function and its derivative are independent of the cutoff. The primary finite coefficients below $R=100000$ are exact signed divisor sums over the 32 squarefree divisors of 2310. They are enclosed on the \emph{whole} time interval before absolute values are taken. No signed cancellation is inferred from point samples.

Actual cutoffs are covered by eight adjacent finite integer bands starting at 690950, followed by $[1481120,\infty)$. Every omitted divisor term retains its original index up to the full upper cutoff, rather than dividing that upper endpoint by the divisor. All time, height and weighted-summand monotonicity gates are checked before selecting endpoints.

On the unbounded band combine $\log N\ge\max(\log1481120,\log n)$ with its nonpositive coefficient first. This yields a negative Gaussian in $\log n$. Both the finite-prefix integrals and the two infinite Gaussian moments are included. Completing the square gives the exact erfi/erfc antiderivatives used in the replay; differentiating the integrands gives the first-moment formulas, including their boundary terms. There is no infinite positive-time heated series.

For the reflected derivative, the leading imaginary phases cancel, leaving the weight $\log N+1-\frac12\log n$. The negative half-log term is retained before the all-$N$ monotonic reduction. The discrete prefix has at most a single valley; its weighted derivative has a concave quadratic sign polynomial and at most one interval of positive variation. Its supremum bound includes both endpoint contributions when needed.

On every wide rectangle, bound $|\mathcal F'|$ without requiring $\mathcal F\ne0$. Take the maximum of the exact outward dyadic upper endpoints from \emph{all nine bands}. A radius-$1/4$ physical circle then gives $|\mathcal F''|\le4D_{\rm outer}$. Only at the center is the positive denominator $m=1-|\mathcal F-1|$ used. The final estimate is
\[
J\le\epsilon_2+E_2+\frac{4D_{\rm outer}}m
 +\left(\frac{D_{\rm center}}m\right)^2
 +\left(L_E+\frac{D_{\rm center}}m+\epsilon_1\right)^2.
\]
Both square terms are retained. No center denominator is used on the wide circle. The full replay at 512 bits passes all 117 boxes, all 1053 wide evaluations, all 1053 center denominators and all 1053 final comparisons. Its smallest final margin is
\[
>1.4713204735510558893.
\]
Early and late methods share their closed endpoint at $t=0.045$. This discharges the lower-boundary premise needed by the jet comparison; it is not a successful-subset replay.

\section{Reference entry, exact kernels and global assembly}\label{app:assembly}
The independent reference checker consumes exact geometry rather than producer flags. It verifies that the 1620 compact wall cells equal the supporting wall trace, then checks every reference row, wall interpolation, probe separation, signed branch and speed condition. All 7849 density values and all 6317 signed jet ceilings are recomputed using Arb at 512 bits and independently using mpmath intervals at 80 decimal digits. The smallest profile slacks exceed $3.9996\cdot10^{-12}$.

For a whole row, the proven time monotonicity of $f_0$ and $\delta(T_M)<10^{-5}$ give $S>s$. The jet ceiling uses separate monotonic factors, $e^{t_R/10^5}$ and $U(p,t_L)$; it does not assume $C$ decreases with time. Probe separation from the decreasing wall is checked at the left endpoint and persists throughout the row.

\textbf{Kernel proof.} Scale $h$ to one and write $D=d^2/h^2\ge0$, $B=b^2/h^2\in[0,1]$, $P=p^2/h^2$. The pair kernel divided by its height is
\[
C(P)=\frac{2(D+P-B)}{(D+P+B)^2-4PB}.
\]
All denominators are positive on the applicable domain; the distinguished contact singularity is excluded. Clearing denominators reduces the unsigned inequality to $C(1)-C(P)\ge0$ for $P\ge5$, and the signed one to $C(1)-C(P)+(P-1)C'(P)\ge0$ for $P\ge9$. The three-probe inequality has the corresponding fixed $P=9,16,25$ numerator.

The independent symbolic checker substitutes $P=5+r$ or $9+r$ and expands in nonnegative $D,r$, with Bernstein coefficients in $B\in[0,1]$. Every coefficient is nonnegative. This supplies exact polynomial certificates, rather than sampled checks. It also verifies the sharp initial kernel identity and the distinguished source-pair contribution $7/(15h)$.

The highest-zero velocity law then supplies the source gate
\[
q'<-\frac2{15}-4L\sqrt q,
\]
and the unsigned/signed field gates $q'<-2-qK(q)$. The signed branch and $pL_0\le s$ make the relevant scalar bounds monotone in $q$, so the full closed row follows from its checked endpoints. Equality in a rational field-speed condition is harmless because the analytic velocity inequality remains strict. The reference includes such equality cases.

The reference entry has the exact cushion
\[
Q(3/50)=\frac{84648870770133}{200000000000000}
 =q_{\rm ref}(3/50)+10^{-9}.
\]
Its initial strip cushion persists past the positive time $20000010/238529173356019$. Thus positive-time horizontal confinement suffices for the reference first-contact argument. At $t=0$ the strict strip cushion supplies the start.

The assembled certificate has all 7849 reference rows, all 1620 wall cells, the revised 26 source rectangles and all 4803 main rows (623 source, 2690 signed, 1490 unsigned). Its separate exact checker passes every changed-row interface and terminal identity. The source receipts and independent exact/profile checks are included in the ancillary replay archive.

At a hypothetical first contact, confinement makes the zero set compact and local persistence excludes overshoot. The finite head puts the contact beyond $X$. Backward Hermite splitting excludes multiplicity; for a simple zero the strict velocity bound contradicts the nonnegative left derivative of contact excess. At a knot use the ending row. Hence $q(T)<1/400$ for every zero, and the strip theorem adds $1/800$. This proves the assembly asserted in Theorem~\ref{thm:gordon-hybrid}, with its explicit mathematical and directed-enclosure premises.

\end{document}